\documentclass[11pt,preprint,english]{elsarticle}

\newcommand{\beq}{\begin{equation}}
\newcommand{\eeq}{\end{equation}}

\usepackage[latin1]{inputenc}
\usepackage{babel}
\usepackage[centertags]{amsmath}
\usepackage{amsfonts}
\usepackage{amssymb}
\usepackage{color}
\usepackage{amsthm}
\usepackage{newlfont}
\usepackage{graphicx}
\usepackage{dsfont}
\usepackage{geometry}
\usepackage{mathrsfs}
\usepackage{hyperref}

\newtheorem{theorem}{Theorem}[section]
\newtheorem{lemma}[theorem]{Lemma}
\newtheorem{corollary}[theorem]{Corollary}
\newtheorem{proposition}[theorem]{Proposition}
\newtheorem{problem}[theorem]{Skorokhod Reflection Problem}
\newtheorem{definition}[theorem]{Definition}
\newtheorem{remark}[theorem]{Remark}
\newtheorem{Assumptions}[theorem]{Assumption}

\newcommand{\msc}[1]{\textbf{MSC2010 Classification:} #1.}
\newcommand{\jel}[1]{\textbf{JEL Classification:} #1.}

\newcommand{\ackn}[1]{\textbf{Acknowledgments:} #1.}

\makeatletter  
\@addtoreset{equation}{section}

\makeatother  
\begin{document}
\begin{frontmatter}
\title{\textbf{A Stochastic Partially Reversible Investment Problem \\on a Finite Time-Horizon: Free-Boundary Analysis}}

\author[man]{Tiziano De Angelis}\ead{tiziano.deangelis@manchester.ac.uk}

\author[bil]{Giorgio Ferrari}\ead{giorgio.ferrari@uni-bielefeld.de}
\address[man]{School of Mathematics, University of Manchester, Oxford Road, M13 9PL Manchester, UK}
\address[bil]{Center for Mathematical Economics, Bielefeld University, Universit\"atsstrasse 25, D-33615 Bielefeld, Germany}

\begin{abstract}
We study a continuous-time, finite horizon, stochastic partially reversible investment problem for a firm producing a single good in a market with frictions. The production capacity is modeled as a one-dimensional, time-homogeneous, linear diffusion controlled by a bounded variation process which represents the cumulative investment-disinvestment strategy. We associate to the investment-disinvestment problem a zero-sum optimal stopping game and characterize its value function through a free-boundary problem with two moving boundaries. These are continuous, bounded and monotone curves that solve a system of non-linear integral equations of Volterra type. The optimal investment-disinvestment strategy is then shown to be a diffusion reflected at the two boundaries.
\end{abstract}

\begin{keyword}
partially reversible investment, singular stochastic control, zero-sum optimal stopping games, free-boundary problems, Skorokhod reflection problem
\vspace{+8pt}

\msc{93E20, 60G40, 35R35, 91A15, 91B70}
\vspace{+4pt}

\jel{C02, C73, E22, D92}
\end{keyword}

\end{frontmatter}

\section{Introduction}

A firm represents the productive sector of a stochastic economy over a finite time-horizon and it adjusts its production capacity $C$ making repeatedly investments and disinvestments of arbitrary size (we do not require the investment-disinvestment rates to be defined) at given proportional costs. Since we consider a market with frictions the firm buys and sells capacity at different fixed prices and it aims at maximizing its total net expected profit. In mathematical terms, following for instance \cite{PhamGuo}, this amounts to solving the bounded variation control problem with finite horizon
\begin{align}
\label{netprofitintro}
\hspace{-6pt}\sup_{(\nu_{+},\nu_{-})}\hspace{-2pt}\mathbb{E}\bigg\{\hspace{-2pt}\int_0^{T}\hspace{-6pt} e^{- \mu_F t}R(C^{y,\nu}(t))dt \hspace{-1pt}-\hspace{-1pt} c_{+}\hspace{-3pt}\int_{0}^T \hspace{-6pt}e^{-\mu_F t} d\nu_{+}(t) \hspace{-1pt}+ \hspace{-1pt}c_{-}\hspace{-3pt}\int_{0}^T \hspace{-6pt}e^{-\mu_F t} d\nu_{-}(t) \hspace{-1pt}+\hspace{-1pt} e^{-\mu_F T}G(C^{y,\nu}(T)) \bigg\}
\end{align}
where the optimization is taken over all the nondecreasing processes $\nu_{+}$ and $\nu_{-}$ representing the (cumulative) investment and disinvestment strategy, respectively.
Here $\mu_F$ is the firm's manager discount factor, $c_{+}$ is the instantaneous cost of investment, $c_{-}$ is the benefit from disinvestment, $R$ the operating profit function and $G$ a terminal gain, often referred to as a \textsl{scrap function}. We assume that the production capacity $C^{y,\nu}$ follows a stochastic, time-homogeneous, linearly controlled dynamics with $\nu:=\nu_{+} - \nu_{-}$ (cf.\ (\ref{capacity}) below).

The main goals of this papers are two. Firstly, we prove an abstract existence and uniqueness result for the optimal solution pair $(\nu^{*}_{+}, \nu^{*}_{-})$ to problem (\ref{netprofitintro}). Secondly, we provide a semi-explicit representation of such pair in terms of two continuous, bounded and monotone free-boundaries arising in a Zero-Sum Optimal Stopping Game (ZSOSG) associated to the control problem. These boundaries are characterized through a system of non-linear integral equations of Volterra type. To the best of our knowledge this is a complete novelty in the literature on bounded variation control problems with finite horizon. Moreover, we would like to remark that, differently to standard optimal stopping problems, a probabilistic analysis of time-dependent free-boundaries in ZSOSG has not received significant attention so far. A somehow related paper is the very recent work by Yam, Yung and Zhou \cite{YYZ12} dealing with a delta-penalty game call option on a stock with a dividend payment. In that paper the optimal stopping region of a ZSOSG is analyzed in both infinite and finite time-horizon cases; Authors find two optimal boundaries that uniquely solve a couple of non-linear integral equations. However, the aims of their analysis and the setting of their problem are substantially different to ours since, e.g., they do not deal with any control problem and (from a more technical point of view) they have no (unbounded) running profit. The analysis we perform on the ZSOSG builds upon the existing probabilistic theory of optimal stopping and goes beyond that extending a number of results and developing new methodologies.

Theory of investment under uncertainty has received increasing attention in the last years in Economics as well as in Mathematics (we refer for instance to Dixit and Pindyck \cite{DixitPindyck} for a review).
In \cite{Bertola} and \cite{Pindyck} a firm maximizes profits over an infinite time-horizon when the operating profit function is Cobb-Douglas and depends on an exogenous stochastic shock modeled by a geometric Brownian motion. In \cite{AbelEberly} and \cite{BentolilaBertola} the Authors consider the optimal investment problem under uncertainty of a firm that produces a single good with costly reversibility. The problem is formulated over an infinite time-horizon with constant returns to scale Cobb-Dougals production facing an isoelastic demand curve. In \cite{AbelEberly} the optimal investment-disinvestment policy is characterized in terms of a generalized concept of user cost of capital 
introduced by Jorgenson \cite{Jorgenson}. We recall that irreversible investment decisions and their timing are also related to real options as pointed out by \cite{McDonaldSiegel} and \cite{Pindyck} among others.

In a more mathematical environment several Authors studied the firm's optimal problem of capacity irreversible expansion via a number of different approaches. These include dynamic programming techniques (see \cite{Kobila}, \cite{AOksendal} and \cite{Pham}, among others), stochastic first-order conditions and the Bank-El Karoui's Representation Theorem \cite{BankElKaroui} (see, e.g., \cite{Bank}, \cite{CFR}, \cite{freeandbase} and \cite{RiedelSu}), connections with optimal switching problems (cf.\ \cite{GuoTomecek}, among others).
Models involving both expansion and reduction of a project's capacity level (i.e., partially reversible investment problems) have been recently considered by \cite{PhamGuo} and \cite{MehriZervos}, among others. In \cite{MehriZervos}, for example, an infinite time-horizon problem of determining the optimal investment-disinvestment strategy that a firm should adopt in the presence of random price and/or demand fluctuations is considered.
On the other hand, in \cite{PhamGuo} the Authors address a one-dimensional, infinite time-horizon partially reversible investment model with entry decisions and a general running payoff function. They study the problem via a dynamic programming approach and characterize the optimal policy as a diffusion reflected along two (constant in time) boundaries. Here we consider the model of \cite{PhamGuo} without entry decision but with a finite time-horizon.

Our optimization problem (\ref{netprofitintro}) falls within the class of \textsl{bounded variation follower problems with finite horizon}. These are singular stochastic control problems as controls are singular (as functions of time) with respect to the Lebesgue measure.
The link existing between singular stochastic control, optimal stopping and free-boundary problems has been thoroughly studied in a number of papers. The case of monotone controls (monotone follower) was considered for instance in \cite{KaratzasBaldursson}, \cite{BoetiusKohlmann}, \cite{Chiarolla2} and \cite{KaratzasShreve84}.
Recently, bounded variation control problems were brought into contact with optimal stopping games in a similar way (cf., for instance, \cite{Boetius} and \cite{KaratzasWang}). In this setting one has $V_y=v$, where $V_y$ is the derivative of the value function $V$ of the control problem along the direction of the control variable and $v$ is the saddle point of a Dynkin game, i.e.\ of a ZSOSG.

Stochastic games with stopping times have been studied through probabilistic and analytical methods. The former include martingale methods (see for instance \cite{Alario}, \cite{Dynkin}, \cite{Morimoto} and \cite{Pes09}), Markovian structures (see \cite{PeskirEkstrom} and \cite{Stettner82}, among others) and connections with stochastic backward differential equations (see for example \cite{CvitanicKaratzas} and \cite{HamadeneLepeltier95}). In Markovian settings, methods from partial differential equations (PDE), variational inequalities and free-boundary problems were largely employed (see the monographies by Bensoussan and Lions \cite{BensoussanLions} and Friedman \cite{Friedman} and references therein for an overview).

In this paper we use the link between bounded variation follower problems and zero-sum optimal stopping games to study problem (\ref{netprofitintro}). That is, we study the ZSOSG (Dynkin game) with value function $v$ (see \eqref{defv} and \eqref{Jey0} below) which is naturally associated to (\ref{netprofitintro}). Our analysis is carried out in several steps by means of arguments borrowed from probability and PDE theory. 

We show that $v$ is a bounded, continuous function on $[0,T] \times (0,\infty)$ and that the state space $(t,y) \in [0,T] \times (0,\infty)$ splits into three regions defined via two continuous, bounded and monotone free-boundaries $\hat{y}_{+}$ and $\hat{y}_{-}$. The triple $(v, \hat{y}_{+}, \hat{y}_{-})$ solves a free-boundary problem on $[0,T] \times (0,\infty)$ and $v$ fulfills the so-called \textsl{smooth-fit property} along the free-boundaries (cf., e.g., \cite{PeskShir}). We use local time-space calculus (cf.\ \cite{Peskir}) to obtain a representation formula for $v$ in terms of the couple $(\hat{y}_{+}, \hat{y}_{-})$ and we show that $(\hat{y}_{+}, \hat{y}_{-})$ uniquely solves a system of non-linear integral equations of Volterra type (see Theorem \ref{Volterra} and Theorem \ref{esistenzaVolterra} below). A numerical solution to such system of equations is evaluated and illustrated in Figure \ref{figura}.

The optimal control $\nu^{*}:=\nu^{*}_{+}-\nu^{*}_{-}$ for problem (\ref{netprofitintro}) turns out to be the minimal effort needed to keep the optimally controlled diffusion inside the closure of the region between the two free-boundaries. Indeed, an application of results in \cite{Burdzy} allows us to prove that the optimally controlled capacity $C^{y,\nu^{*}}$ uniquely solves a Skorokhod reflection problem in the time-dependent interval $[\hat{y}_{+}(t), \hat{y}_{-}(t)]$, $t < T$. Finally, we obtain a semi--explicit expression of the optimal control $\nu^{*}$.


The paper is organized as follows. In Section \ref{reversibleinvestment} we introduce the partially reversible investment problem and we prove existence and uniqueness of the optimal control. In Section \ref{zerosumoptimalstoppinggame} we study the associated zero-sum optimal stopping game by means of a probabilistic approach to free-boundary problems. In particular in this Section we obtain the system of integral equations for $(v, \hat{y}_{+}, \hat{y}_{-})$ mentioned above. Finally, in Section \ref{optimalcontrolstrategy} we find the optimal control strategy and Section \ref{appproofs} contains some technical proofs.

\section{The Partially Reversible Investment Problem}
\label{reversibleinvestment}

A firm represents the productive sector of a stochastic economy on a complete probability space $(\Omega, \mathcal{F},\mathbb{P})$. We consider an exogenous Brownian motion $W:=\{W(t), t \geq 0\}$ and denote by $\mathbb{F}:=\{\mathcal{F}_t, t \geq 0\}$ its natural filtration augmented by $\mathbb{P}$-null sets. Our setting is similar to the one in \cite{PhamGuo} but with finite time-horizon and no entry decision.
The firm produces at rate $R(C)$ when its own capacity is $C$. We assume that the firm can either invest or disinvest in the market and we denote by $\nu_{+}(t)$ ($\nu_{-}(t)$) the cumulative investment (disinvestment) up to time $t$. Both $\nu_{+}$ and $\nu_{-}$ are left-continuous, a.s.\ finite for all $t\ge0$, nondecreasing processes and we do not require the investment-disinvestment rates to be defined. Once the firm's manager adopts an investment-disinvestment strategy $\nu:=\nu_{+} - \nu_{-}$, then the production capacity evolves according to
\begin{align}\label{capacity}
dC^{y,\nu}(t)= C^{y,\nu}(t)[ -\mu_C dt + \sigma_C dW(t)] + f_Cd\nu(t), \qquad C^{y,\nu}(0)=y > 0,
\end{align}
where $\mu_C$, $\sigma_C$ and $f_C$ are given positive constants.
The parameter $f_C$ is a conversion factor: any unit of investment is converted into $f_C$ units of production capacity.
Setting $C^{0}(t):=C^{1,0}(t)$ we may write
\beq
\label{capacitysolution}
C^{y,\nu}(t)= C^0(t)[y + \overline{\nu}(t)],
\eeq
with
\begin{align}\label{nubarradefinizione}
\overline{\nu}(t) := \int_{0}^t \frac{f_C}{C^{0}(s)}d\nu(s),\quad C^{0}(t)=e^{-\mu_C t}\mathcal{M}_0(t)\:\:\:\text{and}\:\:\:\mathcal{M}_0(t):=e^{-\frac{1}{2}\sigma_C^2 t + \sigma_C W(t)}.
\end{align}
Here $C^{0}$ represents the decay of a unit of initial capital without investment.
In what follows we will also denote
\begin{align}\label{mus}
\hat{\mu}_C:=-\mu_C+\frac{1}{2}\sigma^2_C.
\end{align}

The production function of the firm is a nonnegative, measurable function $R: \mathbb{R}_{+} \mapsto \mathbb{R}_{+}$ of the production capacity and it satisfies the following assumption.
\begin{Assumptions}
\label{AssProfit}
The mapping $C \mapsto R(C)$ is nondecreasing with $R(0)=0$ and strictly concave. It is twice continuously differentiable on $(0,\infty)$ and it has first derivative $R_{c}(C):=\frac{\partial}{\partial C}R(C)$ satisfying the Inada conditions $$\lim_{C \rightarrow 0}R_{c}(C)= \infty,\,\,\,\,\,\,\,\,\,\,\,\,\,\,\,\lim_{C \rightarrow \infty}R_{c}(C)= 0.$$
\end{Assumptions}

Denote by
\begin{eqnarray*}
\mathcal{S}\hspace{-0.3cm}&:=&\hspace{-0.3cm}\{\nu :\Omega \times \mathbb{R}_{+}\mapsto \mathbb{R}_+\,\,\mbox{of\,\,bounded variation,\,\,left-continuous,\,\,adapted\,\,s.t.}\,\, \nu(0)=0,\,\,\mathbb{P}\mbox{-a.s.}\}
\end{eqnarray*}
the nonempty, convex set of investment-disinvestment processes. From now on let $\nu_{+}-\nu_{-}$ be the minimal decomposition of any $\nu \in \mathcal{S}$ into the difference of two left-continuous, nondecreasing, adapted processes such that $\nu_{\pm}(0)=0$ a.s.\ and such that the associated Borel measures on $[0,T]$, $d\nu_+$ and $d\nu_-$, are mutually singular (cf.\ \cite{Royden}, Chapter 11, Section 4).

We assume that the optimization runs over a finite time-horizon $[0,T]$ and we impose that the admissible investment-disinvestment strategies are such that the production capacity of (\ref{capacity}) remains positive.

\begin{definition}\label{admis-set}
For any $(t,y)\in[0,T]\times(0,\infty)$ we denote by $\mathcal{S}^y_{t,T}$ the class of $\nu\in\mathcal{S}$ restricted to $[0,T-t]$ and such that $y+\overline{\nu}(s)\ge 0$ $\mathbb{P}$-a.s.\ for any $s\in[0,T-t]$ (cf.\ \eqref{capacitysolution}).
\end{definition}

Starting at time zero and following an investment-disinvestment strategy $\nu \in \mathcal{S}^y_{0,T}$, the firm receives at terminal time $T$ a (discounted) payoff given by $e^{-\mu_F T}G(C^{y,\nu}(T))$. $G$ is the so-called scrap value of the control problem. We also assume that the function $G: \mathbb{R}_{+} \mapsto \mathbb{R}_{+}$ is concave, nondecreasing, continuously differentiable and such that
\beq
\label{derivatascrap}
\frac{c_{-}}{f_C} \leq G_c(C) \leq \frac{c_{+}}{f_C}-\eta_o
\eeq
for some $0<\eta_o<\frac{c_+-c_-}{f_C}$. Here $c_{+} > c_{-} > 0$ are the cost of investment and the benefit from disinvestment, respectively, in a market with frictions ($c_+=c_-$ if the market were frictionless).
Then, the firm's total expected profit, net of the costs, is given by
\begin{align}
\label{netprofit}
\hspace{-6pt}\mathcal{J}_{0,y}(\nu)\hspace{-2pt}=\hspace{-2pt}\mathbb{E}\bigg\{\hspace{-3pt}\int_0^{T}\hspace{-7pt} e^{- \mu_F t}R(C^{y,\nu}(t))dt \hspace{-1pt}-\hspace{-1pt} c_{+}\hspace{-3pt}\int_{0}^T \hspace{-7pt}e^{-\mu_F t} d\nu_{+}(t)\hspace{-1pt} + \hspace{-1pt}c_{-}\hspace{-3pt}\int_{0}^T\hspace{-7pt} e^{-\mu_F t} d\nu_{-}(t) \hspace{-1pt}+\hspace{-1pt} e^{-\mu_F T}G(C^{y,\nu}(T)) \bigg\}
\end{align}
where $\mu_F > 0$ is the firm's manager discount factor. The value $V$ of the optimal investment-disinvestment problem is
\beq
\label{optimalproblem0}
V(0,y):=\sup_{\nu \in \mathcal{S}^y_{0,T}}\mathcal{J}_{0,y}(\nu).
\eeq
Notice that the strict concavity of $R$, concavity of $G$ and the affine nature of $C^{y,\nu}$ in $\nu$ imply that $\mathcal{J}_{0,y}(\nu)$ is strictly concave on $\mathcal{S}^y_{0,T}$. Hence, if a solution $\nu^{*}$ of (\ref{optimalproblem0}) exists, it is unique.

\begin{proposition}
\label{boundKvaluefunction}
Let Assumption \ref{AssProfit} and condition (\ref{derivatascrap}) hold. Then, there exists $K:=K(T,y) > 0$, depending on $T$ and $y$, such that $0 \leq V(0,y) \leq K$.
\end{proposition}
\begin{proof}
Nonnegativity of $V(0,y)$ follows by taking $\nu_{+}(t) = \nu_{-}(t) \equiv 0$, for $t \geq 0$.
To show that $V$ is bounded from above, recall that $C^0(t)=e^{-\mu_C t}\mathcal{M}_0(t)$ (cf.\ (\ref{nubarradefinizione})) and that for any $\epsilon > 0$ there exists $\kappa_{\epsilon}$ such that $R(C) \leq \kappa_{\epsilon} + \epsilon C$, by Inada conditions (cf.\ Assumption \ref{AssProfit}). Also there exists $\kappa_G \geq 0$ s.t.\ $G(C) \leq \kappa_G + \big(\frac{c_+}{f_C}-\eta_o\big)C\le \kappa_G + \frac{c_+}{f_C}C$ by (\ref{derivatascrap}). Hence, setting $\bar{\mu}:=\mu_F + \mu_C$, for $\nu \in \mathcal{S}^y_{0,T}$ and $\overline{\nu}$ as in (\ref{nubarradefinizione}), we may write
\begin{eqnarray}
\label{Vfinita}
\mathcal{J}_{0,y}(\nu) \hspace{-0.25cm}& \leq &\hspace{-0.25cm} \mathbb{E}\bigg\{\int_0^T e^{-\mu_F t} [\kappa_{\epsilon} + \epsilon C^{y,\nu}(t)] dt - \frac{c_{+}}{f_C}\int_{0}^Te^{-\bar{\mu} t}\mathbb{E}\{\mathcal{M}_0(T)|\mathcal{F}_t\}d\overline{\nu}_{+}(t) \nonumber \\
& & \hspace{0.4cm} + \frac{c_{-}}{f_C}\int_{0}^T e^{- \bar{\mu} t}\mathbb{E}\{\mathcal{M}_0(T)|\mathcal{F}_t\} d\overline{\nu}_{-}(t) + \kappa_G + \frac{c_+}{f_C}e^{-\bar{\mu}T}\mathcal{M}_0(T)[y + \overline{\nu}_{+}(T) - \overline{\nu}_{-}(T)] \bigg\} \nonumber \\
\hspace{-0.25cm} & \leq & \hspace{-0.25cm} \kappa_{\epsilon}T + \frac{c_+}{f_C}y + \kappa_G + \epsilon y \mathbb{E}\bigg\{\int_0^T e^{-\bar{\mu} t} \mathcal{M}_0(t) dt \bigg\} + \epsilon \mathbb{E}\bigg\{\int_0^T e^{-\bar{\mu} t} \mathcal{M}_0(t) \overline{\nu}_{+}(t) dt\bigg\} \nonumber \\
& & \hspace{0.4cm} - \epsilon \mathbb{E}\bigg\{\int_0^T e^{-\bar{\mu} t} \mathcal{M}_0(t) \overline{\nu}_{-}(t) dt\bigg\} - \frac{c_{+}}{f_C}\mathbb{E}\bigg\{\int_{0}^T e^{-\bar{\mu} t}\mathbb{E}\{\mathcal{M}_0(T)|\mathcal{F}_t\}d\overline{\nu}_{+}(t)\bigg\}  \\
& & \hspace{0.4cm} +\, \frac{c_{-}}{f_C}\mathbb{E}\bigg\{\int_{0}^T e^{-\bar{\mu} t}\mathbb{E}\{\mathcal{M}_0(T)|\mathcal{F}_t\}d\overline{\nu}_{-}(t)\bigg\} +\,\frac{c_+}{f_C}\mathbb{E}\Big\{e^{-\bar{\mu}T}\mathcal{M}_0(T)[\overline{\nu}_{+}(T) - \overline{\nu}_{-}(T)]\Big\}. \nonumber
\end{eqnarray}
Notice now that $\mathbb{E}\{\int_{[0,T)}e^{-\bar{\mu}t}\mathbb{E}\{\mathcal{M}_0(T)|\mathcal{F}_t\} d\overline{\nu}_{\pm}(t)\} = \mathbb{E}\{\mathcal{M}_0(T) \int_{[0,T)}e^{-\bar{\mu}t} d\overline{\nu}_{\pm}(t)\}$, by \cite{RevuzYor}, Chapter V, Exercise 1.13, p.186, and introduce the new probability measure $\tilde{\mathbb{P}}$ defined by
\beq
\label{Girsanov}
\frac{d\tilde{\mathbb{P}}}{d\mathbb{P}}\Big |_{\mathcal{F}_t}:= \mathcal{M}_0(t)= e^{-\frac{1}{2}\sigma^2_C t + \sigma_C W(t)}, \,\,\,t \geq 0.
\eeq
Then, integrating by parts the integrals with respect to $d\nu_{\pm}$, we obtain from (\ref{Vfinita}) that
\begin{eqnarray}
\label{Vfinita2}
\mathcal{J}_{0,y}(\nu)
\hspace{-0.25cm} & \leq & \hspace{-0.25cm} (\kappa_{\epsilon}T + \frac{c_+}{f_C}y + \kappa_G + \epsilon yT) + \left(\epsilon - \frac{c_{+}\bar{\mu}}{f_C}\right)\tilde{\mathbb{E}}\bigg\{\int_{0}^T e^{-\bar{\mu} t}\overline{\nu}_{+}(t) dt\bigg\} \nonumber\\
& & \hspace{0.6cm} + \left(\frac{c_{-}\bar{\mu}}{f_C} - \epsilon\right)\tilde{\mathbb{E}}\bigg\{\int_{0}^Te^{-\bar{\mu} t}\overline{\nu}_{-}(t) dt\bigg\} + \left(\frac{c_{-}}{f_C} - \frac{c_{+}}{f_C}\right) \tilde{\mathbb{E}}\Big\{e^{-\bar{\mu}T}\overline{\nu}_{-}(T)\Big\} \nonumber \\
\hspace{-0.25cm} & \leq & \hspace{-0.25cm} K + \left(\epsilon - \frac{c_{+}\bar{\mu}}{f_C}\right)\tilde{\mathbb{E}}\bigg\{\int_{0}^T e^{-\bar{\mu} t}\overline{\nu}_{+}(t) dt\bigg\} + \left(\frac{c_{-}\bar{\mu}}{f_C} - \epsilon\right)\tilde{\mathbb{E}}\bigg\{\int_{0}^Te^{-\bar{\mu} t}\overline{\nu}_{-}(t) dt\bigg\}, \nonumber
\end{eqnarray}
with $\tilde{\mathbb{E}}\{\cdot\}$ denoting the expectation under $\tilde{\mathbb{P}}$ and $K$ a positive constant independent of $\nu_{\pm}$ but depending on $y,T,\epsilon,c_+,c_{-},f_C,\kappa_G$.
Taking $\epsilon = \frac{c_{+} \bar{\mu}}{f_C}$, it follows
$\mathcal{J}_{0,y}(\nu) \leq K$ for all $\nu \in \mathcal{S}^y_{0,T}$,
since $c_+ > c_{-}$ and $\overline{\nu}_{-}(t) \geq 0$ a.s.\ for every $t\geq 0$.
\end{proof}

Recall that $(\nu^{n})_{n \in \mathbb{N}} \subset \mathcal{S}^y_{0,T}$ is a maximizing sequence if $\lim_{n\rightarrow \infty}\mathcal{J}_{0,y}(\nu^n) = V(0,y)$. Then, we have the following
\begin{corollary}
\label{corollboundnubarra}
For any maximizing sequence $(\nu^{n})_{n \in \mathbb{N}} \subset \mathcal{S}^y_{0,T}$ there exists $C:=C(T,y)>0$, depending on $T$ and $y$, such that $\tilde{\mathbb{E}}\{\overline{\nu}^n_+(T)\}+\tilde{\mathbb{E}}\{\overline{\nu}^n_-(T)\} \leq C$ for all $n\in\mathbb{N}$.
\end{corollary}
\begin{proof}
Note that the mapping $\nu \mapsto \overline{\nu}$ is one to one and onto. Therefore $(\nu^{n})_{n \in \mathbb{N}} \subset \mathcal{S}^y_{0,T}$ is a maximizing sequence if and only if the associated sequence $(\overline{\nu}^n)_{n \in \mathbb{N}} \subset \mathcal{S}^y_{0,T}$ is maximizing as well. There is no restriction assuming $\mathcal{J}_{0,y}(\overline{\nu}^n)\ge V(0,y)-\frac{1}{n}$. Under the measure $\tilde{\mathbb{P}}$ (cf.\ \eqref{Girsanov}) we may write the net profit functional $\mathcal{J}_{0,y}$ in (\ref{netprofit}) for any $\overline{\nu}^n \in \mathcal{S}^y_{0,T}$ as
\begin{eqnarray}
\label{netprofitptilde2}
\mathcal{J}_{0,y}(\overline{\nu}^n)\hspace{-0.25cm}&=&\hspace{-0.25cm}\tilde{\mathbb{E}}\bigg\{\int_0^{T} e^{- \mu_F t}\,\frac{1}{\mathcal{M}_0(t)}R(C^{y,\overline{\nu}^n}(t))dt - \frac{c_{+}}{f_C} \int_{0}^T e^{-\bar{\mu} t} d\overline{\nu}^n_{+}(t) + \frac{c_{-}}{f_C} \int_{0}^T e^{-\bar{\mu} t} d\overline{\nu}^n_{-}(t) \nonumber \\
& &\hspace{1.5cm} + e^{-\mu_F T}\frac{1}{\mathcal{M}_0(T)}G(C^{y,\overline{\nu}^n}(T)) \bigg\}.
\end{eqnarray}
From Assumption \ref{AssProfit}, for any $\epsilon>0$ there exists $\kappa_\epsilon>0$ such that $R(C)\le \kappa_\epsilon+\epsilon C$, then recalling \eqref{capacitysolution} we find
\begin{align}
\label{caciottari01}
\tilde{\mathbb{E}}&\bigg\{\int_0^{T} e^{- \mu_F t}\,\frac{1}{\mathcal{M}_0(t)}R(C^{y,\overline{\nu}^n}(t))dt\bigg\}\le\tilde{\mathbb{E}}\bigg\{\int_0^{T} e^{- \mu_F t}\,\frac{\kappa_\epsilon}{\mathcal{M}_0(t)}dt\bigg\}+\epsilon\,y+\epsilon\tilde{\mathbb{E}}\bigg\{\int_0^{T}{\hspace{-5pt}e^{-\bar\mu t}\,\overline{\nu}^n(t)dt}\bigg\}.
\end{align}
Analogously, from \eqref{derivatascrap} it follows that there exists $\kappa_G>0$ such that $G(C)\le \kappa_G+\big(\frac{c_+}{f_C}-\eta_o\big)C$ and therefore
\begin{align}\label{caciottari02}
\tilde{\mathbb{E}}&\bigg\{e^{-\mu_F T}\frac{1}{\mathcal{M}_0(T)}G(C^{y,\overline{\nu}^n}(T))\bigg\}\le\tilde{\mathbb{E}}\bigg\{e^{-\mu_F T}\frac{\kappa_G}{\mathcal{M}_0(T)}\bigg\}+\frac{c_+}{f_C}y+\big(\frac{c_+}{f_C}-\eta_o\big)e^{-\bar{\mu}T}\tilde{\mathbb{E}}\big\{\overline{\nu}^n(T)\big\}.
\end{align}
Now, from \eqref{netprofitptilde2}, \eqref{caciottari01}, \eqref{caciottari02} and integrating by parts integrals with respect to the measures $d\overline{\nu}_\pm$, we obtain
\begin{eqnarray*}
\label{caciottari03}
& & V(0,y)-\frac{1}{n} \leq \, \mathcal{J}_{0,y}(\overline{\nu}^n)\leq\, c(\epsilon,y,T)+\big(\frac{\bar\mu c_-}{f_C}-\epsilon\big)\tilde{\mathbb{E}}\bigg\{\int^T_0{e^{-\bar\mu t}\overline{\nu}^n_-(t)dt}\bigg\} \nonumber \\
& & -\big(\frac{c_+-c_-}{f_C}-\eta_o\big)e^{-\bar\mu T}\tilde{\mathbb{E}}\big\{\overline{\nu}^n_-(T)\big\} -\big(\frac{\bar\mu c_+}{f_C}-\epsilon\big)\tilde{\mathbb{E}}\bigg\{\int^T_0{e^{-\bar\mu t}\overline{\nu}^n_+(t)dt}\bigg\}-\eta_oe^{-\bar\mu T}\tilde{\mathbb{E}}\big\{\overline{\nu}^n_+(T)\big\} \nonumber
\end{eqnarray*}
where $c(\epsilon,y,T)>0$ is a suitable constant depending on $\epsilon$, $y$ and $T$. Setting $\epsilon=\bar\mu c_-/ f_C$ and recalling that $\eta_o<\frac{c_+-c_-}{f_C}$ (cf.\ \eqref{derivatascrap}) we find the bound
\begin{align}\label{caciottari04}
\big(\frac{c_+-c_-}{f_C}-\eta_o\big)e^{-\bar\mu T}\tilde{\mathbb{E}}\big\{\overline{\nu}^n_-(T)\big\}+\eta_oe^{-\bar\mu T}\tilde{\mathbb{E}}\big\{\overline{\nu}^n_+(T)\big\}\le c(\epsilon,y,T)+1.
\end{align}
Since \eqref{caciottari04} holds for all $n\in\mathbb{N}$ this completes the proof.
\end{proof}

Next we show the existence of a unique optimal solution pair $(\nu^{*}_{+},\nu^{*}_{-})$ to problem (\ref{optimalproblem0}).
\begin{theorem}
\label{existence}
Under Assumption \ref{AssProfit} and condition (\ref{derivatascrap}) there exists a unique admissible investment-disinvestment strategy $\nu^{*}$ which is optimal for problem (\ref{optimalproblem0}).
\end{theorem}
\begin{proof}
From Corollary \ref{corollboundnubarra} we have that admissible maximizing sequences $(\tilde{\mathbb{E}}\{\overline{\nu}^n_{\pm}(T)\})_{n \in \mathbb{N}}$ are uniformly bounded and hence by a version of Koml\`{o}s' Theorem for predictable increasing processes (cf.\ \cite{Kabanov}, Lemma $3.5$) there exist two subsequences $(\overline{\nu}^{n_k}_{\pm})_{k \in \mathbb{N}}$ that converge a.s.\ for each time $t \in [0,T]$ in the Ces\`{a}ro sense to some integrable, increasing and predictable $\overline{\nu}^{*}_{\pm}$; i.e., if we define
\beq
\label{CesaroSPmeno}
\theta^{j}_{\pm}(t) := \frac{1}{j+1}\sum_{k=0}^{j}\overline{\nu}^{n_k}_{\pm}(t),
\eeq
then $\lim_{j \rightarrow \infty}\theta^{j}_{\pm}(t) = \overline{\nu}^{*}_{\pm}(t)$ a.s.\ for every $t \in [0,T]$.
Clearly, this implies a.s.\ weak convergence of $\theta^{j}_{\pm}$ to $\overline{\nu}^{*}_{\pm}$; that is,
\begin{equation}
\label{weakconvI}
\lim_{j \rightarrow \infty} \int_0^T f(t)\,d\theta^{j}_{\pm}(t) = \int_0^T f(t)\,d\overline{\nu}^{*}_{\pm}(t),\,\,\,\tilde{\mathbb{P}}-\text{a.s.},
\end{equation}
for every continuous and bounded function $f(\cdot)$ (see, e.g., \cite{RevuzYor}, Chapter $0$, Section $5$). We still denote by $\overline{\nu}^*_\pm$ the left-continuous modifications of $\overline{\nu}^*_\pm$ and clearly $\overline{\nu}^*:=\overline{\nu}^*_+-\overline{\nu}^*_-\in\mathcal{S}^{y}_{0,T}$.

Since $(\overline{\nu}^{n})_{n \in \mathbb{N}}$ is a maximizing sequence, then $(\theta^{j})_{j \in \mathbb{N}}$, $\theta^j:=\theta^j_{+} -\theta^j_{-}$, is maximizing as well by concavity of the profit functional. Now, if we could use (reverse) Fatou's Lemma, we would obtain
\beq
\label{ifFatou}
V(0,y) \leq \limsup_{j \rightarrow \infty}J_{0,y}(\theta^j) \leq \mathcal{J}_{0,y}(\overline{\nu}^{*}),
\eeq
thus the optimality of $\nu^{*}_{\pm}(t):= \int_{0}^t\frac{C^0(s)}{f_C}d\overline{\nu}^{*}_{\pm}(s)$. Uniqueness follows as usual from strict concavity of $\mathcal{J}_{0,y}$ and convexity of $\mathcal{S}^y_{0,T}$.

It remains to show that (reverse) Fatou's Lemma can be applied. Recall \eqref{netprofitptilde2} and set $\bar{\mu}=\mu_C + \mu_F$, then under $\tilde{\mathbb{P}}$ and for any $\overline{\nu} \in \mathcal{S}^y_{0,T}$, integrating by parts the integrals with respect to $d\overline{\nu}_{\pm}$, $\mathcal{J}_{0,y}$ may be written in terms of
\begin{eqnarray*}
\label{netprofitptilde}
\mathcal{J}_{0,y}(\overline{\nu})
\hspace{-0.25cm} & = & \hspace{-0.25cm} \tilde{\mathbb{E}}\bigg\{\int_0^T \Phi^{y, \overline{\nu}}(t)\,dt + \hat{G}^{y,\overline{\nu}}(T)\bigg\},
\end{eqnarray*}
$$\Phi^{y, \overline{\nu}}(t):= e^{-\mu_Ft}\frac{1}{\mathcal{M}_0(t)}R(C^{y,\overline{\nu}}(t)) - e^{-\bar{\mu}t}\Big(\frac{c_{+}\bar{\mu}}{f_C}\overline{\nu}_{+}(t) - \frac{c_{-}\bar{\mu}}{f_C}\overline{\nu}_{-}(t)\Big),$$
$$\hat{G}^{y,\overline{\nu}}(T):= e^{-\mu_F T}\frac{1}{\mathcal{M}_0(T)}G(C^{y,\overline{\nu}}(T)) - e^{-\bar{\mu}T}\Big(\frac{c_{+}}{f_C}\overline{\nu}_{+}(T) - \frac{c_{-}}{f_C}\overline{\nu}_{-}(T)\Big).$$
Recall (\ref{capacitysolution}) and $c_+ > c_{-}$.
Since for every $\epsilon > 0$ there exists $\kappa_{\epsilon}>0$ such that $R(C) \leq \kappa_\epsilon + \epsilon C$ (cf.\ Assumption \ref{AssProfit}), then we obtain
\begin{equation*}
\Phi^{y, \overline{\nu}}(t) \leq \frac{\kappa_{\epsilon}e^{-\mu_F t}}{\mathcal{M}_0(t)} + \epsilon y e^{-\bar{\mu}t} + e^{-\bar{\mu}t}\overline{\nu}_{+}(t)\left(\epsilon - \frac{c_{+}\bar{\mu}}{f_C}\right) + e^{-\bar{\mu}t}\overline{\nu}_{-}(t)\left(\frac{c_{-}\bar{\mu}}{f_C} - \epsilon \right),\qquad\,\,\overline{\nu} \in \mathcal{S}^y_{0,T}.
\end{equation*}
We now take $\epsilon = \frac{\bar{\mu}c_{-}}{f_C}$ and we find
\beq
\label{boundPHI}
\Phi^{y,\overline{\nu}}(t) \leq \hat{K}\left(1 + \frac{1}{\mathcal{M}_0(t)}\right),
\eeq
for some $\hat{K}>0$, and the right-hand side of (\ref{boundPHI}) is $d\tilde{\mathbb{P}} \otimes dt$-integrable and independent of $\overline{\nu}$.
Again, $G(C) \leq \kappa_G + \big(\frac{c_+}{f_C}-\eta_o\big)C \le \kappa_G + \frac{c_+}{f_C}C $, for some $\kappa_G \geq 0$ (cf.\ (\ref{derivatascrap})), and hence
\beq
\label{boundscrapeterminiaT}
\hat{G}^{y,\overline{\nu}}(T) \leq \frac{\kappa_G e^{-\mu_F T}}{\mathcal{M}_0(T)} + \frac{c_{+}y}{f_C}e^{-\bar{\mu}T}.
\eeq
Note that the right-hand side of (\ref{boundscrapeterminiaT}) is independent of $\overline{\nu}$ and $\tilde{\mathbb{P}}$-integrable. Therefore we can apply Fatou's Lemma to justify (\ref{ifFatou}).
\end{proof}

\begin{remark}
\label{stoch-op-fun}
Arguments above for existence and uniqueness of the optimal control hold as well in presence of an exogenous stochastic factor influencing the operating profit function. That is when $R:\Omega\times\mathbb{R}_+\mapsto\mathbb{R}_+$ and $G:\Omega\times\mathbb{R}_+\mapsto\mathbb{R}_+$ satisfying standard measurability conditions with $R(\omega,\,\cdot\,)$ as in Assumption \ref{AssProfit} and $G_c(\omega,\,\cdot\,)$ as in \eqref{derivatascrap} for all $\omega\in\Omega$ (see also \cite{KaratzasWang}).
\end{remark}

\begin{remark}
\label{connectionDixitPyndick}
We can draw an analogy between our setting and Dixit-Pindyck-like models \cite{DixitPindyck} by taking for instance $R(y):= y^\alpha$, $\alpha\in(0,1)$ and $G(y)=c_{-}/f_C \,y$. In fact, recall \eqref{capacitysolution} and \eqref{nubarradefinizione} and notice that we can write $R(C^{y,\nu}(t))=\mathcal{M}_0(t)X(t) \big(Z^{y,\nu}(t)\big)^\alpha$ where $Z^{y,\nu}(t):=e^{-\mu_C t}[y+\overline{\nu}(t)]$ and $X(t):=\big(\mathcal{M}_0(t)\big)^{\alpha-1}$. Similarly $G(C^{y,\nu}(T))=c_{-}/f_C\,\mathcal{M}_0(T)Z^{y,\nu}(T)$.

Then, changing measure according to \eqref{Girsanov} and setting $d\xi^{\nu}(t):=e^{-\mu_Ct}d\overline{\nu}(t)$, we obtain
\begin{align*}
\hspace{-6pt}\mathcal{J}_{0,y}(\nu)=\tilde{\mathbb{E}}\bigg\{\hspace{-2pt}\int_0^{T}\hspace{-6pt} e^{-\mu_F t}X(t)(Z^{y,\nu}(t))^{\alpha}dt \hspace{-1pt}-\hspace{-1pt} \frac{c_{+}}{f_C}\hspace{-3pt}\int_{0}^T \hspace{-6pt}e^{-\mu_F t} d\xi^{\nu}_{+}(t) \hspace{-1pt}+ \hspace{-1pt}\frac{c_{-}}{f_C}\hspace{-3pt}\int_{0}^T \hspace{-6pt}e^{-\mu_F t} d\xi^{\nu}_{-}(t) \hspace{-1pt}+\hspace{-1pt} e^{-\mu_F T}\frac{c_-}{f_C}Z^{y,\nu}(T)\bigg\}.
\end{align*}
Now, as in \cite{DixitPindyck}, $Z^{y,\nu}$ can be seen as a production capacity that follows the linearly controlled dynamics $dZ^{y,\nu}(t)=-\mu_C Z^{y,\nu}(t)dt+d\xi^{\nu}(t)$, starting from $Z^{y,\nu}(0)=y$; instead $X$ can be seen as the stochastic price of the produced good described by a suitable geometric Brownian motion with $X(0)=1$.
\end{remark}

\section{The Zero-Sum Optimal Stopping Game}
\label{zerosumoptimalstoppinggame}

In order to characterize the optimal control policy we shall associate to problem (\ref{optimalproblem0}) a suitable zero-sum optimal stopping game, in the spirit of \cite{Dynkin} and \cite{KaratzasWang}, among others. Then, we will show that the value function solves a free-boundary problem with two free-boundaries which are continuous, bounded and monotone solutions of a system of non-linear integral equations.

As usual in the literature of dynamic programming, we let the optimization in (\ref{optimalproblem0}) start at arbitrary time $t \in [0,T]$.
Since the solution of (\ref{capacity}) and the net profit functional are time-homogeneous, then we may simply set a time-horizon $[0,T-t]$ in (\ref{netprofit}) and write
\begin{eqnarray}
\label{netprofitt}
\mathcal{J}_{t,y}(\nu)\hspace{-0.25cm}&=&\hspace{-0.25cm}\mathbb{E}\bigg\{\int_0^{T-t} e^{- \mu_F s}\,R(C^{y,\nu}(s))ds - c_{+} \int_{0}^{T-t} e^{-\mu_F s} d\nu_{+}(s) + c_{-} \int_{0}^{T-t} e^{-\mu_F s} d\nu_{-}(s) \nonumber \\
& &\hspace{1.5cm} + e^{-\mu_F (T-t)}G(C^{y,\nu}(T-t)) \bigg\}.
\end{eqnarray}

It follows that the firm's investment-disinvestment problem now reads (cf.~Definition \ref{admis-set})
\beq
\label{optimalproblem}
V(t,y):=\sup_{\nu \in \mathcal{S}^y_{t,T}}\mathcal{J}_{t,y}(\nu).
\eeq
From (\ref{nubarradefinizione}) and (\ref{capacitysolution}), we may write the value function $V(t,y)$ of the optimal control problem (\ref{optimalproblem}) in terms of a maximization over the controls $\overline{\nu} \in \mathcal{S}^y_{t,T}$; that is,
\begin{eqnarray}
\label{Jnubarra}
V(t,y) \hspace{-0.25cm} & = & \hspace{-0.25cm} \sup_{\overline{\nu}\in \mathcal{S}^y_{t,T}}\mathbb{E}\bigg\{\int_0^{T-t} e^{- \mu_F s}\,R(C^0(s)[ y + \overline{\nu}(s)])ds - \frac{c_{+}}{f_C} \int_{0}^{T-t} e^{-\mu_F s} C^0(s) d\overline{\nu}_{+}(s) \nonumber \\
& & \hspace{2cm}+ \frac{c_{-}}{f_C} \int_{0}^{T-t} e^{-\mu_F t} C^0(s) d\overline{\nu}_{-}(s) + e^{-\mu_F (T-t)}G\big(C^0(T-t)[y + \overline{\nu}(T-t)]\big)\bigg\}. \nonumber
\end{eqnarray}

In order to employ results by \cite{KaratzasWang}, take $\omega \in \Omega$, $s \in [0,T-t]$, $y \in (0,\infty)$ and set
\beq
\label{KWingredients}
\left\{
\begin{array}{ll}
\xi_{\pm}(\omega,s):=\overline{\nu}_{\pm}(\omega,s), \qquad \qquad\,\,\, X(\omega,s):= y + \overline{\nu}(\omega,s) = y + \xi^{+}(\omega,s) - \xi^{-}(\omega,s), \\ \\
H(\omega,s,y):= -e^{-\mu_F s} R(y C^0(\omega,s)), \qquad G(\omega,y):= -e^{-\mu_F (T-t)}G(yC^0(\omega,T-t)), \\ \\
\gamma(\omega,s):= \frac{c_{+}}{f_C} e^{-\mu_F s} C^0(\omega,s)\mathds{1}_{\{s < T-t\}}, \qquad \nu(\omega,s) := - \frac{c_{-}}{f_C} e^{-\mu_F s} C^0(\omega,s)\mathds{1}_{\{s < T-t\}}.
\end{array}
\right.
\eeq
Notice that $H_y(\omega,s,y)$ is $d\mathbb{P} \otimes dt$-integrable for any $y>0$, thanks to concavity of $R$, whereas $G_y(\omega,y)$ is $d\mathbb{P}$-integrable by (\ref{derivatascrap}). Moreover $\mathbb{E}\{\sup_{0 \leq s \leq T-t}|\gamma(s)| + \sup_{0 \leq s \leq T-t}|\nu(s)|\} < \infty$.
Then, thanks to Theorem \ref{existence} and \cite{KaratzasWang}, Theorem $3.1$, we fit into \cite{KaratzasWang}, Theorem $3.2$ (with time-horizon $T-t$). It is important to remark that in \cite{KaratzasWang} the set of admissible controls does not require that the controlled process remains positive. However, the proof of Theorem $3.2$ therein is based on the construction of suitable perturbations of the optimal control $\xi_*$ (namely $\xi_\varepsilon$ in the proof of Lemma $4.2$ and $\vartheta_\varepsilon$ in the proof of Lemma $4.3$ therein); one may easily verify that such perturbations of the optimal control preserve positivity of the process provided that $\xi_*$ does.
\begin{proposition}
\label{2valuefunctions}
For $V_y:=\frac{\partial\,V}{\partial\,y}$ and under Assumption \ref{AssProfit} and \eqref{derivatascrap}, the value function $V(t,y)$ of the control problem (\ref{optimalproblem}) satisfies
$V_y(t,y) = v(t,y)$ for $(t,y) \in [0,T] \times (0, \infty)$,
with
\begin{align}\label{defv}
v(t,y):=\inf_{\sigma \in [0,T-t]}\sup_{\tau \in [0,T-t]}\Psi_0(t,y;\sigma,\tau)=\sup_{\tau \in [0,T-t]}\inf_{\sigma \in [0,T-t]}\Psi_0(t,y;\sigma,\tau)
\end{align}
and
\begin{align}\label{Jey0}
\hspace{-12pt}\Psi_0(t,y;\sigma,\tau):=&\mathbb{E}\bigg\{ \int_0^{\tau \wedge \sigma} e^{-\mu_F s} C^0(s) R_c(yC^0(s)) ds+ e^{-\mu_F (T-t)}C^0(\tau)G_c(yC^0(\tau))\mathds{1}_{\{\tau = \sigma = T-t\}}\nonumber\\
&\hspace{+15pt}+\frac{c_{+}}{f_C} e^{-\mu_F \sigma} C^0(\sigma) \mathds{1}_{\{\sigma \leq \tau\}}\mathds{1}_{\{\sigma < T-t\}}+ \frac{c_{-}}{f_C} e^{-\mu_F \tau} C^0(\tau) \mathds{1}_{\{\tau < \sigma\}} \bigg\}.
\end{align}
\end{proposition}

\noindent Here $v(t,y)$ is the value function of a zero-sum optimal stopping game (Dynkin game). Consider two players, $\mathcal{P}_1$ and $\mathcal{P}_2$, starting playing at time $t \in [0,T]$. Player $\mathcal{P}_1$ can choose the stopping time $\sigma$, whereas player $\mathcal{P}_2$ the stopping time $\tau$. The game ends as soon as one of the two players decides to stop, i.e.\ at the stopping time $\sigma \wedge \tau$. As long as the game is in progress, $\mathcal{P}_1$ keeps paying $\mathcal{P}_2$ at the (random) rate $e^{-\mu_F t} C^0(t) R_c(yC^0(t))$ per unit of time. When the game ends before $T-t$, $\mathcal{P}_1$ pays $\frac{c_+}{f_C}e^{-\mu_F \sigma} C^0(\sigma)$ if she/he decides to stop earlier than $\mathcal{P}_2$; otherwise $\mathcal{P}_1$ pays $\frac{c_{-}}{f_C}e^{-\mu_F \tau} C^0(\tau)$.
If no one decides to stop the game (i.e.\ the game ends at $T-t$), $\mathcal{P}_1$ pays $\mathcal{P}_2$ the (random) amount $e^{-\mu_F (T-t)}C^0(T-t)G_c(yC^0(T-t))$.
As it is natural, $\mathcal{P}_1$ tries to minimize her expected loss, whereas $\mathcal{P}_2$ tries to maximize it. We remark that Proposition \ref{2valuefunctions} implies existence of Stackelberg equilibrium (Nash equilibrium is also provided in \cite{KaratzasWang}).

\begin{remark}
Notice that in \cite{KaratzasWang}, Theorem 3.2, the instantaneous cost functions $\gamma$ and $\nu$ are both positive. This is not true in our setting, however reading carefully the proof of \cite{KaratzasWang}, Theorem 3.2, one can see that such condition is not necessary.
\end{remark}

Recall now $\tilde{\mathbb{P}}$ defined in (\ref{Girsanov}) and set $\tilde{W}(t): = W(t) - \sigma_C t$, $t \geq 0$. This process is a $\tilde{\mathbb{P}}$-Brownian motion and
$C^0(t)=\exp\{\hat{\mu}_Ct+\sigma_C\tilde{W}(t)\}$,
with $\hat{\mu}_C$ as in \eqref{mus}, under the new measure.
Then Girsanov Theorem allows us to rewrite $v(t,y)$ of (\ref{defv}) under $\tilde{\mathbb{P}}$ as
\begin{eqnarray}
\label{defv2}
v(t,y) \hspace{-0.2cm} & := & \hspace{-0.2cm} \inf_{\sigma \in [0,T-t]}\sup_{\tau \in [0,T-t]}\Psi(t,y;\sigma,\tau) = \sup_{\tau \in [0,T-t]} \inf_{\sigma \in [0,T-t]}\Psi(t,y;\sigma,\tau),
\end{eqnarray}
with
\begin{align}\label{Jey}
\Psi(t,y;\sigma,\tau)=\tilde{\mathbb{E}}\bigg\{ &\frac{c_{+}}{f_C} e^{-\bar{\mu} \sigma} \mathds{1}_{\{\sigma \leq \tau\}}\mathds{1}_{\{\sigma < T-t\}} + \frac{c_{-}}{f_C} e^{-\bar{\mu} \tau} \mathds{1}_{\{\tau < \sigma\}} \nonumber\\
& + e^{-\bar{\mu}(T-t)}G_c(yC^0(T-t))\mathds{1}_{\{\tau = \sigma = T-t\}} + \int_0^{\tau \wedge \sigma} e^{-\bar{\mu} s} R_c(yC^0(s)) ds \bigg\}
\end{align}
and, again, $\bar{\mu}:=\mu_C + \mu_F$. Notice that
$c_{-}/f_C \leq v(t,y) \leq c_{+}/f_C$
for all $(t,y) \in [0,T] \times (0,\infty)$.

From now on, our aim will be to characterize the optimal control $\nu^{*}$ for problem (\ref{optimalproblem}) in terms of the optimal strategy of the zero-sum game (\ref{defv2}). We expect the latter to be given by the first exit times $(\sigma^{*},\tau^{*})$ of the process $\{yC^0(s), s \geq 0\}$ from the region bounded between two moving boundaries denoted by $\hat{y}_{+}$ and $\hat{y}_{-}$, respectively. A satisfactory characterization of the free-boundaries is extremely hard to find in general unless the marginal scrap value $G_c$ coincides with either $\frac{c_+}{f_C}$ or $\frac{c_{-}}{f_C}$. That is a common assumption when addressing zero-sum optimal stopping games with variational methods (cf., e.g., \cite{Friedman}, Volume 2, Chapter $16$, Section $9$). We observe that if $G_c(C)=\frac{c_{+}}{f_C}$, the player who aims at maximizing $\Psi$ will choose a `no-action strategy' for $t > [T - \frac{1}{\bar{\mu}}\ln(\frac{c_+}{c_{-}})]^{+}$ regardless of the initial state $y$. In fact, an immediate stopping would get her/him a reward equal to $\frac{c_{-}}{f_C}$, whereas doing nothing would guarantee a payoff larger than $\frac{c_{+}}{f_C}e^{-\bar{\mu}(T-t)}$. Somehow this introduces an advantage for the `sup-player' as her/his strategy is known on a whole time interval before the end of the game. To avoid such a situation and to be consistent with \eqref{derivatascrap} we make the following

\begin{Assumptions}
\label{scrap}
One has $G(C) = G_0+\frac{c_{-}}{f_C}C$ for some $G_0\ge0$.
\end{Assumptions}

\begin{theorem}
\label{jointcontinuityv}
Under Assumptions \ref{AssProfit} and \ref{scrap} the value function $v(t,y)$ defined in (\ref{defv2}) is continuous on $[0,T] \times (0,\infty)$.
\end{theorem}
The full proof of this Theorem is quite technical and it is contained in Section \ref{Appcontinuity}. It follows by a non-trivial extension to the present setting of arguments developed in \cite{Stettner}, Theorem $1$, for bounded running profits. Since our $R_c$ is unbounded we need to find the correct functional space where methods similar to those in \cite{Stettner} can be suitably adapted. Note that techniques analogous to those of \cite{Stettner} were also employed by Menaldi, among others, in an earlier paper (cf.\ \cite{Menaldi}) in the study of an optimal stopping problem for degenerate diffusions. In the context of game Call options, the proof of continuity of the value function crucially relies on the Lipschitz continuity of the payoff function (cf.~\cite{KunitaSeko}, Lemma 5.1, or \cite{YYZ12}, Lemma 3.1) which clearly breaks down in our setting.

\begin{theorem}
\label{thmsaddlepointsv}
Under Assumptions \ref{AssProfit} and \ref{scrap} the stopping times
\beq
\label{stoppingtimesv}
\left\{
\begin{array}{ll}
\sigma^*(t,y):=\inf\{s \in [0,T-t): v(t+s, yC^0(s))\ge \frac{c_{+}}{f_C}\} \wedge (T-t), \\ \\
\tau^*(t,y):=\inf\{s \in [0,T-t): v(t+s, yC^0(s))\le \frac{c_{-}}{f_C}\} \wedge (T-t),
\end{array}
\right.
\eeq
are a saddle point for the zero-sum game (\ref{defv2}).
\end{theorem}
Theorem \ref{thmsaddlepointsv} is proved in Section \ref{saddlepointsvproof}. As a natural byproduct of its proof we obtain the following
\begin{proposition}
\label{semiharmonic}
Take $(t,y)\in[0,T]\times (0,\infty)$ arbitrary but fixed and let $\rho\in[0,T-t]$ be any stopping time. Then under Assumptions \ref{AssProfit} and \ref{scrap} the value function $v$ satisfies
\begin{align}
\vspace{+10pt}
\hspace{-6pt}i)\:\:v(t,y) &\leq \tilde{\mathbb{E}}\bigg\{e^{-\bar{\mu}(\rho\wedge \tau^*)}v(t\hspace{-1pt} +\hspace{-1pt} \rho\hspace{-1pt} \wedge\hspace{-1pt} \tau^*, yC^0(\rho\hspace{-1pt} \wedge\hspace{-1pt} \tau^*))\hspace{-1pt} + \hspace{-1pt}\int_{0}^{\rho \wedge \tau^*}\hspace{-6pt}e^{-\bar{\mu}s} R_c(yC^0(s))\,ds\bigg\} \label{1}\\
\vspace{+10pt}
\hspace{-6pt}ii)\:\:v(t,y) &\geq \tilde{\mathbb{E}}\bigg\{e^{-\bar{\mu}(\sigma^* \wedge \rho)}v(t \hspace{-1pt}+\hspace{-1pt} \sigma^*\hspace{-1pt} \wedge\hspace{-1pt} \rho, yC^0(\sigma^*\hspace{-1pt} \wedge\hspace{-1pt} \rho))\hspace{-1pt} +\hspace{-1pt} \int_{0}^{\sigma^* \wedge \rho}\hspace{-6pt}e^{-\bar{\mu}s} R_c(yC^0(s))\,ds\bigg\}\label{2}\\
\vspace{+10pt}
\hspace{-6pt}iii)\:\:v(t,y) &= \tilde{\mathbb{E}}\bigg\{e^{-\bar{\mu}(\rho\wedge\sigma^*\hspace{-2pt} \wedge \tau^*)}v(t\hspace{-1pt} +\hspace{-1pt} \rho\hspace{-1pt}\wedge\hspace{-1pt}\sigma^*\hspace{-3pt} \wedge\hspace{-1pt} \tau^*, yC^0(\rho\hspace{-1pt}\wedge\hspace{-1pt}\sigma^*\hspace{-3pt} \wedge\hspace{-1pt} \tau^*))\hspace{-1pt} +\hspace{-1pt} \int_{0}^{\rho\wedge\sigma^*\hspace{-1pt} \wedge \tau^*}\hspace{-10pt}e^{-\bar{\mu}s} R_c(yC^0(s))ds\bigg\}\label{3}
\end{align}
\end{proposition}
\begin{proof}
Inequalities $i)$ and $ii)$ are direct consequences of \eqref{semarm02} and \eqref{semarm04}, respectively. Equality $iii)$ follows by arguments as those used to take limits in \eqref{ueps01}.
\end{proof}
The above characterization of the value function was also found via purely probabilistic methods in \cite{Pes09} (under general assumptions) and, in that paper, properties $i),ii)$ and $iii)$ were referred to as \emph{semi-harmonic} characterization of $v$.

\begin{proposition}
\label{primeproprietav}
Under Assumptions \ref{AssProfit} and \ref{scrap} the value function $v(t,y)$ is
\begin{enumerate}
	\item decreasing in $y$ for each $t \in [0,T]$;
	\item decreasing in $t$ for each $y \in (0,\infty)$.
\end{enumerate}
\end{proposition}
\begin{proof}

\begin{enumerate}
\item Fix $t\in [0,T]$ and $y_1 > y_2 > 0$. Let $(\sigma^{*}_1, \tau^{*}_1)$ be optimal for $(t,y_1)$ and $(\sigma^{*}_2, \tau^{*}_2)$ be optimal for $(t,y_2)$ and adopt the sub-optimal strategy $(\sigma^{*}_2, \tau^{*}_1)$ in both the optimization problems of value functions $v(t,y_1)$ and $v(t,y_2)$. Since $R_c(\cdot)$ is decreasing we have
\begin{eqnarray*}
v(t,y_1) - v(t, y_2) \leq \tilde{\mathbb{E}}\bigg\{\int_0^{\tau^{*}_1 \wedge \sigma^{*}_2} e^{-\bar{\mu} s}\Big[R_c(y_1C^0(s)) - R_c(y_2C^0(s))\Big] ds \bigg\} \leq 0. \nonumber
\end{eqnarray*}

\item Given $(t,y) \in [0,T]\times (0,\infty)$, for fixed $\theta \in [0,T-t]$ we define the `$\theta$-shifted' value function as $v^{\theta}(t,y):=v(t+\theta, y)$. Introduce the stopping time
\beq
\label{rhodelta}
\tau^{*}_{\theta}:=\inf\{s \in [0,T-t-\theta):\,\,v^{\theta}(t+s, yC^0(s)) \leq \frac{c_{-}}{f_C}\} \wedge (T-t-\theta),
\eeq
and note that it is optimal for the sup-problem in $v^{\theta}$. Recalling (\ref{stoppingtimesv}) and setting $\rho_{\theta}:=\sigma^{*} \wedge \tau^{*}_{\theta}$, then we obtain
\beq
\label{stimavmenovdelta}
\tilde{\mathbb{E}}\Big\{e^{-\bar{\mu}\rho_{\theta}}\Big[v^{\theta}(t+\rho_{\theta}, yC^0(\rho_{\theta})) - v(t+\rho_{\theta}, yC^0(\rho_{\theta}))\Big]\Big\} \geq v^{\theta}(t,y) - v(t,y),
\eeq
by (\ref{1}) and (\ref{2}).
In order to show that the left-hand side of (\ref{stimavmenovdelta}) is negative, notice that
\begin{itemize}
	\item on $\{\rho_{\theta}= T-t-\theta\}:$ $v^{\theta}(T-\theta, yC^0(T-t-\theta)) = v(T, yC^0(T-t-\theta))=\frac{c_{-}}{f_C}$ and, on the other hand, $v(T-\theta, yC^0(T-t-\theta))\geq \frac{c_{-}}{f_C}$.
	\item on $\{\rho_{\theta}= \tau^{*}_{\theta}\} \bigcap \{\rho_{\theta} < T - t-\theta\}:$ $v^{\theta}(t + \tau^{*}_{\theta}, yC^0(\tau^{*}_{\theta}))=\frac{c_{-}}{f_C}$ and $v(t + \tau^{*}_{\theta}, yC^0(\tau^{*}_{\theta}) \geq \frac{c_{-}}{f_C}$.
	\item on $\{\rho_{\theta}= \sigma^{*}\} \bigcap \{\rho_{\theta} < T - t- \theta\}:$ $v^{\theta}(t + \sigma^{*}, yC^0(\sigma^{*})) \leq \frac{c_{+}}{f_C}$ and $v(t + \sigma^{*}, yC^0(\sigma^{*})) = \frac{c_{+}}{f_C}$.
\end{itemize}
It thus follows that $v(t+\theta,y) \leq v(t,y)$ for any $\theta \in [0,T-t]$ by (\ref{stimavmenovdelta}).
\end{enumerate}
\end{proof}

We notice that in the context of game Call options a proof of the time-monotonicity of the value function somehow related to ours is given in \cite{KunitaSeko}, proof of Lemma 3.3. Authors there compare their payoff functions evaluated at different initial times by picking appropriate stopping times $(\sigma,\tau)$ and by relying on pathwise properties of the underlying process and its time-homogeneity.

We now define the \textsl{continuation region} $\mathcal{C}$ and the two \textsl{stopping regions} $\mathcal{S}_{+}$, $\mathcal{S}_{-}$, by
\begin{align*}
\mathcal{C}\hspace{-1pt}:=\hspace{-1pt}\big\{(t,y)\hspace{-1pt} \in \hspace{-1pt}[0,T]\hspace{-1pt} \times\hspace{-1pt} (0,\infty) : \frac{c_{-}}{f_C}\hspace{-1pt} < \hspace{-1pt}v(t,y)\hspace{-1pt} <\hspace{-1pt} \frac{c_{+}}{f_C}\big\}\:\:\text{and}\:\:
\mathcal{S}_{\pm}\hspace{-1pt}:=\hspace{-1pt}\big\{(t,y)\hspace{-1pt} \in \hspace{-1pt}[0,T] \hspace{-1pt}\times\hspace{-1pt} (0,\infty) : v(t,y)\hspace{-1pt} = \hspace{-1pt}\frac{c_{\pm}}{f_C}\big\}.
\end{align*}
\noindent Notice that $\mathcal{C}$ is an open subset of $[0,T] \times (0,\infty)$ and $\mathcal{S}_{+}, \mathcal{S}_{-}$ are closed ones, due to continuity of $v$ (cf.\ Theorem \ref{jointcontinuityv})
Moreover, for $t \in [0,T]$ fixed, denote by $\mathcal{C}_t := \{y \in (0,\infty) : \frac{c_{-}}{f_C} < v(t,y) < \frac{c_{+}}{f_C}\}$ the $t$-section of the continuation region. Analogously, we introduce the $t$-sections $\mathcal{S}_{+,t}$, $\mathcal{S}_{-,t}$ of the two stopping regions. The fact that $\mathcal{C}$ is open and 1.\ of Proposition \ref{primeproprietav} imply
\begin{proposition}
\label{proprietacontinuazione}
Let Assumptions \ref{AssProfit} and \ref{scrap} hold. Then, for any $t \in [0,T]$, there exist $\hat{y}_{+}(t) < \hat{y}_{-}(t)$ such that $\mathcal{C}_t=(\hat{y}_{+}(t), \hat{y}_{-}(t)) \subset [0,\infty]$, $\mathcal{S}_{+,t}=[0,\hat{y}_{+}(t)]$ and $\mathcal{S}_{-,t}=[\hat{y}_{-}(t),\infty]$.
\end{proposition}

\begin{remark}
\label{tempifrontieretempiv}
It is easy to see that the optimal stopping times $\tau^*$ and $\sigma^*$ of (\ref{stoppingtimesv}) may be written in terms of the free-boundaries $\hat{y}_{+}$ and $\hat{y}_{-}$ of Proposition \ref{proprietacontinuazione} as
\beq
\label{stoppingtimesfrontiere}
\left\{
\begin{array}{ll}
\tau^*(t,y):=\inf\{s \in [0,T-t): yC^0(s)\ge \hat{y}_{-}(t+s)\} \wedge (T-t), \\ \\
\sigma^*(t,y):=\inf\{s \in [0,T-t): yC^0(s) \le \hat{y}_{+}(t+s)\} \wedge (T-t).
\end{array}
\right.
\eeq
\end{remark}

Recalling now Theorem \ref{jointcontinuityv}, Theorem \ref{thmsaddlepointsv}, Proposition \ref{semiharmonic}, Proposition \ref{proprietacontinuazione}, Remark \ref{tempifrontieretempiv} and by using standard arguments based on the strong Markov property (cf. \cite{PeskShir}) we may show that $v$ solves the free-boundary problem
\begin{align}\label{freeb-pr}
\left\{
\begin{array}{cl}
\big(\partial_t+\mathcal{L}-\bar{\mu}\big)v(t,y)=-R_c(y) & \textrm{for $\hat{y}_+(t)<y<\hat{y}_-(t)$,\, $t \in [0,T)$}\\
\\
\big(\partial_t+\mathcal{L}-\bar{\mu}\big)v(t,y)\leq-R_c(y) & \textrm{for $y>\hat{y}_+(t)$,\, $t \in [0,T)$}\\
\\
\big(\partial_t+\mathcal{L}-\bar{\mu}\big)v(t,y)\geq-R_c(y) & \textrm{for $y<\hat{y}_-(t)$,\, $t \in [0,T)$}\\
\\
\frac{c_-}{f_C}\leq v (t,y)\leq \frac{c_+}{f_C} &\textrm{in $[0,T]\times (0,\infty)$} \\
\\
v(t,\hat{y}_\pm(t))=\frac{c_\pm}{f_C} \:\:\:\text{for $t \in [0,T)$} &\text{and}\:\:\: v(T,y)=\frac{c_-}{f_C}\:\:\:\text{for $y>0$}
\end{array}
\right.
\end{align}
with $\mathcal{L}f:=\frac{1}{2}\sigma^2_Cy^2f^{''}-\mu_Cyf^{'}$ for $f\in C^2_b((0,\infty))$. Moreover $v\in C^{1,2}$ in the continuation region $\mathcal{C}$.

\begin{proposition}
\label{freeboundaries}
Under Assumptions \ref{AssProfit} and \ref{scrap} one has
\begin{enumerate}
	\item $\hat{y}_{+}(t)$ and $\hat{y}_{-}(t)$ are decreasing;
	\item $\hat{y}_{+}(t)$ is left-continuous and $\hat{y}_{-}(t)$ is right-continuous;
	\item $ 0 < \hat{y}_{+}(t) < R_c^{-1}(\frac{\bar{\mu}c_+}{f_C})$, for $t \in [0,T)$;
	\item $\lim_{t \uparrow T}\hat{y}_{+}(t)=:\hat{y}_+(T)=0$;
	\item $0 < R_c^{-1}(\frac{\bar{\mu}c_{-}}{f_C}) < \hat{y}_{-}(t) < + \infty$, for $t \in [0,T)$;
	\item $\lim_{t \uparrow T}\hat{y}_{-}(t)=: \hat{y}_{-}(T-) = R_c^{-1}(\frac{\bar{\mu}c_{-}}{f_C})$.
\end{enumerate}
\end{proposition}
\begin{proof}
$1.$ It easily follows from $2.$~of Proposition \ref{primeproprietav} and the fact that $c_\pm/f_C$ are constant (cf.~for instance \cite{Jacka}).
\vspace{+8pt}

$2.$ It follows from point 1.\ and from the fact that $\mathcal{S}_{\pm}$ are closed sets (cf.\ also \cite{Jacka}, proof of Proposition $2.4$).
\vspace{+8pt}

$3.$ To show that $\hat{y}_{+}(t) > 0$ for any $t < T$ we argue by contradiction and we assume that $\hat{y}_{+}(t) = 0$ for some $t \in [0,T)$. From monotonicity of $\hat{y}_{+}(\cdot)$ we have $\hat{y}_{+}(t+s)=0$ for every $s \in [0,T-t)$. Take now $y \in \mathcal{C}_t$ and notice that $yC^0(s)>0$, $s \in [0,T-t)$. It follows that $\sigma^{*}=T-t$,
\begin{eqnarray*}
v(t,y)  \hspace{-1pt} = \hspace{-4pt}  \sup_{\tau \in [0,T-t]}\tilde{\mathbb{E}}\bigg\{\frac{c_{-}}{f_C} e^{-\bar{\mu} \tau} + \int_0^{\tau \wedge (T-t)} \hspace{-5pt} e^{-\bar{\mu}s} R_c(yC^0(s))\,ds\bigg\} \hspace{-1pt} > \hspace{-1pt}  \tilde{\mathbb{E}}\bigg\{\int_0^{T-t} \hspace{-5pt} e^{-\bar{\mu}s} R_c(yC^0(s))\,ds\bigg\},
\end{eqnarray*}
and
\beq
\label{vmenocpiu}
v(t,y) - \frac{c_+}{f_C} > \tilde{\mathbb{E}}\bigg\{\int_0^{T-t} e^{-\bar{\mu}s} R_c(yC^0(s))\,ds\bigg\} - \frac{c_+}{f_C}.
\eeq
The right-hand side of (\ref{vmenocpiu}) may be taken strictly positive by monotone convergence and Inada conditions (cf.\ Assumption \ref{AssProfit}) for $y$ sufficiently small. Such a contradiction proves that $\hat{y}_{+}(t) > 0$ for any $t < T$.

Given that $\mathcal{S}_{+,t}$ is connected (cf.~Proposition \ref{proprietacontinuazione}), $\hat{y}_+$ is positive and decreasing, then $\mathcal{S}_+$ is connected, with non-empty interior $int\,\mathcal{S}_{+}$. Taking $v=c_+/f_C$ in the third equation of \eqref{freeb-pr} one has $int\,\mathcal{S}_{+}\subseteq\big\{(t,y)\in[0,T)\times(0,\infty)\,:\,R_c(y)\ge\frac{\bar{\mu}c_+}{f_C}\big\}$. Therefore, setting $\overline{y}_{+}:=R_c^{-1}(\frac{\bar{\mu}c_+}{f_C})$ one finds $\hat{y}_+(t)\le\overline{y}_{+}$ for all $t\in[0,T)$.
\vspace{+8pt}

$4.$ If $\hat{y}_+(T) > 0$ then we would have $\lim_{y \downarrow \hat{y}_{+}(T)}v(T,y) = \frac{c_{-}}{f_C}$ and $\lim_{t \uparrow T}v(t,\hat{y}_{+}(t))=\frac{c_{+}}{f_C}$, but this contradicts the continuity of $v$ on $[0,T] \times (0,\infty)$ (cf.\ Theorem \ref{jointcontinuityv}).
\vspace{+8pt}

$5.$ We shall first show that $\hat{y}_{-}(t) < + \infty$. To accomplish that we introduce an auxiliary optimal stopping problem with free-boundary $b(t)$ such that $\hat{y}_{-}(t) \leq b(t)$ and $b(t) < +\infty$.
Notice that for any $(t,y) \in [0,T] \times (0,\infty)$ one has
$v(t,y) \leq \tilde{v}(t,y)$,
with
\beq
\label{defsvtilde}
\tilde{v}(t,y) := \sup_{\tau \in [0,T-t]}\tilde{\mathbb{E}}\bigg\{\frac{c_{-}}{f_C}e^{-\bar{\mu}\tau} + \int_0^{\tau}e^{-\bar{\mu}s}R_c(yC^0(s))\,ds\bigg\},
\eeq
by simply taking $\sigma=T-t$ in \eqref{defv2}. It is not hard to see that $\tilde{v}(t,y) \geq \frac{c_{-}}{f_C}$ for any $(t,y) \in [0,T] \times (0,\infty)$, $y \mapsto \tilde{v}(t,y)$ is decreasing for any $t \in [0,T]$ due to the concavity of $R$, $t \mapsto \tilde{v}(t,y)$ is decreasing and continuous for any $y \in (0,\infty)$, and $y \mapsto \tilde{v}(t,y)$ is continuous uniformly in $t$. Then $(t,y) \mapsto \tilde{v}(t,y)$ is continuous on $[0,T] \times (0,\infty)$ and the stopping time $\tilde{\tau}^{*}(t,y):=\inf\big\{s \in [0,T-t): \tilde{v}(t+s, yC^0(s))\leq \frac{c_{-}}{f_C}\big\} \wedge (T-t)$ is optimal (cf.~for instance \cite{PeskShir}). Moreover, there exists a unique monotone decreasing free-boundary
$b(t):=\inf\Big\{y \in (0,\infty): \tilde{v}(t,y) = \frac{c_{-}}{f_C}\Big\}$ for $t<T$,
such that the continuation region $\tilde{\mathcal{C}}$ is the open set $\tilde{\mathcal{C}}:=\{y \in (0,\infty): y < b(t),\,\,t < T\}$.

Since $v(t,y) \leq \tilde{v}(t,y)$, then it is not hard to show that $\hat{y}_{-}(t) \leq b(t)$. We will now prove that $b(t)<\infty$ for all $t\in[0,T]$ adapting arguments by \cite{PeskShir}, Chapter VII, Section $26.2$. Assume there exists $0 < t_o < T$ such that $b(t_o)=+\infty$, then $\tilde{\tau}^{*}(0,y)\geq t_o$ for any $y > 0$ and
$$\tilde{v}(0,y) = \frac{c_{-}}{f_C} + \tilde{\mathbb{E}}\bigg\{\int_0^{\tilde{\tau}^{*}(0,y)}e^{-\bar{\mu}s}\left(R_c(yC^0(s)) - \frac{\bar{\mu}c_{-}}{f_C}\right)\,ds\bigg\},$$
by \eqref{defsvtilde} and an integration by parts.
Fix $\epsilon > 0$, set $\overline{y}_{-}:=R_c^{-1}(\frac{\bar{\mu}c_{-}}{f_C})$ and define the stopping time
$\tau^{\epsilon}_{\overline{y}_{-}}(0,y):=\inf\{s \in [0,T): yC^0(s) \leq \overline{y}_{-} + \epsilon\} \wedge T.$
Observe that there exists $q_\epsilon>0$ such that $R_c(y) - \frac{\bar{\mu}c_{-}}{f_C} < -q_\epsilon$ for all $y \geq \overline{y}_{-} + \epsilon$, by (\ref{freeb-pr}).

From now on we write $\tilde{\tau}^{*}\equiv\tilde{\tau}^{*}(0,y)$ and $\tau^{\epsilon}_{\overline{y}_{-}}\equiv\tau^{\epsilon}_{\overline{y}_{-}}(0,y)$ to simplify notation. We then have
\begin{align}
\tilde{v}(0,y) -\frac{c_{-}}{f_C}=&\, \tilde{\mathbb{E}}\bigg\{\mathds{1}_{\{ \tilde{\tau}^{*} \leq \tau^{\epsilon}_{\overline{y}_{-}}\}}\int_0^{\tilde{\tau}^{*}}e^{-\bar{\mu}s}\left(R_c(yC^0(s)) - \frac{\bar{\mu}c_{-}}{f_C}\right)\,ds\bigg\} \nonumber \\
& +\, \tilde{\mathbb{E}}\bigg\{\mathds{1}_{\{ \tilde{\tau}^{*} > \tau^{\epsilon}_{\overline{y}_{-}}\}}\int_{0}^{\tilde{\tau}^{*}}e^{-\bar{\mu}s}\left(R_c(yC^0(s)) - \frac{\bar{\mu}c_{-}}{f_C}\right)\,ds\bigg\} \\
\leq & - q_\epsilon\, t_o\, \tilde{\mathbb{P}}(\tilde{\tau}^{*} \leq \tau^{\epsilon}_{\overline{y}_{-}}) + c(y)\tilde{\mathbb{P}}(\tilde{\tau}^{*} > \tau^{\epsilon}_{\overline{y}_{-}})^{\frac{1}{2}}, \nonumber
\end{align}
where we have used H\"older inequality and set $c(y):=\tilde{\mathbb{E}}\{|\int_{0}^{T}e^{-\bar{\mu}s}[R_c(yC^0(s)) - \frac{\bar{\mu}c_{-}}{f_C}]\,ds|^2\}^{\frac{1}{2}} < \infty$ (which is bounded by some positive constant $\bar{\kappa}$ as $y \uparrow \infty$ by Lemma \ref{stimaconcavitaR} in Section \ref{appproofs}). If now
\beq
\label{ifnowlimit}
\lim_{y \uparrow \infty}\tilde{\mathbb{P}}(\tilde{\tau}^{*} (0,y)> \tau^{\epsilon}_{\overline{y}_{-}}(0,y))=0,
\eeq
then we have $\tilde{v}(0,y) - \frac{c_{-}}{f_C} < 0$ for $y$ sufficiently large, thus reaching a contradiction.

To verify that, take now $y>\overline{y}_-+\epsilon$ and notice that $\{\tau^{\epsilon}_{\overline{y}_{-}}(0,y) < \tilde{\tau}^{*}(0,y)\} \subseteq \{\inf_{0 \leq s \leq T}yC^0(s) \leq \overline{y}_{-} + \epsilon\}$.
Then from \eqref{capacitysolution} and \eqref{mus} we obtain
\begin{eqnarray*}
\tilde{\mathbb{P}}(\tilde{\tau}^{*}(0,y) > \tau^{\epsilon}_{\overline{y}_{-}}(0,y))\,  \hspace{-2pt} \leq \hspace{-2pt}\,  \tilde{\mathbb{P}}\Big(\sup_{0 \leq s \leq T}|\sigma_C \tilde{W}(s) + \hat{\mu}_C s| \geq \ln\left(\frac{y}{\overline{y}_{-} + \epsilon}\right)\Big)\, \hspace{-2pt} \leq \hspace{-2pt} \,C_T\Big[\ln\left(\frac{y}{\overline{y}_{-} + \epsilon}\right)\Big]^{-1},\nonumber
\end{eqnarray*}
by Markov inequality and standard estimates on the solutions of stochastic differential equations (cf.~\cite{Friedman}, Volume 1, Chapter 5). It follows (\ref{ifnowlimit}) and that $b(t_o)<+\infty$.

It remains now to exclude the case $t_o=0$ as well. Assume $b(0)=+\infty$, take $\delta>0$, $0 < t < \delta$ and define
\begin{align}\label{aux}
\tilde{v}_{\delta}(t,y):=\sup_{\tau \in [0,T + \delta -t]}\tilde{\mathbb{E}}\bigg\{\frac{c_{-}}{f_C}e^{-\bar{\mu}\tau} + \int_0^{\tau}e^{-\bar{\mu}s}R_c(yC^0(s))\,ds\bigg\}.
\end{align}
Hence $\tilde{v}_{\delta}(t,y) \ge \tilde{v}(t,y)$ and $\tilde{v}_{\delta}(t+\delta,y) = \tilde{v}(t,y)$. If we now denote by $b_\delta$ the free-boundary of problem \eqref{aux}, we easily find $b(0)=b_\delta(\delta)$. We may thus repeat same arguments as those employed in the case $t_o >0$ to obtain a contradiction and conclude that $b(0) < +\infty$. Finally, we may proceed as in the second part of the proof of $3.$~to show that $\hat{y}_-(t)\ge \overline{y}_-$ for all $t\in[0,T]$.
\vspace{+8pt}

$6.$ Define $\hat{b}_{-}(t):=\hat{y}_{-}(t) - \overline{y}_{-}$, with $\overline{y}_{-}:=R_c^{-1}(\frac{\bar{\mu}c_{-}}{f_C})$. This curve is nonnegative thanks to $5.$~and $\hat{b}_{-}(t) < \hat{y}_{-}(t)$ for all $t \leq T$; that is, $(t,\hat{b}_{-}(t)) \in \mathcal{S}_+\cup\,\mathcal{C}$ and $v(t, \hat{b}_{-}(t)) > \frac{c_{-}}{f_C}$ for all $t\leq T$.
Assume now $\hat{b}_{-}(T-) > 0$ then $\lim_{y \uparrow \hat{b}_{-}(T-)}v(T,y)=\frac{c_{-}}{f_C}$ and $\lim_{t \uparrow T}v(t,\hat{b}_{-}(t))> \frac{c_{-}}{f_C}$, but this is not possible being $v(t,y)$ continuous on $[0,T] \times (0,\infty)$ by Theorem \ref{jointcontinuityv}. 
\end{proof}

\begin{theorem}\label{cont-bdry}
The free-boundaries $t\mapsto\hat{y}_+(t)$ and $t\mapsto\hat{y}_-(t)$ are continuous on $[0,T]$.
\end{theorem}
\begin{proof}
A proof of continuity by standard use of Newton-Leibnitz formula (cf.~\cite{PeskShir} for a list of examples) seems rather hard to implement for the lower free-boundary $\hat{y}_+$. In fact, inequalities that one would normally try to use cannot be obtained in that case. For this reason we abandon that approach and proceed via arguments inspired by PDE theory (cf.~also \cite{TDA13}).

$1.$ We start by considering the upper free-boundary, $\hat{y}_-(t)$, which is right-continuous (cf.~Proposition \ref{freeboundaries}). Let us argue by contradiction and assume that there exists $t_o\in(0,T)$ where a discontinuity of $\hat{y}_-(\,\cdot\,)$ occurs; that is, $t_o$ is such that $\hat{y}_-(t_o-)>\hat{y}_-(t_o)$. Fix $t^\prime\in(0,t_o)$, $y_1$ and $y_2$ such that $\hat{y}_-(t_o)<y_1<y_2<\hat{y}_-(t_o-)$ and define a domain $\mathcal{R}\subset\mathcal{C}$ by $\mathcal{R}:=(t^\prime,t_o)\times(y_1,y_2)$. Its parabolic boundary $\partial_P\mathcal{R}$ is clearly formed by the horizontal lines $[t^\prime,t_o)\times\{y_i\}$, $i=1,2$ and by the vertical line $\{t_o\}\times[y_1,y_2]$. From the first equation in \eqref{freeb-pr} and the definition of $\mathcal{R}$ we obtain that $v$ (uniquely) solves the Dirichlet-Cauchy problem
\begin{eqnarray}\label{continuity01}
\begin{array}{ll}
\vspace{+5pt}
\big(\partial_t+\mathcal{L}-\bar{\mu}\big)u(t,y)=-R_c(y) & \textrm{in $\mathcal{R}$}\\
\vspace{+5pt}
u(t,y_{\,i})=v(t,y_{\,i}) & \text{$i=1,2$ and $t\in[t^\prime,t_0)$}\\
\vspace{+5pt}
u(t_0,y)=\frac{c_-}{f_C} & y\in[y_1,y_2].
\end{array}
\end{eqnarray}
We denote by $C^{\infty}_c([y_1,y_2])$ the set of functions with infinitely many continuous derivatives and compact support in $[y_1,y_2]$. Take $\psi\ge0$ arbitrary in $C^{\infty}_c([y_1,y_2])$ and such that $\int^{y_2}_{y_1}{\psi(y)dy=1}$. Multiply the first equation in \eqref{continuity01} (with $v$ instead of $u$) by $\psi$ and integrate over $[y_1,y_2]$. It gives
\begin{align}\label{continuity02}
\int_{y_1}^{y_2}{\partial_t v(t,y)\psi(y)dy}=-\int_{y_1}^{y_2}{\Big[\big(\mathcal{L}-\bar{\mu}\big)v(t,y)+R_c(y)\Big]\psi(y)dy}\qquad
\textrm{for all $t\in[t^\prime,t_o)$}.
\end{align}
We now integrate by parts twice the term on the right hand side of \eqref{continuity02} and obtain
\begin{align}\label{continuity03}
\int_{y_1}^{y_2}{\partial_t v(t,y)\psi(y)dy}=-\int_{y_1}^{y_2}{\Big[v(t,y)\big(\mathcal{L}^*-\bar{\mu}\big)+R_c(y)\Big]\psi(y)dy}
\qquad\textrm{for all $t\in[t^\prime,t_o)$},
\end{align}
where ${\mathcal{L}}^*$ is the adjoint of $\mathcal{L}$. 
Recall that $\partial_tv$ is negative by $3.$~of Proposition \ref{primeproprietav}. Take the limit as $t\to t_o$ in \eqref{continuity03}, use dominated convergence, Theorem \ref{jointcontinuityv} and the last equation in \eqref{continuity01} to obtain
\begin{align}\label{continuity04}
0\ge&\lim_{t\uparrow t_o} \int_{y_1}^{y_2}{\partial_t v(t,y)\psi(y)dy}=-\int_{y_1}^{y_2}{\Big[v(t_o,y)\big(\mathcal{L}^*-\bar{\mu}\big)+R_c(y)\Big]\psi(y)dy}
\nonumber\\
=&-\int_{y_1}^{y_2}{\Big[\frac{c_-}{f_C}\big(\mathcal{L}^*-\bar{\mu}\big)+R_c(y)\Big]\psi(y)dy}=
-\int_{y_1}^{y_2}{\Big[R_c(y)-\frac{\bar{\mu}c_-}{f_C}\Big]\psi(y)dy}.
\end{align}
Notice that $y\mapsto R_c(y)-\frac{\bar{\mu}c_-}{f_C}$ is continuous and strictly negative for $y\in[y_1,y_2]$, by $5.$~of Proposition \ref{freeboundaries} as $y_1>\overline{y}_-$ and $R_{c}(\cdot)$ is strictly decreasing. Hence, there exists a positive constant $\ell:=\ell(y_1,y_2)$ such that $\sup_{y\in[y_1,y_2]}\Big[R_c(y)-\frac{\bar{\mu}c_-}{f_C}\Big]\le-\ell$ and from the last term of \eqref{continuity04} we find
\begin{align}\label{continuity05}
0\ge-\int_{y_1}^{y_2}{\Big[R_c(y)-\frac{\bar{\mu}c_-}{f_C}\Big]\psi(y)dy}\geq\ell\int^{y_2}_{y_1}{\psi(y)dy}=
\ell>0,
\end{align}
by using that $\int_{y_1}^{y_2}{\psi(y)dy}=1$. Therefore, we reach a contradiction and $\hat{y}_-(t_o-)=\hat{y}_-(t_o)$.

$2.$ We will now prove continuity of the lower boundary $\hat{y}_+(\cdot)$. Again we argue by contradiction and assume that there exists $t_o \in(0,T)$ where a discontinuity of $\hat{y}_+(\cdot)$ occurs.
Then $t_o$ is such that $\hat{y}_+(t_o) > \hat{y}_+(t_o+)$. As before we define an open bounded domain $\mathcal{R}\subset \mathcal{C}$ with parabolic boundary $\partial_P\mathcal{R}$ formed by the horizontal lines $[t_o, t^\prime) \times\{y_i\}$, $i=1,2$ and by the vertical line $\{t^\prime\}\times[y_1,y_2]$ with $y_1$ and $y_2$ such that $\hat{y}_+(t_o+)<y_1<y_2<\hat{y}_+(t_o)$ and arbitrary $t^\prime\in(t_o,T)$.
We have that $u:=\frac{c_{+}}{f_C} - v$ solves
\beq
\label{PDEu}
\big(\partial_t+\mathcal{L}-\bar{\mu}\big)u(t,y) = R_c(y) - \frac{\bar{\mu}c_+}{f_C}, \quad (t,y) \in \mathcal{R},
\eeq
by \eqref{freeb-pr} and additionally $u(t_o,y) \equiv 0$ for $y \in [y_1,y_2]$. Regularity of $R_c$ and of the coefficients in $\mathcal{L}$ imply that $u_{yyy}$ and $u_{ty}$ exist and are continuous in $\mathcal{R}$ (cf.~\cite{FriedmanPDE}, Theorem $10$, Chapter $3$). Differentiating (\ref{PDEu}) with respect to $y$ and defining $\bar{u}:=u_y$ we easily obtain (cf.~\eqref{mus})
\beq
\label{PDEbaru}
\bar{u}_t(t,y) + \frac{1}{2}\sigma^2_C y^2 \bar{u}_{yy}(t,y) + (2\sigma^2_C - \mu_C)y\bar{u}_y(t,y) + (\sigma^2_C -\mu_C- \bar{\mu})\bar{u}(t,y) = R_{cc}(y) < 0, \quad (t,y) \in \mathcal{R},
\eeq
as $R$ is strictly concave. It will be useful in what follows to define
\begin{align}\label{A-op}
\mathcal{G}f(y):=\frac{1}{2}\sigma^2_C y^2 {f}^{\prime\prime}(y) + (2\sigma^2_C -{\mu}_C)y{f}^\prime(y) + (\sigma^2_C-\mu_C- \bar{\mu}){f}(y)\quad\text{for $f\in C^2_b(\mathbb{R})$}.
\end{align}

Again we consider a test function $\psi\in C^\infty_c([y_1,y_2])$ such that $\psi\ge0$ and $\int^{y_2}_{y_1}{\psi(y)dy}=1$. We define a function $F_\psi:(t_o,T)\to\mathbb{R}$ by
\begin{align}\label{F-psi}
F_\psi(t):=\int^{y_2}_{y_1}{\bar{u}_t(t,y)\psi(y)dy},\qquad t\in(t_o,T).
\end{align}
Now, denoting by $\mathcal{G}^*$ the formal adjoint of $\mathcal{G}$ in \eqref{A-op}, \eqref{PDEbaru} gives
\begin{align}\label{F-psi2}
F_\psi(t)=\int^{y_2}_{y_1}{\big[R_{cc}(y)-\mathcal{G}\bar{u}(t,y)\big]\psi(y)dy}=
\int^{y_2}_{y_1}{\big[R_{cc}(y)\psi(y)+{u}(t,y)\frac{\partial}{\partial y}\big(\mathcal{G}^*\psi\big)(y)\big]dy}.
\end{align}

The map $t\mapsto F_\psi(t)$ is clearly continuous on $(t_o,T)$, its right-limit at $t_o$ is well defined thanks to dominated convergence and it is equal to
\begin{align}\label{F-lim}
F_\psi(t_o+):=\lim_{t\downarrow t_o}F_\psi(t)=\int^{y_2}_{y_1}{R_{cc}(y)\psi(y)dy},
\end{align}
by recalling that $u(t_o,y)\equiv0$ for $y\in[y_1,y_2]$. From strict concavity of $R$, there exists $\ell>0$ such that $R_{cc}(y)<-\ell$ in $[y_1,y_2]$ and hence $F_\psi(t_o+)<-\ell$. It follows that there exists $\epsilon>0$ such that $F_\psi(t)<-\ell/2$ for all $t\in(t_o,t_o+\epsilon]$ by continuity of $F_\psi$. Now, take $0<\delta<\epsilon$ arbitrary, then \eqref{F-psi} and Fubini's theorem give
\begin{align}\label{F-int}
-\frac{\ell}{2}(\epsilon-\delta)>&\int^{\epsilon}_{\delta}{F_\psi(t_o+s)ds}=\int^{y_2}_{y_1}{\big[\bar{u}(t_o+\epsilon,y)
-\bar{u}(t_o+\delta,y)\big]\psi(y)dy}\nonumber\\
=&\int^{y_2}_{y_1}{{u}_y(t_o+\epsilon,y)\psi(y)dy}+\int^{y_2}_{y_1}{u(t_o+\delta,y)\psi^\prime(y)dy}.
\end{align}
Taking limits as $\delta\to0$ we obtain
\begin{align}\label{F-int02}
-\frac{\ell}{2}\epsilon\ge\int^{y_2}_{y_1}{u_y(t_o+\epsilon,y)\psi(y)dy}=
-\int^{y_2}_{y_1}{v_y(t_o+\epsilon,y)\psi(y)dy}\ge0
\end{align}
since $y \mapsto v(t,y)$ is decreasing (cf.\ Proposition \ref{primeproprietav}). Therefore we reach a contradiction and $\hat{y}_{+}$ must be continuous on $(0,T)$.

$3.$ It remains only to prove continuity at $T$. Since $\hat{y}_+$ is left-continuous (cf.~$2.$~of Proposition \ref{freeboundaries}) then it is continuous on $[0,T]$. On the other hand, $\hat{y}_-$ is right-continuous and decreasing with $\hat{y}_-(T-)=R^{-1}_c\big(\frac{\bar{\mu}c_+}{f_C}\big)$ (see $6.$~of Proposition \ref{freeboundaries}). Then, it must be continuous on $[0,T]$ since $\hat{y}_-(t)\geq R^{-1}_c\big(\frac{\bar{\mu}c_+}{f_C}\big)$ for all $t\in[0,T]$.
\end{proof}

Recall that $R \in C^2((0,\infty))$ and it is strictly concave. We now make the following
\begin{Assumptions}
\label{integrabilitysmooth}
For any $ y_o \geq R_c^{-1}(\overline{\mu}c_{-}/f_C)$ there exists $\delta_o:=\delta_o(y_o)$ such that
\begin{align}\label{integrabilitysmooth-01}
 \tilde{\mathbb{E}}\bigg\{\int_0^T e^{-\bar{\mu}s}
\inf_{\{y:|y-y_o| \le\delta_o\}}R_{cc}(yC^0(s))\,ds\bigg\} > - \infty.
\end{align}
\end{Assumptions}
\noindent Since $R_{cc}$ is continuous away from zero and $C^0$ is a geometric Brownian motion, it is easy to see that Assumption \ref{integrabilitysmooth} is fulfilled by a large class of production functions meeting Inada conditions (cf.\ Assumption \ref{AssProfit}). That is the case for example of a Cobb-Douglas production function.

\begin{proposition}
\label{propsmoothfit}
Let Assumption \ref{AssProfit}, \ref{scrap} and \ref{integrabilitysmooth} hold. Then the smooth-fit property holds at the free-boundaries $\hat{y}_{+}$ and $\hat{y}_{-}$. That is,
\beq
\label{smoothpm}
v_{y}(t,\hat{y}_{-}(t)-) = 0\quad\text{and}\quad v_{y}(t,\hat{y}_{+}(t)+) = 0 \quad \text{for $t \in [0,T)$.}
\eeq
\end{proposition}
\begin{proof}
We start by proving the first in \eqref{smoothpm}. Fix $\epsilon > 0$ and $t_o \in [0,T)$ and let $(\sigma^{*}_{-\epsilon}, \tau^{*}_{-\epsilon})$ be optimal in
$v(t_o,\hat{y}_{-}(t_o) - \epsilon)$ in the sense of (\ref{stoppingtimesv}). Since the free-boundary $\hat{y}_-$ is monotone decreasing, it is not hard to show that
\begin{align}\label{smooth00}
\lim_{\epsilon\to 0}\tau^*_{-\epsilon}=0, \quad \mbox{a.s.}
\end{align}
by the law of iterated logarithm at zero for Brownian motion.

Take $\sigma^{*}:=\sigma^{*}(t_o, \hat{y}_{-}(t_o))$ as in (\ref{stoppingtimesv}) and adopt the sub-optimal stopping strategy
$(\sigma^*,\tau^*_{-\epsilon})$ in both the optimization problems with value functions $v(t_o,\hat{y}_-(t_o))$ and $v(t_o,\hat{y}_-(t_o)
-\epsilon)$. Then, using that $y \mapsto v(t,y)$ is decreasing (cf.\ Proposition \ref{primeproprietav}) we obtain
$$0 \leq  v(t_o, \hat{y}_-(t_o) - \epsilon) - v(t_o, \hat{y}_-(t_o)) \leq \tilde{\mathbb{E}}\bigg\{\int_0^{\sigma^* \wedge \tau^{*}_{-\epsilon}}\hspace{-6pt} e^{-\bar{\mu} s} \Big[R_c\big((\hat{y}_{-}(t_o)-\epsilon)C^0(s)\big)-R_c\big(\hat{y}_{-}(t_o)C^0(s)\big)\Big] ds\bigg\}$$
and an application of the mean value theorem gives
\begin{align}\label{smooth02}
0 \leq & \,v(t_o, \hat{y}_-(t_o) - \epsilon) - v(t_o, \hat{y}_-(t_o)) \leq
-\epsilon\tilde{\mathbb{E}}\bigg\{\int_0^{\sigma^* \wedge \tau^{*}_{-\epsilon}} e^{-\bar{\mu} s}
R_{cc}\big(\xi_{\epsilon}C^0(s)\big) ds\bigg\}
\end{align}
for some $\xi_\epsilon\in[\hat{y}_-(t_o)-\epsilon,\hat{y}_-(t_o)]$.
Thanks to Assumption \ref{integrabilitysmooth}, fixed $y_o:=\hat{y}_-(t_o)$, we can always find $\delta_o> \epsilon$, such that
\eqref{integrabilitysmooth-01} holds. Then, dividing \eqref{smooth02} by $\epsilon$ we have
\begin{align}\label{smooth03}
0 \leq & \frac{v(t_o, y_o - \epsilon) - v(t_o, y_o)}{\epsilon} \leq
\tilde{\mathbb{E}}\bigg\{-\int_0^{\sigma^* \wedge \tau^{*}_{-\epsilon}} e^{-\bar{\mu} s}
\inf_{y\in[y_o-\delta_o,y_o]}R_{cc}\big(yC^0(s)\big) ds\bigg\},
\end{align}
for all $\epsilon<\delta_o$. Note that
\begin{align}\label{smooth04}
-\int_0^{\sigma^* \wedge \tau^{*}_{-\epsilon}} e^{-\bar{\mu} s}
\inf_{y\in[y_o-\delta_o,y_o]}R_{cc}\big(yC^0(s)\big) ds\le -\int_0^{T} e^{-\bar{\mu} s}
\inf_{y\in[y_o-\delta_o,y_o]}R_{cc}\big(yC^0(s)\big) ds=:H
\end{align}
and $H$ is $\tilde{\mathbb{P}}$--integrable by Assumption \ref{integrabilitysmooth}. Therefore, Fatou's lemma, \eqref{smooth00}
and \eqref{smooth03} imply the first equation of \eqref{smoothpm}.
To prove the second one in \eqref{smoothpm} arguments as above seem not to be applicable. In fact, fixed $t_o \in [0,T)$, if we take $(\sigma^{*}_{+\epsilon}, \tau^{*}_{+\epsilon})$ optimal for $v(t_o,\hat{y}_{+}(t_o)+ \epsilon)$, then it might happen that $\tilde{\mathbb{P}}\big(\lim_{\epsilon\to0}\sigma^*_{+\epsilon}>0\big)>0$ 
since $\hat{y}_{+}$ is only proved to be continuous and decreasing. Roughly speaking, we cannot exclude the case that $\hat{y}{\,'}_{+}(t_o) = -\infty$ at countably many points $t_o$. To avoid this difficulty, we shall adopt a different argument that extends \cite{PeskShir}, Chapter IV, Section $9.3$, to the present setting of a Zero-Sum Game on finite time-horizon and with a running cost.

Let $h$ be a $C^2$ solution on $(0,\infty)$ of the second-order ordinary differential equation $\mathcal{L}h(y)=R_c(y)$.
Fix $(t_o,y) \in [0,T) \times (0,\infty)$ and let $\rho$ be a stopping time. Then, from $\textit{i)}$ of Proposition \ref{semiharmonic} one has
\begin{eqnarray}
\label{vh}
v(t_o,y) &\hspace{-0.25cm} \leq \hspace{-0.25cm}& \tilde{\mathbb{E}}\bigg\{e^{-\bar{\mu}(\rho \wedge \tau^{*})}v(t_o + \rho \wedge \tau^{*}, yC^0(\rho \wedge \tau^{*})) + \int_0^{\rho \wedge \tau^{*}}e^{-\bar{\mu}s}R_c(yC^0(s))ds\bigg\} \nonumber \\
&\hspace{-0.25cm} \leq \hspace{-0.25cm}& \tilde{\mathbb{E}}\bigg\{v(t_o + \rho \wedge \tau^{*}, yC^0(\rho \wedge \tau^{*})) + \int_0^{\rho}R_c(yC^0(s))ds\bigg\}. 
\end{eqnarray}
Therefore, it follows
\beq
\label{veh}
v(t_o,y) + h(y) \leq \tilde{\mathbb{E}}\Big\{v(t_o + \rho \wedge \tau^{*}, yC^0(\rho \wedge \tau^{*})) + h(yC^0(\rho))\Big\}
\eeq
by Dynkin formula and the definition of $h$. For any $\alpha > 0$ we define the hitting time
$\tau_{\alpha}:=\inf\{s \geq 0: yC^0(s) = \alpha\}$. Take $0 < c < y < d < \overline{y}_{-}$ and set $\rho:=\tau_c \wedge \tau_d$. Then $\rho \wedge \tau^{*} = \rho \wedge (T-t_o)$ and (\ref{veh}) gives
\begin{align}
\label{veh2}
v(t_o,y)\hspace{-1pt} +\hspace{-1pt} h(y)  \leq&
\,\tilde{\mathbb{E}}\Big\{v(t_o + \tau_c, c)\mathds{1}_{\{\rho < T-t_o\}}\mathds{1}_{\{\tau_c < \tau_d\}} + v(t_o + \tau_d, d)\mathds{1}_{\{\rho < T-t_o\}}\mathds{1}_{\{\tau_d < \tau_c\}}\Big\} \nonumber \\
 &  +\, \frac{c_{-}}{f_C}\tilde{\mathbb{E}}\Big\{\mathds{1}_{\{\rho \geq T-t_o\}}\Big\} + h(c)\tilde{\mathbb{P}}(\tau_c < \tau_d) + h(d)\tilde{\mathbb{P}}(\tau_d < \tau_c).
\end{align}
Recall now that $t \mapsto v(t,y)$ is decreasing (cf.\ Proposition \ref{primeproprietav}), that $v(t_o,y) \geq \frac{c_{-}}{f_C}$ for any $y \in (0,\infty)$, and that $v(T,c) = v(T,d) = c_{-}/f_C$. Hence (\ref{veh2}) implies
\begin{eqnarray}
\label{veh3}
v(t_o,y) + h(y)
& \hspace{-0.25cm} \leq \hspace{-0.25cm}& v(t_o,c)\tilde{\mathbb{E}}\Big\{\mathds{1}_{\{\rho < T-t_o\}}\mathds{1}_{\{\tau_c < \tau_d\}}\Big\} + v(t_o,d)\tilde{\mathbb{E}}\Big\{\mathds{1}_{\{\rho < T-t_o\}}\mathds{1}_{\{\tau_d < \tau_c\}}\Big\} \\
& & \hspace{0.8cm} +\,v(t_o,c)\tilde{\mathbb{E}}\Big\{\mathds{1}_{\{\rho \geq T-t_o\}}\mathds{1}_{\{\tau_c < \tau_d\}}\Big\} + v(t_o,d)\tilde{\mathbb{E}}\Big\{\mathds{1}_{\{\rho \geq T-t_o\}}\mathds{1}_{\{\tau_d < \tau_c\}}\Big\} \nonumber \\
& & \hspace{0.8cm} +\,h(c)\tilde{\mathbb{P}}(\tau_c < \tau_d) + h(d)\tilde{\mathbb{P}}(\tau_d < \tau_c) \nonumber \\
& \hspace{-0.25cm} = \hspace{-0.25cm}& [v(t_o,c) + h(c)]\tilde{\mathbb{P}}(\tau_c < \tau_d) + [v(t_o,d) + h(d)]\tilde{\mathbb{P}}(\tau_d < \tau_c) \nonumber \\
& \hspace{-0.25cm} = \hspace{-0.25cm}& [v(t_o,c) + h(c)]\frac{S(d)-S(y)}{S(d)-S(c)} + [v(t_o,d) + h(d)]\frac{S(y)-S(c)}{S(d)-S(c)}, \nonumber
\end{eqnarray}
where $S$ is the scale function of $C^0$ (see, e.g., \cite{RevuzYor}, Chapter VII, Section $3$).
It follows that, for fixed $t_o \in [0,T)$, the function $y \mapsto u(t_o,y)$, defined by $u(t_o,y):=v(t_o,y) + h(y)$, is $S$-convex (see, e.g., \cite{RevuzYor}, p.\ $546$). Therefore
$$y \mapsto \frac{u(t_o,y) - u(t_o,x)}{S(y) - S(x)}$$
is increasing on $[c,d]$, for every $x \in (c,d)$.

Notice now that $S(\cdot)\in C^1((0,\infty))$ and recall that $h \in C^2((0,\infty))$.
Then, for arbitrary but fixed $t_o \in [0,T)$, we can apply arguments as in \cite{PeskShir}, Chapter IV, Section $9.3$, and obtain $u_y(t_o, \hat{y}_{+}(t_o)+) = h'(\hat{y}_{+}(t_o))$. Hence $v_y(t_o, \hat{y}_{+}(t_o)+) = 0$ for $t_o\in[0,T)$, by definition of $u$.
\end{proof}

In the next Theorem we will find non-linear integral equations that characterize the free-boundaries and the value function $v$ of our zero-sum optimal stopping game.
\begin{theorem}
\label{Volterra}
Under Assumption \ref{AssProfit}, \ref{scrap} and \ref{integrabilitysmooth}, the value function $v$ of problem \eqref{defv2} has the following representation
\begin{align}\label{vrepr}
v(t,y)=& \,e^{-\bar{\mu}(T-t)}\frac{c_-}{f_C}+\int^{T-t}_0{e^{-\bar{\mu}s}\tilde{\mathbb{E}}\bigg\{
R_c(yC^0(s))\mathds{1}_{\{\hat{y}_+(t+s)<yC^0(s)<\hat{y}_-(t+s)\}}\bigg\}}ds\nonumber\\
&+\frac{\bar{\mu}}{f_C}\int^{T-t}_0{e^{-\bar{\mu}s}\Big[c_+\tilde{\mathbb{P}}\Big(yC^0(s)<\hat{y}_+(t+s)\Big)+c_-
\tilde{\mathbb{P}}\Big(yC^0(s)>\hat{y}_-(t+s)\Big)\Big]ds},
\end{align}
where $\hat{y}_+$ and $\hat{y}_-$ are continuous, decreasing curves solving the coupled integral equations
\begin{align}
\frac{c_{\pm}}{f_C}=&\, e^{-\bar{\mu}(T-t)}\frac{c_-}{f_C}+\int^{T-t}_0{e^{-\bar{\mu}s}\tilde{\mathbb{E}}\bigg\{
R_c(\hat{y}_{\pm}(t)C^0(s))\mathds{1}_{\{\hat{y}_+(t+s)<\hat{y}_{\pm}(t)C^0(s)<\hat{y}_-(t+s)\}}\bigg\}}ds\nonumber\\
&+\frac{\bar{\mu}}{f_C}\int^{T-t}_0{e^{-\bar{\mu}s}\Big[c_+\tilde{\mathbb{P}}\Big(\hat{y}_{\pm}(t)C^0(s)<\hat{y}_+
(t+s)\Big)+c_-\tilde{\mathbb{P}}\Big(\hat{y}_{\pm}(t)C^0(s)>\hat{y}_-(t+s)\Big)\Big]ds}
\label{int-eq-}
\end{align}
with boundary conditions
\begin{align}\label{int-eq-bd}
\hspace{-45pt}\hat{y}_-(T)=R^{-1}_c\Big(\frac{\bar{\mu} c_{-}}{f_C}\Big) \qquad\&\qquad \hat{y}_+(T)=0
\end{align}
and such that
\begin{align}\label{int-eq-bd2}
R^{-1}_c\Big(\frac{\bar{\mu}c_-}{f_C}\Big)<\hat{y}_-(t)<+\infty\qquad\&\qquad0<\hat{y}_+(t)< R^{-1}_c\Big(\frac{\bar{\mu}c_+}{f_C}\Big)\quad\textrm{for all $t\in[0,T)$.}
\end{align}
\end{theorem}
\begin{proof}
We aim at applying local time-space formula by \cite{Peskir}, Theorem $3.1$. In order to do so we will verify that $v$ fulfils suitable sufficient conditions. That is, for $\eta>0$ arbitrary but fixed
\begin{align}
&\big(\partial_t+\mathcal{L}-\bar{\mu}\big)v\quad\textrm{is bounded on any compact $K$ in $[0,T-\eta]\times(0,+\infty)$}\label{lts01}\\
\nonumber\\
&t\mapsto v_y(t,\hat{y}_\pm(t)\pm)=0\quad\textrm{is continuous on $[0,T-\eta]$},&\label{lts02}\\
\nonumber\\
&t\mapsto v(t,\hat{y}_\pm(t)\pm)\quad\textrm{is of bounded variation on $[0,T-\eta]$}.&\label{lts03}
\end{align}
Conditions \eqref{lts01} and \eqref{lts02} follow from \eqref{freeb-pr} and the smooth-fit property (cf.\ Proposition \ref{propsmoothfit}).

To verify \eqref{lts03} we need a bit more work.
There exists $\delta_\eta:=\delta(\eta)>0$ such that $\hat{y}_+(t)>\delta_\eta$ for all $t\in[0,T-\eta]$,
by $3.$ of Proposition \ref{freeboundaries} and Theorem \ref{cont-bdry}. Also, there exist:
$L_\eta:=L(\delta_\eta)>0$
such that $\big|v_y(t,y)\big|\le L_\eta$ for all $y\in[\hat{y}_+(t)-\delta_\eta,\hat{y}_+(t)+\delta_\eta]$,
$t\in[0,T-\eta]$ by \eqref{lts02} and $R_\eta:=R(\delta_\eta)>0$ such that $R_c(y)\le R_\eta$ on
$y\ge\hat{y}_+(T - \eta)-\delta_\eta$. From these bounds, $2.$ of Proposition \ref{primeproprietav}, and the first equation
in \eqref{freeb-pr} we find
\begin{align}\label{lts04}
\frac{\sigma^2_C}{2}y^2v_{yy}\ge -R_\eta-\big|\hat{\mu}_C+\sigma^2_C/2\big|\,L_\eta y+\bar{\mu}\frac{c_-}{f_C},
\quad y\in[\hat{y}_+(t)-\delta_\eta,\hat{y}_+(t)+\delta_\eta],\:\:t\in[0,T-\eta].
\end{align}
Now, divide both sides of \eqref{lts04} by $\frac{\sigma^2_C}{2}y^2$ to obtain
\begin{align}\label{lts05}
v_{yy}\ge -\Big(\frac{2R_\eta}{\sigma^2_C}\Big)\frac{1}{y^2}-\Big(\frac{2\big|\hat{\mu}_C+\sigma^2_C/2\big|\,L_\eta}{\sigma^2_C}\Big)
\frac{1}{y},
\quad y\in[\hat{y}_+(t)-\delta_\eta,\hat{y}_+(t)+\delta_\eta],\:\:t\in[0,T-\eta],
\end{align}
and recall that $\hat{y}_+(T-\eta)\le\hat{y}_+(t)$ for $t\in[0,T-\eta]$. If we define
\begin{align}\label{lts06}
F(y):=-\int^y_{\hat{y}_+(T-\eta)-\delta_\eta}\int^z_{\hat{y}_+(T-\eta)-\delta_\eta}\Big[
\Big(\frac{2R_\eta}{\sigma^2_C}\Big)\frac{1}{r^2}+\Big(\frac{2\big|\hat{\mu}_C+\sigma^2_C/2\big|\,L_\eta}{\sigma^2_C}\Big)
\frac{1}{r}\Big]dr\,dz,
\end{align}
then $y\mapsto \Lambda(t,y):=[v-F](t,y)$ is convex on $[\hat{y}_+(t),\hat{y}_+(t)+\delta_\eta]$ and on
$[\hat{y}_+(t)-\delta_\eta,\hat{y}_+(t)]$ for all $t\in[0,T-\eta]$. Also, it is easily verified that
$t\mapsto \Lambda_y(t,\hat{y}_\pm(t)\pm)$ is
continuous on $[0,T-\eta]$ by \eqref{lts02} and \eqref{lts06}. From \eqref{lts01} and \eqref{lts06} we obtain that
$\partial_t\Lambda+\mathcal{L}\Lambda-\bar{\mu}\Lambda$ is bounded on any compact $K\subset[0,T-\eta]\times(0,+\infty)$.
It follows that $t\mapsto\Lambda(t,\hat{y}_\pm(t)\pm)$ is of bounded variation on $[0,T-\eta]$, by \cite{Peskir}, Remark 3.2
(see in particular eqs.~(3.35)--(3.36) therein). Therefore \eqref{lts03} holds as $t\mapsto F(\hat{y}_\pm(t)\pm)$
is of bounded variation and hence $v$ has to be such as well.

The local time-space formula may now be employed on $[0,T-\eta]\times(0,+\infty)$. For any $(t,y)\in[0,T-\eta]\times(0,+\infty)$ and arbitrary $s \leq T- \eta - t$, we have
\begin{align}
\label{lts07}
e^{-\bar{\mu}s}v(t\hspace{-1pt}+\hspace{-1pt}s,yC^0(s))\hspace{-1pt}=&\, v(t,y)\hspace{-1pt}+\hspace{-2pt}\int^s_0{\hspace{-5pt}e^{-\bar{\mu}u}\big(\partial_t v+\mathcal{L}v-\bar{\mu}v\big)(t\hspace{-1pt}+\hspace{-1pt}u,yC^0(u))
\mathds{1}_{\{\hat{y}_+(t+u)<yC^0(u)<\hat{y}_-(t+u)\}}du}\nonumber\\
&-\frac{\bar{\mu}}{f_C}\hspace{-2pt}\int^{s}_0{\hspace{-5pt}e^{-\bar{\mu}u}\Big[c_+\mathds{1}_{\{yC^0(u)<\hat{y}_+(t+u)\}}\hspace{-1pt}+\hspace{-1pt}c_-
\mathds{1}_{\{yC^0(u)>\hat{y}_-(t+u)\}}\Big]}du\hspace{-1pt}+\hspace{-1pt}M(s),
\end{align}
by \eqref{smoothpm} and with $M:=\{M(s), s \in [0,T-\eta-t]\}$ a local martingale. We can take
expectations in \eqref{lts07} and use standard localization arguments to cancel the local martingale term. Then,
setting $s=T-\eta-t$ we obtain
\begin{align}\label{lts08}
v(t,y)=&\, \tilde{\mathbb{E}}\bigg\{e^{-\bar{\mu}(T-t-\eta)}v(T-\eta,yC^0(T-t-\eta))\bigg\} \nonumber \\
&+\int^{T-t-\eta}_0
{e^{-\bar{\mu}u}\tilde{\mathbb{E}}\bigg\{
R_c(yC^0(u))\mathds{1}_{\{\hat{y}_+(t+u)<yC^0(u)<\hat{y}_-(t+u)\}}\bigg\}}du\\
&+\frac{\bar{\mu}}{f_C}\int^{T-t-\eta}_0{e^{-\bar{\mu}u}\Big[c_+\tilde{\mathbb{P}}\Big(yC^0(u)<\hat{y}_+(t+u)\Big)+c_-
\tilde{\mathbb{P}}\Big(yC^0(u)>\hat{y}_-(t+u)\Big)\Big]du} \nonumber
\end{align}
by \eqref{freeb-pr} and after rearranging terms. Since \eqref{lts08} holds for any $\eta>0$, in the limit as
$\eta\downarrow0$ we find \eqref{vrepr} by dominated convergence and continuity of $v$.
If we now take $y=\hat{y}_+(t)$ (or $y=\hat{y}_-(t)$) in both sides of \eqref{vrepr} we easily obtain \eqref{int-eq-} 
by recalling that $v(t,\hat{y}_\pm(t))=c_\pm/f_C$.
\end{proof}
\vspace{+3cm}
\begin{center}
FIGURE 1
\end{center}
\begin{figure}[h!]
\centering
\caption{A computer drawing of the free-boundaries obtained by numerical solution of \eqref{int-eq-} with $R_c(y)=y^{-\frac{1}{2}}$, $\bar{\mu}=0.8$, $\mu_C=0.2$, $\sigma_C=1$, $f_C=1$, $c_+=1$, $c_-=0.8$ and $T=1$. The lower line represents $\hat{y}_+$ and the upper line represents $\hat{y}_-$.}\label{figura}
\end{figure}

It is now natural to ask whether the couple $(\hat{y}_+,\hat{y}_-)$ is the unique solution of problem \eqref{int-eq-}--\eqref{int-eq-bd2}. In many optimal stopping problems it is possible to show that the free-boundary is in fact the unique solution of a (single) non-linear integral equation of Volterra type similar to \eqref{int-eq-} (see for instance \cite{PeskShir}, Chapter VII, Section $25$). The proof crucially relies on the characterization of the value function of a sup (inf) problem as the smallest (largest) super-harmonic (sub-harmonic) function lying above (below) the obstacle.
In our zero-sum optimal stopping game instead a further complication arises from the fact that $v$ is a saddle point. Assuming that $(\alpha_+,\alpha_-)$ is another solution of \eqref{int-eq-}--\eqref{int-eq-bd2} and trying to argue as in \cite{PeskShir}, Theorem $25.3$, we define a function $u_\alpha:[0,T]\times (0,\infty) \mapsto \mathbb{R}$ by
\begin{align}\label{mm01}
u_\alpha(t,y):=& \,e^{-\bar{\mu}(T-t)}\frac{c_-}{f_C}+\int^{T-t}_0{e^{-\bar{\mu}s}\tilde{\mathbb{E}}\bigg\{
R_c(yC^0(s))\mathds{1}_{\{\alpha_+(t+s)<yC^0(s)<\alpha_-(t+s)\}}\bigg\}}ds\nonumber\\
&+\frac{\bar{\mu}}{f_C}\int^{T-t}_0{e^{-\bar{\mu}s}\Big[c_+\tilde{\mathbb{P}}\Big(yC^0(s)<\alpha_+(t+s)\Big)+c_-
\tilde{\mathbb{P}}\Big(yC^0(s)>\alpha_-(t+s)\Big)\Big]ds}.
\end{align}
It seems rather hard to prove that $u_\alpha$ of \eqref{mm01} is either larger or smaller than $v$. However, this issue may be overcome by adapting arguments from \cite{YYZ12} 
(cf.~in particular Lemmas 6.3 and 6.4) to show that $c_-/f_C\le u_\alpha\le c_+/f_C$.

\begin{lemma}\label{cci0}
Assume $(\alpha_+,\alpha_-)$ is another solution of \eqref{int-eq-}--\eqref{int-eq-bd2} and let $u_\alpha$ be as is \eqref{mm01}. Then, for any $t\in[0,T)$ one has that $u_\alpha(t,y)=c_-/f_C$ for $y\ge\alpha_-(t)$ and $u_\alpha(t,y)=c_+/f_C$ for $y\le\alpha_+(t)$.
\end{lemma}
\begin{proof}
Set $Y^y(s):=yC^0(s)$, under $\tilde{\mathbb{P}}$ to simplify notation. The map $(t,y)\mapsto u_\alpha(t,y)$ is continuous and $u_\alpha(t,\alpha_\pm(t))=c_\pm/f_C$ for $t\in[0,T)$, by \eqref{int-eq-}.
It is not hard to verify that the process $U^{t,y}_{\alpha}:=\{U^{t,y}_{\alpha}(s), s \in [0,T-t]\}$, defined by
\begin{align}\label{mm02}
U^{t,y}_{\alpha}(s):=e^{-\bar{\mu}s}u_\alpha(t+s,Y^y(s))&+\int^s_0{e^{-\bar{\mu}u}R_c(Y^y(u))\mathds{1}_{\{\alpha_+(t+u)<Y^y(u)<
\alpha_-(t+u)\}}du}\\
&+\frac{\bar{\mu}}{f_C}\int^s_0{e^{-\bar{\mu}u}\big[c_+\mathds{1}_{\{Y^y(u)<\alpha_+(t+u)\}}+c_-\mathds{1}_{\{Y^y(u)>
\alpha_-(t+u)\}}\big]du}\nonumber
\end{align}
is a $\tilde{\mathbb{P}}$-martingale.

Consider $y<\alpha_+(t)$ for a given $t\in[0,T)$, define the stopping time
$$\tau_{\alpha_+}(t,y):=\inf\big\{s\in[0,T-t)\,:\,Y^y(s)\ge \alpha_{+}(t+s)\big\}\wedge(T-t)$$
and as usual set $\tau_{\alpha_+}:=\tau_{\alpha_+}(t,y)$, to simplify notation. From the martingale property of $U^{t,y}_\alpha$ we obtain
\begin{align}\label{mm07}
u_\alpha(t,y)=\tilde{\mathbb{E}}\bigg\{e^{-\bar{\mu}\tau_{\alpha_+}}u_\alpha(t+\tau_{\alpha_+},Y^y(\tau_{\alpha_+}))
+\bar{\mu}\frac{c_+}{f_C}\int^{\tau_{\alpha_+}}_0{e^{-\bar{\mu}s}ds}\bigg\}.
\end{align}
Note that on the set $\{\tau_{\alpha_+}<T-t\}$ one has $Y^y(\tau_{\alpha_+})=\alpha_+(t+\tau_{\alpha_+})$, by continuity of $Y^y$ and $\alpha_+$. On the other hand, $\{\tau_{\alpha_+}=T-t\}\subset\{Y^y(T-t)=0\}$, since $\alpha_+$ is continuous and $\alpha_+(T)=0$; however, $\{Y^y(T-t)=0\}$ is a $\tilde{\mathbb{P}}$-null set and hence we conclude that $u_\alpha(t+\tau_{\alpha_+},Y^y(\tau_{\alpha_+}))=c_+/f_C$, $\tilde{\mathbb{P}}$-a.s. Then, from \eqref{mm07} we obtain
\begin{align}\label{mm08}
u_\alpha(t,y)=\tilde{\mathbb{E}}\bigg\{e^{-\bar{\mu}\tau_{\alpha_+}}\frac{c_+}{f_C}
+\bar{\mu}\frac{c_+}{f_C}\int^{\tau_{\alpha_+}}_0{e^{-\bar{\mu}s}ds}\bigg\}=\frac{c_+}{f_C} \quad\textrm{for all $y<\alpha_+(t)$ and $t\in[0,T)$}.
\end{align}
Similar arguments lead to $u_\alpha(t,y)=c_-/f_C$ for $y>\alpha_-(t)$ and $t\in[0,T)$.
\end{proof}

\begin{lemma}\label{cciuno}
Assume $(\alpha_+,\alpha_-)$ is another solution of \eqref{int-eq-}--\eqref{int-eq-bd2} and let $u_\alpha$ be as is \eqref{mm01}. Then, for any fixed $t\in[0,T)$ the map $y\mapsto u_\alpha(t,y)$ is $C^1$ on $(0,+\infty)$.
\end{lemma}
\begin{proof}
From Lemma \ref{cci0} we know already that $y\mapsto u_\alpha(t,y)$ is $C^1$ on $(0,\alpha_+(t)]\cup[\alpha_-(t),\infty)$, therefore it remains to prove continuity across the two curves $(\alpha_+,\alpha_-)$.

Recalling \eqref{mm01}, the function $u_\alpha$ may be written as
\begin{align}\label{cciloro01}
u_\alpha(t,y):=&\,e^{-\bar{\mu}(T-t)}\frac{c_-}{f_C}+\int^{T-t}_0{e^{-\bar{\mu}s}K_1
\big(y;s,\alpha_+(t+s),\alpha_-(t+s)\big)ds}
\nonumber\\
&+\frac{\bar{\mu}}{f_C}\int^{T-t}_0{e^{-\bar{\mu}s}\Big[c_+\,K_2\big(y;s,\alpha_+(t+s)\big)+c_-\,
K_3\big(y;s,\alpha_-(t+s)\big)\Big]ds},
\end{align}
with
\begin{align}
&K_1\big(y;s,\alpha,\beta\big):=\int^{\beta}_{\alpha}{R_c(z)p_C(y,s;z)dz},\quad K_2\big(y;s,\alpha\big):=\int^{\alpha}_0{p_C(y,s;z)dz},\label{cciloro02-1}\\
&K_3\big(y;s,\beta\big):=\int^{+\infty}_\beta{\hspace{-6pt}p_C(y,s;z)dz}\quad\text{and}\quad p_C(y,s;z):=\frac{\exp\big\{-\frac{1}{2\sigma^2_C\,s}\big[\ln(z/y)-\hat{\mu}_C\,s\big]^2\big\}}{\sqrt{2\pi\,s}\,\sigma_C\,z}.\label{cciloro02-3}
\end{align}
For simplicity we denote by $K_i(y;s,\alpha,\beta)$, $i=1,2,3$, expressions in \eqref{cciloro02-1} and \eqref{cciloro02-3}.

Fix $t_o\in[0,T)$ and $\delta>0$ such that $2\delta\le T-t_o$, then $\alpha_+(t_o)>\epsilon(t_o)=:\epsilon_o>0$ and $\alpha_+(T-\delta)>\epsilon^\prime(\delta)=:\epsilon^\prime_\delta>0$ by \eqref{int-eq-bd2}. Hence, it is not hard to verify that $(y,s)\mapsto K_i(t_o,y;s,\alpha_+(t_o+s),\alpha_-(t_o+s))$ and $(y,s)\mapsto \frac{\partial K_i}{\partial y}(t_o,y;s,\alpha_+(t_o+s),\alpha_-(t_o+s))$, $i=1,2,3$, are continuous and bounded on $(y,s)\in[\epsilon_o/2,\kappa]\times[\delta,T-t_o-\delta]$ for arbitrary $\kappa>\alpha_-(0)$. It follows that
\begin{align}\label{cciloro03}
\frac{\partial}{\partial y}\int^{T-t_o-\delta}_\delta \hspace{-8pt}e^{-\bar{\mu}s}K_i
\big(y;s,\alpha_+(t_o+s),\alpha_-(t_o+s)\big)ds=\hspace{-4pt}\int^{T-t_o-\delta}_\delta{\hspace{-8pt}e^{-\bar{\mu}s}\frac{\partial K_i}{\partial y}
\big(y;s,\alpha_+(t_o+s),\alpha_-(t_o+s)\big)ds}
\end{align}
for $y\in[\epsilon_o/2,\kappa]$ and $i=1,2,3$.
In the proof of Lemma \ref{stimaconcavitaR} we provide estimates for $K_1$ and simple bounds for $K_2$, $K_3$ that imply
\begin{align}\label{cciloro04}
\hspace{-8pt}\int^{T-t_o-\delta}_\delta \hspace{-10pt}e^{-\bar{\mu}s}K_i
\big(y;s,\alpha_+(t_o+s),\alpha_-(t_o+s)\big)ds
\rightarrow\int^{T-t_o}_0 \hspace{-10pt}e^{-\bar{\mu}s}K_i
\big(y;s,\alpha_+(t_o+s),\alpha_-(t_o+s)\big)ds
\end{align}
as $\delta\downarrow 0$ uniformly for $y\in[\epsilon_o/2,\kappa]$, with $i=1,2,3$. Also, it is shown in Section \ref{cciappendix} that
\begin{align}\label{cciloro05}
\hspace{-10pt}\int^{T-t_o-\delta}_\delta \hspace{-12pt}e^{-\bar{\mu}s}\frac{\partial K_i}{\partial y}
\big(y;s,\alpha_+(t_o\hspace{-2pt}+\hspace{-2pt}s),\alpha_-(t_o\hspace{-2pt}+\hspace{-2pt}s)\big)ds\rightarrow\hspace{-5pt}\int^{T-t_o}_0 \hspace{-10pt}e^{-\bar{\mu}s}\frac{\partial K_i}{\partial y}
\big(y;s,\alpha_+(t_o\hspace{-2pt}+\hspace{-2pt}s),\alpha_-(t_o\hspace{-2pt}+\hspace{-2pt}s)\big)ds
\end{align}
as $\delta\downarrow 0$ uniformly for $y\in[\epsilon_o/2,\kappa]$, with $i=1,2,3$ as well. Therefore, it follows that
\begin{align}\label{cciloro06}
\hspace{-10pt}\frac{\partial}{\partial y}\int^{T-t_o}_0 \hspace{-12pt}e^{-\bar{\mu}s}K_i
\big(y;s,\alpha_+(t_o\hspace{-2pt}+\hspace{-2pt}s),\alpha_-(t_o\hspace{-2pt}+\hspace{-2pt}s)\big)ds=\hspace{-4pt}\int^{T-t_o}_0{\hspace{-12pt}e^{-\bar{\mu}s}\frac{\partial K_i}{\partial y}
\big(y;s,\alpha_+(t_o\hspace{-2pt}+\hspace{-2pt}s),\alpha_-(t_o\hspace{-2pt}+\hspace{-2pt}s)\big)ds}
\end{align}
for $y\in[\epsilon_o/2,\kappa]$, with $i=1,2,3$ and then $y\mapsto u_\alpha(t_o,y)$ is $C^1$ on $y\in[\epsilon_o/2,\kappa]$, i.e.~across $\alpha_+(t_o)$ and $\alpha_-(t_o)$. Since $t_o$ and $\kappa$ are arbitrary, $y\mapsto u_\alpha(t,y)$ is $C^1$ on $(0,+\infty)$, for all $t<T$ by Lemma \ref{cci0}.
\end{proof}
The martingale property \eqref{mm02}, Lemma \ref{cciuno} and standard arguments imply that $u_\alpha$ solves a free-boundary problem as \eqref{freeb-pr} with $\hat{y}_+$ and $\hat{y}_-$ replaced by $\alpha_+$ and $\alpha_-$, respectively. We define a set $\mathcal{C}_{\alpha}:=\big\{(t,y)\,:\,\alpha_+(t)<y<\alpha_-(t),\,t\in(0,T)\big\}$ and notice that $\frac{\partial^2 u_\alpha}{\partial t\partial y}$ and $\frac{\partial^3 u_\alpha}{\partial y^3}$ exist and are continuous in $\mathcal{C}_\alpha$ (cf.~\cite{FriedmanPDE}, Chapter 3, Theorem 10). Recall the operator $\mathcal{G}$ of \eqref{A-op} and set $\bar{u}:=\frac{\partial u_\alpha}{\partial y}$, then $\bar{u}\in C^{1,2}(\mathcal{C}_\alpha)$ and solves
\begin{align}
&\bar{u}_t(t,y)+\mathcal{G}\bar{u}(t,y)=-R_{cc}(y)\qquad\textrm{in $\mathcal{C}_\alpha$}\label{cciloro07-1}
\end{align}
with $\bar{u}(t,\alpha_\pm(t))=0$ and $\bar{u}(T,y)=0$ by Lemma \ref{cciuno} and \eqref{mm01}. We find now useful bounds on $u_\alpha$ by using properties of $\bar{u}$.
\begin{proposition}\label{obst}
One has
$c_-/f_C\le u_\alpha(t,y)\le c_+/f_C$ for all $(t,y)\in[0,T]\times(0,+\infty)$.
\end{proposition}
\begin{proof}
The result is obvious in $(0,\alpha_+(t)]\cup[\alpha_-(t),+\infty)$, $t\in[0,T]$ by Lemma \ref{cci0}; it remains to show that the same holds in $\mathcal{C}_\alpha$. It is sufficient to prove that $y\mapsto u_\alpha(t,y)$ is decreasing for all $t\in(0,T)$. In order to do so, it is useful to introduce functions $\hat{u}(t,y):=e^{(\hat{\mu}_C+\sigma^2_C/2-\bar{\mu})t}\bar{u}(t,y)$ and $\hat{R}_{cc}(t,y):=e^{(\hat{\mu}_C+\sigma^2_C/2-\bar{\mu})t}R_{cc}(y)$. Then recalling \eqref{A-op} and \eqref{cciloro07-1} it follows that
\begin{align}\label{cciloro07-2}
\hat{u}_t(t,y)+\frac{1}{2}\sigma^2_Cy^2\hat{u}_{yy}(t,y)+\big(2\sigma^2_C-{\mu}_C\big)y\hat{u}_y(t,y)=
-\hat{R}_{cc}(t,y)\quad\textrm{in $\mathcal{C}_\alpha$}
\end{align}
with $\hat{u}(t,\alpha_\pm(t))=0$ and $\hat{u}(T,y)=0$.

Assume that there exists $\epsilon>0$ and $(t_\epsilon,y_\epsilon)\in\mathcal{C}_\alpha$ such that $\hat{u}(t_\epsilon,y_\epsilon)>\epsilon$. We may define a set $\Omega_\epsilon:=\big\{(t,y)\,:\,\hat{u}(t,y)>\epsilon/2\big\}$ and observe that $(t_\epsilon,y_\epsilon)\in\Omega_\epsilon$, that $\Omega_\epsilon\subset\mathcal{C}_\alpha$ and that the set $\overline{\Omega}_\epsilon\setminus\overline{\Omega}_\epsilon\cap\mathcal{C}_\alpha$ consists at most of the points $(T,\alpha_\pm(T))$ by Lemma \ref{cciuno}. Moreover, $\Omega_\epsilon$ itself has positive $dt\otimes dy$ measure since $\hat{u}$ is continuous in $\mathcal{C}_\alpha$. On the other hand, the second order differential operator in \eqref{cciloro07-2} may be associated to a diffusion $X^{y_\epsilon}:=\big\{X^{y_\epsilon}(s),\,s\in[0,T-t_\epsilon]\big\}$ that solves
\begin{align}\label{cciloro08}
dX^{y_\epsilon}(s)=\big(2\sigma^2_C-{\mu}_C\big)X^{y_\epsilon}(s)ds+\sigma_CX^{y_\epsilon}d\tilde{W}(s)\qquad\text{for $s > 0$ and $X^{y_\epsilon}(0)=y_\epsilon$.}
\end{align}
We now set $\rho_{\epsilon}:=\inf\big\{s\in[0,T-t_\epsilon]\,:\,X^{y_\epsilon}(s)\notin \Omega_\epsilon\big\}$ and note that $\rho_\epsilon>0$ $\tilde{\mathbb{P}}$-a.s. Recall that $R_{cc}<0$ and use \eqref{cciloro07-2} and Dynkin's formula to obtain
$\hat{u}(t_\epsilon,y_\epsilon)<\tilde{\mathbb{E}}\big\{\hat{u}
(t_\epsilon+\rho_\epsilon, X^{y_\epsilon}(\rho_\epsilon))\big\}\le \epsilon/2$
which contradicts the definition of $(t_\epsilon,y_\epsilon)$. Therefore, since $\epsilon>0$ is arbitrary it follows $\hat{u}\le 0$ and hence $\bar{u}\le 0$ and $y\mapsto u_\alpha(t,y)$ is decreasing.
\end{proof}
\begin{theorem}
\label{esistenzaVolterra}
The couple $(\hat{y}_+(t),\hat{y}_-(t))$ is the unique solution of \eqref{int-eq-} in the class of continuous, decreasing functions such that \eqref{int-eq-bd}, \eqref{int-eq-bd2} hold.
\end{theorem}

\begin{proof}
Set again $Y^y(s):=yC^0(s)$, under $\tilde{\mathbb{P}}$ to simplify notation. Assume there exist two continuous functions $\alpha_-$ and $\alpha_+$ solving \eqref{int-eq-}--\eqref{int-eq-bd2} and take $u_\alpha$ as in \eqref{mm01}.

We shall now prove that $\alpha_+ \equiv \hat{y}_{+}$ and $\alpha_{-}\equiv \hat{y}_{-}$.
Initially we show that
\begin{align}\label{up}
\alpha_+(t)\le\hat{y}_+(t)\qquad\textrm{and}\qquad \alpha_-(t)\ge\hat{y}_-(t)\qquad\textrm{for all $t\in[0,T)$}.
\end{align}
Full details are only provided for the first of (\ref{up}) as the ones for the second can be obtained analogously.
Assume that there exists $t_o\in[0,T)$ such that $\hat{y}_+(t_o)<\alpha_+(t_o)$. Then, take $y_o\in(\hat{y}_+(t_o),\alpha_+(t_o))$ and define the stopping time
\begin{align}\label{mm10}
\rho_{\alpha_-}(t_o,y_o):=\inf\{s\in[0,T-t_o)\,:\,Y^{y_o}(s)\ge \alpha_-(t_o+s)\}\wedge(T-t_o).
\end{align}
Let $\sigma^*(t_o,y_o)$ be as in \eqref{stoppingtimesv} (or equivalently as in \eqref{stoppingtimesfrontiere}) and set $\rho_{\alpha_-}:=\rho_{\alpha_-}(t_o,y_o)$ and $\sigma^*:=\sigma^*(t_o,y_o)$ for simplicity. From the martingale property of $U^{t_o,y_o}_{\alpha}$ in \eqref{mm02} we obtain
\begin{align}\label{mm11}
u_\alpha(t_o,y_o)=\tilde{\mathbb{E}}\bigg\{&e^{-\bar{\mu}\sigma^*\wedge\rho_{\alpha_-}}u_\alpha
(t_o+\sigma^*\wedge\rho_{\alpha_-},Y^{y_o}(\sigma^*\wedge\rho_{\alpha_-}))\\
&+\int^{\sigma^*\wedge\rho_{\alpha_-}}_0
{e^{-\bar{\mu}s}\Big[R_c(Y^{y_o}(s))\mathds{1}_{\{Y^{y_o}(s)>\alpha_+(t_o+s)\}}+\bar{\mu}\frac{c_+}{f_C}
\mathds{1}_{\{Y^{y_o}(s)
<\alpha_+(t_o+s)\}}\Big]ds}\bigg\}.\nonumber
\end{align}
The first term in the expectation of \eqref{mm11} is such that
\begin{align}\label{mm13}
\hspace{-8pt}u_\alpha
(t\hspace{-2pt}+\hspace{-2pt}\sigma^*\hspace{-2pt}\wedge\hspace{-2pt}\rho_{\alpha_-},Y^{y_o}(\sigma^*\hspace{-2pt}\wedge\hspace{-2pt}\rho_{\alpha_-}))\hspace{-2pt}\le\hspace{-2pt}\frac{c_-}{f_C}
\mathds{1}_{\{\rho_{\alpha_-}<\sigma^*\}}\hspace{-2pt}+\hspace{-2pt}\frac{c_+}{f_C}
\mathds{1}_{\{\sigma^*\le\rho_{\alpha_-} \}}\mathds{1}_{\{\sigma^*<T-t\}}\hspace{-2pt}+\hspace{-2pt}\frac{c_-}{f_C}
\mathds{1}_{\{\rho_{\alpha_-}=\sigma^*=T-t\}}
\end{align}
by \eqref{mm01} and Proposition \ref{obst}.
Observe that all (continuous) sample paths starting from $y_o$ spend a strictly positive amount of time under the curve $\{\alpha_+(t_o+s),\,s\in[0,T-t_o)\}$ by continuity of $\alpha_+$. Moreover, from (\ref{int-eq-bd2}) we have
\begin{align}\label{mm12}
\textrm{$\bar{\mu}\frac{c_+}{f_C}<R_c(Y^{y_o}(s))$ on the set $\{Y^{y_o}(s)<\alpha_+(t_o+s).\}$}
\end{align}
Recall \eqref{Jey} and note that $\sigma^*\wedge\rho_{\alpha_-}>0$, $\tilde{\mathbb{P}}$-a.s., by continuity of $t\mapsto Y^{y_o}(t)$. Then, using \eqref{mm13} and \eqref{mm12} inside \eqref{mm11}, we find
$u_\alpha(t_o,y_o)<\Psi(t_o,y_o;\sigma^*,\rho_{\alpha_-})$ (cf.~\eqref{Jey}).
It follows that $u_\alpha(t_o,y_o) < v(t_o,y_o)$. However, $u_\alpha(t_o,y)=c_+/f_C$ for all $y\in(0, \alpha_+(t))$ by \eqref{mm08} and hence $v(t_o,y_o) > c_+/f_C$. This is a contradiction as $(t_o,y_o) \in \mathcal{C}$. Similarly, one can find analogous contradiction by assuming that there exists $t_o \in [0,T)$ such that $\alpha_-(t_o) < \hat{y}_{-}(t_o)$.

We show now that it must in fact be $\alpha_+ \equiv \hat{y}_+$ and $\alpha_- \equiv \hat{y}_-$. Again, we provide full details only for $\alpha_+$ as the other case follows by straightforward modifications. Assume that there exists $t_o\in[0,T)$ such that $\alpha_+(t_o)<\hat{y}_+(t_o)$. Take $y_o\in(\alpha_+(t_o),\hat{y}_+(t_o))$, set $\tau^*(t_o,y_o)$ as in \eqref{stoppingtimesv} and 
\begin{align}\label{mm15}
\rho_{\alpha_+}(t_o,y_o):=\inf\{s\in[0,T-t_o)\,:\,Y^{y_o}(s)\le \alpha_+(t_o+s)\}\wedge(T-t_o).
\end{align}
Denote $\tau^* := \tau^*(t_o,y_o)$ and $\rho_{\alpha_+}:=\rho_{\alpha_+}(t_o,y_o)$ for simplicity. We now set $s\hspace{-1pt}:=\hspace{-1pt}\tau^*\hspace{-1pt}\wedge\hspace{-1pt}\rho_{\alpha_+}\hspace{-1pt}\wedge(T\hspace{-1pt}-\hspace{-1pt}t_o\hspace{-1pt}-\hspace{-1pt}\eta)$ in \eqref{lts07}, take expectation on both sides and pass to the limit as $\eta\to0$ to obtain
\begin{align}\label{mm16}
v(t_o,y_o)=&\tilde{\mathbb{E}}\Big\{e^{-\bar{\mu}\tau^*\wedge\rho_{\alpha_+}}v(t_o+\tau^*\wedge\rho_{\alpha_+},
Y^{y_o}(\tau^*
\wedge\rho_{\alpha_+}))
\nonumber\\
&+\int^{\tau^*\wedge\rho_{\alpha_+}}_0{e^{-\bar{\mu}s}\big[R_c(Y^{y_o}(s))\mathds{1}_{\{\hat{y}_+(t_o+s)<Y^{y_o}(s)\}}+
\frac{\bar{\mu}c_+}{f_C}\mathds{1}_{\{Y^{y_o}(s)<\hat{y}_+(t_o+s)\}}
\big]ds}\Big\}.
\end{align}
Since $\alpha_+(t)\le\hat{y}_+(t)$ for $t\in[0,T)$ (cf.\ (\ref{up})), it is not hard to see that $v(t_o+\rho_{\alpha_+},Y^{y_o}(\rho_{\alpha_+}))=\frac{c_+}{f_C}$ on ${\{\rho_{\alpha_+}\le\tau^*\}}\cap{\{\rho_{\alpha_+}<T-t_o\}}$. Again we notice that $\tau^*\wedge\rho_{\alpha_+}>0$, $\tilde{\mathbb{P}}$-a.s., by continuity of the sample paths of $Y^{y_o}$ and that from (\ref{int-eq-bd2})
\begin{align}\label{mm17}
\textrm{$\bar{\mu}\frac{c_+}{f_C}<R_c(Y^{y_o}(s))$ on the set $\{Y^{y_o}(s)<\hat{y}_+(t_o+s)\}$}.
\end{align}
Since all sample paths starting from $y_o$ spend a strictly positive amount of time below $\{\hat{y}_+(t_o+s),\,s\in[0,T-t_o)\}$ by continuity of $\hat{y}_+$, we obtain
$v(t_o,y_o)<\Psi(t_o,y_o;\rho_{\alpha_+},\tau^*)$
by (\ref{mm17}) (cf.~\eqref{Jey}).
On the other hand, recalling Proposition \ref{obst}, \eqref{up} and using the martingale property of $U^{t_o,y_o}_{\alpha}$ as in \eqref{mm11} we obtain
$u_\alpha(t_o,y_o)\ge\Psi(t_o,y_o;\rho_{\alpha_+},\tau^*)$
so that 
$u_\alpha(t_o,y_o)> v(t_o,y_o)$. However, $v(t_o,y_o)=c_+/f_C$ for $y_o\in(\alpha_+(t_o),\hat{y}_+(t_o))$ by \eqref{freeb-pr} and hence $u_\alpha(t_o,y_o)> c_+/f_C$, contradicting Proposition \ref{obst}. Therefore $\alpha_+\equiv\hat{y}_+$ and by obvious extensions of arguments above one also finds $\alpha_-\equiv\hat{y}_-$.
\end{proof}

\section{The Optimal Control Strategy}
\label{optimalcontrolstrategy}

In Theorem \ref{existence} we proved existence and uniqueness of the optimal control process $\nu^{*}$, but we provided no information about its nature. In this Section we characterize the optimal control in terms of the two free-boundaries $\hat{y}_{+}$ and $\hat{y}_{-}$ (cf.\ Proposition \ref{proprietacontinuazione}) of the zero-sum optimal stopping game (\ref{defv2}).
We shall see that the optimal investment-disinvestment strategy for problem (\ref{optimalproblem}) consists in keeping the optimally controlled diffusion $C^{y,\nu^{*}}$ inside the closure of the continuation region, with the optimal controls behaving as the local times of $C^{y,\nu^{*}}$ at $\hat{y}_{+}$ and $\hat{y}_{-}$. To accomplish that we will rely on results in \cite{Burdzy} on the pathwise construction of a process in a space-time region defined by two moving boundaries.
Recall (\ref{capacity}) and (\ref{nubarradefinizione}) and introduce the following
\begin{problem}
\label{Skorohodproblemdef}
Let $t \in [0,T]$ and $y > 0$ be arbitrary but fixed. Given the two free-boundaries $\hat{y}_+$ and
$\hat{y}_{-}$ of Proposition \ref{proprietacontinuazione}, respectively, we seek a left-continuous adapted process
$C^{y,\overline{\nu}^{*}}$ and a process of bounded variation $\overline{\nu}^{*}=\overline{\nu}_+^{*}
- \overline{\nu}_{-}^{*} \in \mathcal{S}^y_{t,T}$ such that
\begin{equation}
\label{Skorohodproblem}
\left\{
\begin{array}{ll}
\displaystyle C^{y,\overline{\nu}^{*}}(0) = y,\quad C^{y,\overline{\nu}^{*}}(s)= C^0(s)[y + \overline{\nu}^{*}_{+}(s) - \overline{\nu}^{*}_{-}(s)],
 \quad s \in [0,T-t),\\ \\
\displaystyle \hat{y}_{+}(t+s) \leq C^{y,\overline{\nu}^{*}}(s) \leq \hat{y}_{-}(t+s),\quad \text{a.e.} \,\, s \in [0,T-t], \\ \\
\displaystyle \int_{0}^{T-t}\mathds{1}_{\{C^{y,\overline{\nu}^{*}}(s) < \hat{y}_{-}(t+s)\}}d\overline{\nu}_{-}^{*}(s)
 = 0, \quad\text{and} \quad\displaystyle \int_{0}^{T-t}\mathds{1}_{\{C^{y,\overline{\nu}^{*}}(s) > \hat{y}_{+}(t+s)\}}d\overline{\nu}_{+}^{*}(s)
 = 0 \\ \\
\end{array}
\right.
\end{equation}
hold $\tilde{\mathbb{P}}$-a.s.
Moreover, if $y \in [\hat{y}_{+}(t), \hat{y}_{-}(t)]$ then $\overline{\nu}_+^{*}(\omega,\cdot)$ and
$\overline{\nu}_{-}^{*}(\omega,\cdot)$ are continuous. When $y < \hat{y}_{+}(t)$, then
 $\overline{\nu}_+^{*}(\omega, 0+)= \hat{y}_{+}(t) - y$, $\overline{\nu}_{-}^{*}(\omega, 0+)=0$ and
  $C^{y,\overline{\nu}^{*}}(\omega,0+)=\hat{y}_{+}(t)$; when $y > \hat{y}_{-}(t)$, then
  $\overline{\nu}_{-}^{*}(\omega, 0+)= y - \hat{y}_{-}(t)$, $\overline{\nu}_{+}^{*}(\omega, 0+)=0$ and
  $C^{y,\overline{\nu}^{*}}(\omega,0+)=\hat{y}_{-}(t)$.
\end{problem}

\begin{proposition}
There exists a unique solution of Problem \ref{Skorohodproblemdef} given by
\begin{equation}
\label{explicitformoptimalcontrol}
\left\{
\begin{array}{ll}
\displaystyle C^{y,\overline{\nu}^{*}}(s) = C^0(s)[y + \overline{\nu}^{*}(s)], \\ \\
\displaystyle \overline{\nu}^{*}(s+) = - \max\Big\{\Big[ \big(y - \hat{y}_{-}(t)\big)^{+} \wedge \inf_{u \in
[0,s]}\left(\frac{yC^0(u) - \hat{y}_{+}(t+u)}{C^0(u)}\right)\Big], \vspace{0.25cm} \\
\displaystyle \hspace{2cm} \sup_{r \in [0,s]}\Big[\left(\frac{yC^0(r) - \hat{y}_{-}(t+r)}
{C^0(r)}\right) \wedge \inf_{u \in [r,s]}\left(\frac{yC^0(u) - \hat{y}_{+}(t+u)}{C^0(u)}\right)\Big]
\Big\},
\end{array}
\right.
\end{equation}
for every $s \in [0,T-t)$.
\end{proposition}
\begin{proof}
Take $t \in [0,T]$ and $s \in [0,T-t]$ arbitrary but fixed and set
\begin{equation*}
\left\{
\begin{array}{ll}
\displaystyle \phi(s):=\frac{C^{y,\overline{\nu}^{*}}(s +)}{C^0(s)}, \quad \psi(s):= y, \quad \ell(s):=\frac{\hat{y}_{+}(t+s)}{C^0(s)}, \quad r(s):=\frac{\hat{y}_{-}(t + s)}{C^0(s)},\\ \\
\displaystyle \eta(s)=\eta_{\ell}(s) - \eta_{r}(s):= \overline{\nu}^{*}_{+}(s +) - \overline{\nu}^{*}_{-}(s +). \\ \\
\end{array}
\right.
\end{equation*}
Notice that $\inf_{s \in [0,T-t]}[r(s) - \ell(s)]>0$, by Proposition \ref{freeboundaries}. Hence, we can apply \cite{Burdzy}, Corollary $2.4$ and Theorem $2.6$ to obtain existence and uniqueness of the solution of Problem \ref{Skorohodproblemdef}. Moreover, equations $(2.6)$ and $(2.7)$ in \cite{Burdzy}, give (\ref{explicitformoptimalcontrol}) by \eqref{capacitysolution} above.
\end{proof}

In order to prove that $C^{y,\overline{\nu}^{*}}$ is optimal for the control problem (\ref{optimalproblem}) it is useful to observe that $V_t, V_y, V_{yy}$ belong to $L^{\infty}((0,T) \times (0,K))$, for arbitrary $K > 0$, by Proposition \ref{2valuefunctions}, (\ref{freeb-pr}) and Proposition \ref{propsmoothfit} ($V_y$ and $V_{yy}$ are in fact continuous). Therefore, the value function $V$ of \eqref{optimalproblem} is a weak solution of the HJB equation
\begin{align*}
\min\big\{-R+\mu_F V-\mathcal{L}V-V_t\,,\,c_+/f_C-V_y\,,\,V_y-c_-/f_C\big\}=0
\end{align*}
for $(t,y)\in[0,T]\times(0,\infty)$, with $\mathcal{L}$ as in \eqref{freeb-pr}, and
$V(T,y)=G_0+\frac{c_-}{f_C}y$, for $y\in(0,\infty)$ (see also \cite{PhamGuo}, eq.\ (3.6) for a similar framework). Now, recalling that $c_-/f_C<V_y<c_+/f_C$ inside the continuation region $\mathcal{C}$, we can apply It\^o's formula for semi-martingales (cf.\ \cite{Protter}, Theorem $32$, p.\ $79$, among others) in the generalized sense of \cite{BensoussanLions}, Lemma $8.1$ and Theorem $8.5$, pp.\ 183--186, to obtain a verification theorem.
\begin{theorem}
Let $(C^{y,\overline{\nu}^{*}},\overline{\nu}^{*})$ denote the unique solution of Problem \ref{Skorohodproblemdef}. Then $C^{y,\nu^{*}}$ is the optimally controlled production capacity for problem (\ref{optimalproblem}) with $\nu^{*}:=\nu^{*}_+ - \nu^{*}_{-}$ and
\begin{equation*}
\nu^{*}_{+}(s):=\int_{0}^{s}\frac{C^0(u)}{f_C}\,d\overline{\nu}^{*}_{+}(u), \qquad  \nu^{*}_{-}(s):=\int_{0}^s\frac{C^0(u)}{f_C}\,d\overline{\nu}^{*}_{-}(u),\qquad\text{for $s \in [0,T-t)$.}
\end{equation*}
\end{theorem}
As expected (cf.\ \cite{KaratzasWang}, Theorem $3.1$), the optimal time to invest (disinvest) coincides with the first time at which the uncontrolled diffusion hits the moving boundary $\hat{y}_{+}$ ($\hat{y}_{-}$).


\section{Some Proofs for Section \ref{zerosumoptimalstoppinggame}}
\label{appproofs}
\subsection{Proof of Theorem \ref{jointcontinuityv}}
\label{Appcontinuity}

In this Section we show that the value function of the zero-sum optimal stopping game (\ref{defv2}) is continuous on $[0,T] \times (0,\infty)$. First we prove preliminary results and we introduce some new definitions that will be useful in the rest of this Section.
Recall that
\beq
\label{stimaattesoCzero}
\tilde{\mathbb{E}}\left\{\left(\frac{1}{C^0(s)}\right)^{\alpha}\right\} = e^{(\alpha\mu_C - \frac{1}{2}\alpha^2\sigma^2_C)s},\quad\text{for any $\alpha \geq 1$ and $s \in [0,T]$.}
\eeq

\begin{lemma}
\label{stimaconcavitaR}
Under Assumption \ref{AssProfit}, for any $\alpha \geq 1$ one has
\beq\label{basic01}
\tilde{\mathbb{E}}\bigg\{\int_0^{T} R^{\alpha}_c(yC^0(s))ds\bigg\} \leq \kappa\left(1 + \frac{1}{y^{\alpha}}\right),
\eeq
where $\kappa>0$ is a suitable constant independent of $y$.
\end{lemma}
\begin{proof}
Since $R(0)=0$ (cf. Assumption \ref{AssProfit}), for any $y>0$ we have $R_c(y) \leq y^{-1}R(y)$, by concavity of $R$. Also, Inada conditions imply that there exist $\kappa_1>0$ and $\kappa_2 > 0$ such that $R(y) \leq \kappa_1 + \kappa_2 y$ for all $y \in (0,\infty)$.
Hence we have
\begin{eqnarray}
\tilde{\mathbb{E}}\bigg\{\int_0^{T} R^{\alpha}_c(yC^0(s))ds\bigg\}  \leq  \tilde{\mathbb{E}}\bigg\{\int_0^{T} \left(\frac{1}{yC^0(s)}\right)^{\alpha}[\kappa_1 + \kappa_2 yC^0(s)]^{\alpha}\,ds\bigg\} \nonumber
\end{eqnarray}
and \eqref{basic01} easily follows from \eqref{stimaattesoCzero}.
\end{proof}

From now on and throughout this Section, we will define $Y^y(s):=yC^0(s)$ (cf.\ (\ref{nubarradefinizione})) under the measure $\tilde{\mathbb{P}}$; also we denote by $\mathcal{L}$ the infinitesimal generator associated to $Y$ as in \eqref{freeb-pr}.
Inspired by Stroock and Varadhan \cite{StrVar} we adopt the following
\begin{definition}
\label{martingalesense}
Take measurable functions $h:[0,T] \times (0,\infty) \mapsto \mathbb{R}$ and $u: [0,T] \times (0,\infty) \mapsto \mathbb{R}$ such that
$$\tilde{\mathbb{E}}\bigg\{\int_0^s e^{-\bar{\mu}r}|h(t+r,Y^y(r))|\,dr\bigg\} < \infty, \quad \tilde{\mathbb{E}}\Big\{e^{-\bar{\mu}s}|u(t+s,Y^y(s))|\Big\} < \infty, \quad s\geq 0,$$
for any $(t,y) \in [0,T] \times (0,\infty)$ arbitrary but fixed. We say that $u$ solves
$$\left(\partial_t + \mathcal{L} - \bar{\mu}\right)u(t,y) = h(t,y),\qquad (t,y) \in [0,T] \times (0,\infty),$$
in the martingale sense if and only if the process
\begin{align}\label{mart01}
M^{t,y}:=\bigg\{e^{-\bar{\mu}s}u(t+s, Y^y(s)) - \int_0^s e^{-\bar{\mu}r}h(t+r,Y^y(r))dr,\,\,s\geq0\bigg\}
\end{align}
is a $\tilde{\mathbb{P}}$-martingale.
\end{definition}
\begin{remark}\label{remark-mart}
For any adapted, bounded process $Z:=\{Z(s), s\ge0\}$, if $u$ and $h$ are as in Definition \ref{martingalesense} and $M^{t,y}$ of \eqref{mart01} is a $\tilde{\mathbb{P}}$-martingale, then the process
\begin{align*}
N^{t,y}:=\bigg\{e^{-\bar{\mu}s-\int_0^s{Z(t+r)dr}}&u(t+s, Y^y(s))\\
 &\hspace{-10pt}- \int_0^s e^{-\bar{\mu}r-\int_0^r{Z(t+v)dv}}\Big[h(t+r,Y^y(r))+Z(t+r)u(t+r,Y^y(r))\Big]dr,\,\,s\geq0\bigg\}
\end{align*}
is a $\tilde{\mathbb{P}}$-martingale as well (cf.~for instance \cite{Menaldi}, Remark 1.3).
\end{remark}
Denote by $C_b^{\infty}$ the space of functions which are differentiable infinitely many times and which are bounded with all their derivatives. In order to set our problem in a suitable space we define a real valued function $w$ by
\begin{align}\label{dabbliu}
w(y):=\frac{y}{1+y}\qquad y\ge0.
\end{align}
This is a positive, increasing, $C^\infty_b$-function on $[0,+\infty)$ and it is not hard to see that
\begin{align}\label{dabbliu-1}
\tilde{\mathbb{E}}\Big\{\int^T_0{e^{-\rho s}\frac{1}{w\big(Y^y(s)\big)}ds}\Big\}<\frac{1}{\rho}+\frac{1}{y}\left[\frac{1}{\rho+\mu_F+ \frac{1}{2}\sigma^2_C-\bar{\mu}}\right]
\end{align}
for any $\rho>0$ and $\rho \neq \bar{\mu} - \mu_F - \frac{1}{2}\sigma^2_C$, by \eqref{stimaattesoCzero}.
\begin{definition}
\label{spazioBanach}
For $w$ as in \eqref{dabbliu} we write
\beq
\label{wnorm}
||f||_{w,\infty}:= \sup_{(t,y) \in [0,T] \times [0,\infty)}|w(y)f(t,y)|
\eeq
and define
\beq
\label{newspaceBanach}
\mathcal{C}_b^{w}([0,T] \times [0,\infty)):=\{f: f \in C([0,T] \times (0,\infty)) \,\mbox{ and }\, ||f||_{w,\infty} < \infty\}.
\eeq
\end{definition}
\noindent It easily follows that $||\cdot||_{w,\infty}$ is a norm and that $\mathcal{C}_b^{w}([0,T] \times [0,\infty))$ is a Banach space. Now we study a penalized problem.

\begin{proposition}
\label{uepsiloncontinuabounded}
For any given $\epsilon > 0$ there exists a unique $u^\epsilon\in\mathcal{C}^w_b([0,T]\times[0,\infty))$ that solves
\beq
\label{penalized1}
\displaystyle \left(\partial_t + \mathcal{L} - \bar{\mu}\right)u^{\epsilon}(t,y) = -R_c(y) - \frac{1}{\epsilon}\left(\frac{c_{-}}{f_C} - u^{\epsilon}(t,y)\right)^{+} + \frac{1}{\epsilon}\left(u^{\epsilon}(t,y) - \frac{c_{+}}{f_C}\right)^{+}
\eeq
in the martingale sense of Definition \ref{martingalesense} with $u^{\epsilon}(T,y)=c_{-}/f_C$.
\end{proposition}

\begin{proof}
Fix $\epsilon > 0$ and note that
\begin{align*}
- \frac{1}{\epsilon}\left(\frac{c_{-}}{f_C} - u^{\epsilon}\right)^{+} = \frac{1}{\epsilon}u^{\epsilon} - \frac{1}{\epsilon}\left(\frac{c_{-}}{f_C} \vee u^{\epsilon}\right)\quad\text{and}\quad
\frac{1}{\epsilon}\left(u^{\epsilon} - \frac{c_{+}}{f_C}\right)^{+} = \frac{1}{\epsilon}u^{\epsilon} - \frac{1}{\epsilon}\left(\frac{c_{+}}{f_C} \wedge u^{\epsilon}\right).
\end{align*}
From Remark \ref{remark-mart}, with $u:=u^\epsilon$, $h:=-R_c - \frac{1}{\epsilon}(\frac{c_{-}}{f_C} - u^{\epsilon})^{+} + \frac{1}{\epsilon}(u^{\epsilon} - \frac{c_{+}}{f_C})^{+}$ and $Z=\frac{1}{\epsilon}$, it follows that \eqref{penalized1} may be rewritten as
\beq
\label{penalized2}
\left\{
\begin{array}{ll}
\displaystyle \left(\partial_t + \mathcal{L} - \left(\bar{\mu}+\frac{2}{\epsilon}\right)\right)u^{\epsilon}(t,y) = -R_c(y) - \frac{1}{\epsilon}\left(\frac{c_{-}}{f_C} \vee u^{\epsilon}(t,y)\right) - \frac{1}{\epsilon}\left(u^{\epsilon}(t,y) \wedge \frac{c_{+}}{f_C}\right) \\ \\
\displaystyle u^{\epsilon}(T,y)=\frac{c_{-}}{f_C},
\end{array}
\right.
\eeq
and the solution of \eqref{penalized2} in the martingale sense (cf.\ Definition \ref{martingalesense}), if it exists, is given by
\begin{eqnarray}
\label{solpenalmgsense}
u^{\epsilon}(t,y) \hspace{-0.25cm}& = &\hspace{-0.25cm} \tilde{\mathbb{E}}\bigg\{\frac{c_{-}}{f_C}e^{-\bar{\mu}(T-t)} + \int_0^{T-t} e^{-(\bar{\mu} + \frac{2}{\epsilon})s}\Big[ R_c(Y^y(s)) + \frac{1}{\epsilon}\left(\frac{c_{-}}{f_C} \vee u^{\epsilon}(t+s,Y^y(s))\right) \nonumber \\
& & \hspace{2cm} + \frac{1}{\epsilon}\left(\frac{c_{+}}{f_C} \wedge u^{\epsilon}(t+s,Y^y(s))\right)\,ds\bigg\}.
\end{eqnarray}
We now show that (\ref{solpenalmgsense}) admits a unique solution in $\mathcal{C}_b^w([0,T]\times [0, \infty))$ by a fixed point argument.
For $g \in \mathcal{C}_b^w([0,T]\times [0, \infty))$ we define the operator $\cal{T}^{\epsilon}$ by
\begin{eqnarray}
\label{calTepsilon}
(\cal{T}^{\epsilon}g)(t,y) \hspace{-0.25cm}& = &\hspace{-0.25cm} \tilde{\mathbb{E}}\bigg\{\frac{c_{-}}{f_C}e^{-\bar{\mu}(T-t)} + \int_0^{T-t} e^{-(\bar{\mu} + \frac{2}{\epsilon})s} \Big[ R_c(Y^y(s)) + \frac{1}{\epsilon}\left(\frac{c_{-}}{f_C} \vee g(t+s,Y^y(s))\right) \nonumber \\
& & \hspace{2cm} + \frac{1}{\epsilon}\left(\frac{c_{+}}{f_C} \wedge g(t+s,Y^y(s))\right)\Big]\,ds\bigg\}
\end{eqnarray}
that maps $\mathcal{C}_b^w([0,T]\times [0, \infty))$ into itself. In order to prove that $(t,y)\mapsto\cal{T}^{\epsilon}g(t,y)$ is indeed continuous, take $(t_1,y_1)$ and $(t_2,y_2)$ in $[0,T] \times (0, \infty)$ (without loss of generality we may take, $t_2 > t_1$ and $y_2 > y_1 > \delta$ for some $\delta>0$) and notice that
\begin{eqnarray*}
|(\cal{T}^{\epsilon}g)(t_1,y_1) - (\cal{T}^{\epsilon}g)(t_2,y_2)| \hspace{-0.25cm} & \leq & \hspace{-0.25cm} |(\cal{T}^{\epsilon}g)(t_1,y_1) - (\cal{T}^{\epsilon}g)(t_2,y_1)| + |(\cal{T}^{\epsilon}g)(t_2,y_1) - (\cal{T}^{\epsilon}g)(t_2,y_2)| \nonumber \\
\hspace{-0.25cm} & =: & \hspace{-0.25cm} (I) + (II).
\end{eqnarray*}
Then, for $(I)$ we have
\begin{align}
\label{stimaI1}
(I) \leq &\, \Big|\frac{c_{-}}{f_C}\left(e^{-\bar{\mu}(T-t_1)} - e^{-\bar{\mu}(T-t_2)}\right)\Big|\nonumber\\
&+\,\Big|\tilde{\mathbb{E}}\bigg\{\int_0^{T-t_1} e^{-(\bar{\mu} + \frac{2}{\epsilon})s} R_c(Y^{y_1}(s))\,ds - \int_0^{T-t_2} e^{-(\bar{\mu} + \frac{2}{\epsilon})s} R_c(Y^{y_1}(s))\,ds\bigg\}\Big| \nonumber \\
& +\, \Big|\tilde{\mathbb{E}}\bigg\{\int_0^{T-t_1} e^{-(\bar{\mu} + \frac{2}{\epsilon})s} \frac{1}{\epsilon}\left(\frac{c_{-}}{f_C} \vee g(t_1+s,Y^{y_1}(s))\right)\,ds \nonumber \\
& \hspace{+3cm} - \int_0^{T-t_2} e^{-(\bar{\mu} + \frac{2}{\epsilon})s} \frac{1}{\epsilon}\left(\frac{c_{-}}{f_C} \vee g(t_2+s,Y^{y_1}(s))\right)\,ds\bigg\}\Big| \\
& +\, \Big|\tilde{\mathbb{E}}\bigg\{\int_0^{T-t_1} e^{-(\bar{\mu} + \frac{2}{\epsilon})s} \frac{1}{\epsilon}\left(\frac{c_{+}}{f_C} \wedge g(t_1+s,Y^{y_1}(s))\right)\,ds \nonumber \\
&\hspace{+3cm} - \int_0^{T-t_2} e^{-(\bar{\mu} + \frac{2}{\epsilon})s} \frac{1}{\epsilon}\left(\frac{c_{+}}{f_C} \wedge g(t_2+s,Y^{y_1}(s))\right)\,ds\bigg\}\Big|. \nonumber
\end{align}
The second term on the right-hand side of (\ref{stimaI1}) converges to zero as $|t_2-t_1|\to0$, by dominated convergence and Lemma \ref{stimaconcavitaR}.

We only analyze the third term on the right-hand side of \eqref{stimaI1} as the same arguments apply to the fourth one. Observe that
\begin{align}
\label{stimaterzopezzoI}
\bigg|\int_0^{T-t_1} &\hspace{-0.3cm}e^{-(\bar{\mu} + \frac{2}{\epsilon})s} \frac{1}{\epsilon}\left(\frac{c_{-}}{f_C} \vee g(t_1+s,Y^{y_1}(s))\right)\,ds-\int_0^{T-t_2}\hspace{-0.3cm} e^{-(\bar{\mu} + \frac{2}{\epsilon})s} \frac{1}{\epsilon}\left(\frac{c_{-}}{f_C} \vee g(t_2+s,Y^{y_1}(s))\right)\,ds\bigg| \nonumber \\
\leq&\, \int_0^T e^{-(\bar{\mu} + \frac{2}{\epsilon})s} \frac{w(Y^{y_1}(s))}{w(Y^{y_1}(s))}\Big[\Big|\frac{c_{-}}{f_C} \vee g(t_1+s,Y^{y_1}(s))\Big|+ \Big|\frac{c_{-}}{f_C} \vee g(t_2+s,Y^{y_1}(s))\Big|\Big]\,ds \nonumber \\
\leq &\,2\,||\frac{c_{-}}{f_C} \vee g||_{w,\infty} \int_0^T e^{-(\bar{\mu} + \frac{2}{\epsilon})s} \frac{1}{w(Y^{y_1}(s))}\,ds.
\end{align}
Using \eqref{dabbliu-1} with $\rho=\bar{\mu}+2/\epsilon$ and recalling that $y_2>y_1>\delta$, one may easily verify that the last expression in \eqref{stimaterzopezzoI} is independent of $t_1$, $t_2$, $y_1$, $y_2$ and it is $\tilde{\mathbb{P}}$-integrable. Therefore, from (\ref{stimaterzopezzoI}) and dominated convergence $\lim_{t_1 \rightarrow t_2}|(\cal{T}^{\epsilon}g)(t_1,y_1) - (\cal{T}^{\epsilon}g)(t_2,y_1)| =0$. One can show that $(II)$ goes to zero as $|y_2-y_1|\to0$ by similar arguments.
Since the lower bound $\delta$ on $y_1$ and $y_2$ is arbitrary, we conclude that
$(t,y) \mapsto (\cal{T}^{\epsilon}g)(t,y) \in C([0,T] \times (0,\infty))$ for all $g \in \mathcal{C}_b^w([0,T] \times [0,\infty))$.

Our next step is proving that $||\cal{T}^{\epsilon}g||_{w,\infty} < \infty.$ Consider again $g \in \mathcal{C}_b^w([0,T] \times [0,\infty))$ and notice that
\begin{align}
\label{boundTepsilong}
|w(y)(\cal{T}^{\epsilon}g)(t,y)|\leq &\, w(y)\,\tilde{\mathbb{E}}\bigg\{\int_0^{T-t} e^{-(\bar{\mu} + \frac{2}{\epsilon})s} R_c(Y^{y}(s))\,ds\bigg\} + w(y)\frac{c_{-}}{f_C} \\
&+\,\frac{1}{\epsilon}w(y)\tilde{\mathbb{E}}\bigg\{\hspace{-0.1cm}\int_0^{T-t}\hspace{-0.3cm} e^{-(\bar{\mu} + \frac{2}{\epsilon})s}\frac{w(Y^y(s))}{w(Y^y(s))} \Big[\Big( \frac{c_{-}}{f_C} \vee g\Big)+ \Big( \frac{c_{+}}{f_C} \wedge g\Big)\Big](t+s,Y^{y}(s))\,ds\bigg\} \nonumber \\
\leq &\, \kappa w(y)\Big(1 + \frac{1}{y}\Big) + w(y)\frac{c_{-}}{f_C} \nonumber\\
&+\, \frac{w(y)}{\epsilon}\,\tilde{\mathbb{E}}\bigg\{\int_0^T e^{-(\bar{\mu} + \frac{2}{\epsilon})s} \frac{ds}{w(Y^y(s))}\bigg\}\Big[ ||\frac{c_{-}}{f_C} \vee g||_{w,\infty} + ||\frac{c_{-}}{f_C} \wedge g||_{w,\infty}\Big], \nonumber
\end{align}
where we have used Lemma \ref{stimaconcavitaR} to find the first term in the last expression above and the same arguments as in (\ref{stimaterzopezzoI}) for the third one. Finally, recalling \eqref{dabbliu-1} and taking the supremum for $(t,y) \in [0,T] \times [0,\infty)$ we conclude that $\big\|\cal{T}^{\epsilon}g\big\|_{w,\infty}<\infty$.

To complete the proof we have now to show that $\cal{T}^{\epsilon}$ is a contraction. Take $g_1,g_2 \in \mathcal{C}_b^w([0,T]\times [0, \infty))$. Then, arguments as those employed to obtain (\ref{boundTepsilong}) and \eqref{dabbliu-1} with $\rho=\bar{\mu}+2/\epsilon$ give
\begin{align}
\label{calTepsiloncontaction}
|w(y)(\cal{T}^{\epsilon}g_1 - \cal{T}^{\epsilon}g_2)(t,y)| \le \, w(y)\,||g_1 - g_2||_{w,\infty}\left[\frac{1}{y}\left(\frac{1}{1 + (\mu_F + \frac{1}{2}\sigma^2_C)\frac{\epsilon}{2}}\right) + \frac{1}{1 + \frac{\bar{\mu}\epsilon}{2}}\right]. \nonumber
\end{align}
Set $c_1:=1/(1 + (\mu_F + \frac{1}{2}\sigma^2_C)\frac{\epsilon}{2})$ and $c_2:=1/(1 + \frac{\bar{\mu}\epsilon}{2})$. Then, $w(y)\Big[\frac{c_1}{y}+ c_2\Big]\le c_1\vee c_2<1$ and $\cal{T}^{\epsilon}$ is a contraction. Hence, there exists a unique solution of the penalized problem (\ref{penalized1}) in $\mathcal{C}_b^w([0,T]\times [0, \infty))$, by Banach fixed point theorem.
\end{proof}

From Definition \ref{martingalesense} and Proposition \ref{uepsiloncontinuabounded} it follows
\begin{corollary}\label{continuousmg}
For any $(t,y)\in[0,T]\times(0,\infty)$ the process $H^{t,y}:=\{H^{t,y}(s),\,\,s \geq 0\}$ defined by
\begin{align}
H^{t,y}(s):=&e^{-\bar{\mu}s}u^{\epsilon}(t+s, Y^y(s))\\
& + \int_0^s e^{-\bar{\mu}r} \Big[ R_c(Y^y(r))+ \frac{1}{\epsilon}\Big(\frac{c_{-}}{f_C} - u^{\epsilon}(t+r,Y^y(r))\Big)^+ - \frac{1}{\epsilon}\Big(u^{\epsilon}(t+r,Y^y(r))-\frac{c_{+}}{f_C}\Big)^+\Big]\,dr\nonumber
\end{align}
is a continuous $\tilde{\mathbb{P}}$-martingale.
\end{corollary}

\begin{proposition}
\label{uespsiloncontrolspropo}
Define
$\mathcal{A}:=\{\nu: \Omega \times [0,T] \mapsto [0,1], \text{adapted\,}\}.$
Then, the solution of the penalized problem (\ref{penalized1}) may be written as
\beq
\label{uepsiloncontrols}
u^{\epsilon}(t,y)=\sup_{\nu_1 \in \mathcal{A}}\inf_{\nu_2 \in \mathcal{A}}\Xi^{\epsilon}(t,y;\nu_1,\nu_2) = \inf_{\nu_2 \in \mathcal{A}}\sup_{\nu_1 \in \mathcal{A}}\Xi^{\epsilon}(t,y;\nu_1,\nu_2),
\eeq
where
\begin{eqnarray}
\label{csi}
\Xi^{\epsilon}(t,y;\nu_1,\nu_2) \hspace{-0.25cm}& := &\hspace{-0.25cm} \tilde{\mathbb{E}}\bigg\{\int_0^{T-t} e^{-\bar{\mu}r - \frac{1}{\epsilon}\int_0^{r}(\nu_1(\alpha) + \nu_2(\alpha))d\alpha}\Big[R_c(Y^y(r)) + \frac{1}{\epsilon}\nu_1(r)\frac{c_{-}}{f_C} + \frac{1}{\epsilon}\nu_2(r)\frac{c_{+}}{f_C}\Big]\,dr \nonumber \\
& & \hspace{1cm} + \frac{c_{-}}{f_C}e^{-\bar{\mu}(T-t) - \frac{1}{\epsilon}\int_0^{T-t}(\nu_1(\alpha) + \nu_2(\alpha))d\alpha}\bigg\}.
\end{eqnarray}
\end{proposition}
\begin{proof}
For any $\nu_1, \nu_2 \in \mathcal{A}$ and $s\le T-t$, we may write
\begin{eqnarray}
\label{alls}
u^{\epsilon}(t,y) & \hspace{-0.25cm} = \hspace{-0.25cm} &\tilde{\mathbb{E}}\bigg\{e^{-\bar{\mu}s - \frac{1}{\epsilon}\int_0^{s}(\nu_1(\alpha) + \nu_2(\alpha))d\alpha} u^{\epsilon}(t+s,Y^y(s)) \\
&\hspace{-0.25cm} + \hspace{-0.25cm}& \int_0^s e^{-\bar{\mu}r - \frac{1}{\epsilon}\int_0^{r}(\nu_1(\alpha) + \nu_2(\alpha))d\alpha}\Big[R_c(Y^y(r)) + \frac{1}{\epsilon}\left(\frac{c_{-}}{f_C} - u^{\epsilon}(t+r,Y^y(r))\right)^{+} \nonumber \\
&\hspace{-0.25cm} - \hspace{-0.25cm}& \frac{1}{\epsilon}\left(u^{\epsilon}(t+r,Y^y(r)) - \frac{c_{+}}{f_C}\right)^{+} + \frac{1}{\epsilon}(\nu_1(r) + \nu_2(r))u^{\epsilon}(t+r, Y^y(r))\Big]\,dr\bigg\},\nonumber
\end{eqnarray}
by Corollary \ref{continuousmg} and Remark \ref{remark-mart}, with $u:=u^\epsilon$, $h:=-R_c-\frac{1}{\epsilon}\big(\frac{c_-}{f_C}-u^\epsilon\big)^++\frac{1}{\epsilon}\big(u^\epsilon-\frac{c_+}{f_C}\big)^+$ and $Z(s):=\frac{1}{\epsilon}\big[\nu_1(s)+\nu_2(s)\big]$. Taking now $s=T-t$ in \eqref{alls},
\beq
\begin{array}{cc}
\nu^{*}_1 :=
\left\{
\begin{array}{ll}
1 \quad \text{on }\,\,\{u^{\epsilon} \leq \frac{c_{-}}{f_C}\} \\ \\
0 \quad \text{on }\,\,\{u^{\epsilon} > \frac{c_{-}}{f_C}\}
\end{array}
\right.
\hspace{+10pt}&\text{and}\hspace{+20pt}
\nu^{*}_2 :=
\left\{
\begin{array}{ll}
1 \quad \text{on }\,\,\{u^{\epsilon} \geq \frac{c_+}{f_C}\} \\ \\
0 \quad \text{on }\,\,\{u^{\epsilon} < \frac{c_+}{f_C}\}
\end{array}
\right.
\end{array}
\eeq
and following \cite{Menaldi}, Section 2.1, we easily find
\begin{align*}
u^{\epsilon}(t,y) \leq \Xi^{\epsilon}(t,y;\nu^{*}_1,\nu_2) \quad\text{and}\quad u^{\epsilon}(t,y) \geq \Xi^{\epsilon}(t,y;\nu_1,\nu^{*}_2) \quad \textrm{for all $\nu_1,\nu_2 \in \mathcal{A}$}.
\end{align*}
Then \eqref{uepsiloncontrols} follows.
\end{proof}

Since $\Xi^{\epsilon}(t,y;\nu_1,\nu_2) \geq 0$ for all $\nu_1, \nu_2 \in \mathcal{A}$ (cf.\ \eqref{csi}), then $u^\epsilon(t,y)\geq 0$ for all $(t,y)\in[0,T]\times(0,\infty)$ by \eqref{uepsiloncontrols}.

\begin{proposition}
\label{convergenzapartipositive}
One has
\begin{align}\label{convpp}
\lim_{\epsilon \downarrow 0}\Big|\Big|\left(u^{\epsilon} - \frac{c_+}{f_C}\right)^{+}\Big|\Big|_{w,\infty} =0\quad\text{and}\quad
\lim_{\epsilon \downarrow 0}\Big|\Big|\left(\frac{c_-}{f_C} - u^{\epsilon}\right)^{+}\Big|\Big|_{w,\infty} =0.
\end{align}
\end{proposition}
\begin{proof}
For any $\nu_1,\nu_2 \in \mathcal{A}$ we may write
\begin{eqnarray}
\label{integrazioniparte}
\frac{c_+}{f_C} &\hspace{-0.25cm} = \hspace{-0.25cm}& \frac{c_+}{f_C}e^{-\bar{\mu}(T-t) - \frac{1}{\epsilon}\int_0^{T-t}(\nu_1(\alpha) + \nu_2(\alpha))d\alpha} \nonumber \\
& \hspace{-0.25cm} + \hspace{-0.25cm} & \int_0^{T-t}e^{-\bar{\mu}r - \frac{1}{\epsilon}\int_0^{r}(\nu_1(\alpha) + \nu_2(\alpha))d\alpha}\frac{c_+}{f_C}\left(\bar{\mu} + \frac{1}{\epsilon}[\nu_1(r) + \nu_2(r)]\right)\,dr,
\end{eqnarray}
by an integration by parts. Then from (\ref{uepsiloncontrols}) and (\ref{csi}) it follows
\begin{eqnarray}
\label{uepsiloncontrolssimplified}
u^{\epsilon}(t,y) - \frac{c_+}{f_C} &\hspace{-0.25cm} \leq \hspace{-0.25cm}& \inf_{\nu_2 \in \mathcal{A}}\sup_{\nu_1 \in \mathcal{A}}\tilde{\mathbb{E}}\bigg\{\int_0^{T-t} e^{-\bar{\mu}r - \frac{1}{\epsilon}\int_0^{r}(\nu_1(\alpha) + \nu_2(\alpha))d\alpha}\Big[R_c(Y^y(r)) - \bar{\mu}\frac{c_+}{f_C}\Big]\,dr\bigg\} \nonumber \\
&\hspace{-0.25cm} \leq \hspace{-0.25cm}& \tilde{\mathbb{E}}\bigg\{\int_0^{T-t} e^{-\bar{\mu}r - \frac{1}{\epsilon}r} R_c(Y^y(r))\,dr\bigg\} \leq \left[\frac{\epsilon}{2(1 + \bar{\mu}\epsilon)}\right]^{\frac{1}{2}}\left[\kappa\left( 1 + \frac{1}{y^2}\right)\right]^{\frac{1}{2}}, \nonumber
\end{eqnarray}
where the third expression follows by H\"older inequality and Lemma \ref{stimaconcavitaR}. Similarly,
\begin{eqnarray}
\label{uepsiloncontrolssimplified2}
u^{\epsilon}(t,y) - \frac{c_{-}}{f_C} &\hspace{-0.25cm} \geq \hspace{-0.25cm}& \inf_{\nu_2 \in \mathcal{A}}\sup_{\nu_1 \in \mathcal{A}}\tilde{\mathbb{E}}\bigg\{\int_0^{T-t} e^{-\bar{\mu}r - \frac{1}{\epsilon}\int_0^{r}(\nu_1(\alpha) + \nu_2(\alpha))d\alpha}\Big[R_c(Y^y(r)) - \bar{\mu}\frac{c_{-}}{f_C}\Big]\,dr\bigg\} \nonumber \\
&\hspace{-0.25cm} \geq \hspace{-0.25cm}& -\left[\frac{\epsilon}{2(1+\bar{\mu}\epsilon)}\right]^{\frac{1}{2}}\left[\kappa\left(1+\frac{1}{y^2}\right)
\right]^{\frac{1}{2}}-\frac{\bar{\mu}c_-}{f_C}\left[\frac{\epsilon}{1+\bar{\mu}\epsilon}\right].
\end{eqnarray}
Hence \eqref{convpp} follows from Definition \ref{spazioBanach}.
\end{proof}

Before proving Theorem \ref{jointcontinuityv} we shall make further observations on $u^{\epsilon}$.
Take $\sigma$ and $\tau$ arbitrary stopping times in $[0,T-t]$. From Corollary \ref{continuousmg}, with $s$ replaced by $\sigma\wedge\tau$, we find
\begin{align}
\label{uepsilondamgproperty}
u^{\epsilon}(t,y)  = \mathbb{E}&\bigg\{e^{-\bar{\mu}(\tau \wedge \sigma)}u^{\epsilon}(t + \tau \wedge \sigma, Y^y(\tau \wedge \sigma)) \\
& + \hspace{-4pt}\int_{0}^{\tau \wedge \sigma}\hspace{-6pt}e^{-\bar{\mu}r}\Big[R_c(\cdot) + \frac{1}{\epsilon}\left(\frac{c_{-}}{f_C} - u^{\epsilon}(\cdot, \cdot)\right)^{+} - \frac{1}{\epsilon}\left(u^{\epsilon}(\cdot, \cdot) -\frac{c_{+}}{f_C}\right)^{+} \Big](t+r,Y^y(r))\,dr\bigg\}.\nonumber
\end{align}
Now, recalling that $u^\epsilon(T,y)=\frac{c_-}{f_C}$ and noting that $u^{\epsilon} \leq \frac{c_+}{f_C} + (u^{\epsilon} - \frac{c_+}{f_C})^{+}$ and $u^{\epsilon} \geq \frac{c_{-}}{f_C} - (\frac{c_{-}}{f_C} - u^{\epsilon})^{+}$, we have
\begin{eqnarray}\label{uepsmin}
u^{\epsilon}(t,y) & \hspace{-0.25cm} \leq \hspace{-0.25cm}&\mathbb{E}\bigg\{e^{-\bar{\mu}\tau}u^{\epsilon}(t + \tau, Y^y(\tau))\mathds{1}_{\{\tau < \sigma\}} + \frac{c_+}{f_C}e^{-\bar{\mu}\sigma}\mathds{1}_{\{\sigma \leq \tau\}}\mathds{1}_{\{\sigma < T-t\}} \nonumber \\
&\hspace{-0.25cm} + \hspace{-0.25cm}& e^{-\bar{\mu}\sigma}\left(u^{\epsilon}(t+\sigma, Y^y(\sigma))-\frac{c_+}{f_C}\right)^{+}\mathds{1}_{\{\sigma \leq \tau\}}\mathds{1}_{\{\sigma< T-t\}} + \frac{c_{-}}{f_C}e^{-\bar{\mu}(T-t)}\mathds{1}_{\{\tau=\sigma=T-t\}} \nonumber \\
&\hspace{-0.25cm} + \hspace{-0.25cm}& \int_{0}^{\tau \wedge \sigma}e^{-\bar{\mu}s}\Big[R_c(Y^y(r)) + \frac{1}{\epsilon}\left(\frac{c_{-}}{f_C} - u^{\epsilon}(t+r, Y^y(r))\right)^{+} \Big]\,dr\bigg\}
\end{eqnarray}
and
\begin{eqnarray}\label{uepsmaj}
u^{\epsilon}(t,y) &\hspace{-0.25cm} \geq \hspace{-0.25cm}& \mathbb{E}\bigg\{e^{-\bar{\mu}\sigma}u^{\epsilon}(t + \sigma, Y^y(\sigma))\mathds{1}_{\{\sigma \leq \tau\}}\mathds{1}_{\{\sigma < T-t\}} + \frac{c_{-}}{f_C}e^{-\bar{\mu}\tau}\mathds{1}_{\{\tau < \sigma\}}\nonumber \\
&\hspace{-0.25cm} - \hspace{-0.25cm}& e^{-\bar{\mu}\tau}\left(\frac{c_{-}}{f_C} - u^{\epsilon}(t+\tau, Y^y(\tau))\right)^{+}\mathds{1}_{\{\tau < \sigma\}} + \frac{c_{-}}{f_C}e^{-\bar{\mu}(T-t)}\mathds{1}_{\{\tau=\sigma=T-t\}} \nonumber \\
&\hspace{-0.25cm} + \hspace{-0.25cm}& \int_{0}^{\tau \wedge \sigma}e^{-\bar{\mu}r}\Big[R_c(Y^y(r)) - \frac{1}{\epsilon}\left(u^{\epsilon}(t+r, Y^y(r)) - \frac{c_{+}}{f_C}\right)^{+} \Big]\,dr\bigg\}.
\end{eqnarray}

We are now able to prove Theorem \ref{jointcontinuityv}.
\begin{proof}[Proof of Theorem \ref{jointcontinuityv}]
From \eqref{defv2} and \eqref{uepsmin} we find
\begin{eqnarray*}
\label{prima}
& & u^{\epsilon}(t,y) - v(t,y) \leq\inf_{\tau\in [0,T-t]}\sup_{\sigma\in [0,T-t]}\tilde{\mathbb{E}}\bigg\{e^{-\bar{\mu}\tau}\Big(u^\epsilon(t+\tau,Y^y(\tau))-\frac{c_-}{f_C}\Big)\mathds{1}_{\{\tau<\sigma\}}\nonumber \\
& & +\,e^{-\bar{\mu}\sigma}\Big(u^\epsilon(t+\sigma,Y^y(\sigma))-\frac{c_+}{f_C}\Big)^+\mathds{1}_{\{\sigma\le\tau\}}\mathds{1}_{\{\sigma<T-t\}}
+\frac{1}{\epsilon}\int^{\sigma\wedge\tau}_0{e^{-\bar{\mu}r}\Big(\frac{c_-}{f_C}-u^\epsilon(t+r,Y^y(r))\Big)^+dr}\bigg\}. \nonumber
\end{eqnarray*}
Take $\tau=\tau^\epsilon:=\inf\big\{s\in[0,T-t)\,:\,u^\epsilon(t+s,Y^y(s))\le c_-/f_C\big\}\wedge(T-t)$ in the equation above to obtain
\begin{eqnarray*}
\label{prima2}
& & u^{\epsilon}(t,y) - v(t,y) \leq \sup_{\sigma\in [0,T-t]}\tilde{\mathbb{E}}\bigg\{e^{-\bar{\mu}\sigma}\Big(u^\epsilon(t+\sigma,Y^y(\sigma))-\frac{c_+}{f_C}\Big)^+\bigg\}\nonumber\\
& & \le \sup_{\sigma\in [0,T-t]}\tilde{\mathbb{E}}\bigg\{\frac{1}{w\big(Y^y(\sigma)\big)}\bigg\}\Big\|\Big(u^\epsilon-\frac{c_+}{f_C}\Big)^+\Big\|_{w,\infty}
\le \left[1+\frac{1}{y}\tilde{\mathbb{E}}\bigg\{\sup_{0\le s\le T-t}\frac{1}{C^0(s)}\bigg\}\right]\Big\|\Big(u^\epsilon-\frac{c_+}{f_C}\Big)^+\Big\|_{w,\infty}. \nonumber
\end{eqnarray*}
Arguing in a similar way and using \eqref{uepsmaj} we also obtain
\begin{align}\label{seconda2}
\hspace{-20pt}
u^{\epsilon}(t,y) - v(t,y) \geq&- \left[1+\frac{1}{y}\tilde{\mathbb{E}}\bigg\{\sup_{0\le s\le T-t}\frac{1}{C^0(s)}\bigg\}\right]\Big\|\Big(\frac{c_-}{f_C}-u^\epsilon\Big)^+\Big\|_{w,\infty}.
\end{align}
Therefore (cf.\ Definition \ref{spazioBanach})
\begin{align}\label{terza}
\big\|u^\epsilon-v\big\|_{w,\infty}\le\kappa\Big(\Big\|\Big(u^\epsilon-\frac{c_+}{f_C}\Big)^+\Big\|_{w,\infty}+\Big\|\Big(\frac{c_-}{f_C}-u^\epsilon\Big)^+\Big\|_{w,\infty}\Big)
\end{align}
for a suitable constant $\kappa>0$ depending only on $\hat{\sigma}_C$, $\hat{\mu}_C$ and $T>0$. Now, the right-hand side of \eqref{terza} goes to zero as $\epsilon\to0$ and $w\,v\in C([0,T]\times[0,\infty))$, thus implying $v\in C([0,T]\times(0,\infty))$.
\end{proof}
\begin{remark}\label{rem-unifconv}
Note that for any $\delta>0$, one has $\|u^\epsilon-v\|_{w,\infty}\ge \delta/(1+\delta)\sup_{[0,T]\times[\delta,\infty)}|u^\epsilon-v|(t,y)$ and hence $u^\epsilon\to v$ uniformly on $[0,T]\times[\delta,\infty)$ as $\epsilon\to0$.
\end{remark}

\subsection{Proof of Theorem \ref{thmsaddlepointsv}}
\label{saddlepointsvproof}

For $\epsilon>0$ set
\beq
\label{approxstopping}
\left\{
\begin{array}{ll}
\tau^{\epsilon}(t,y):=\inf\{s \in [0,T-t): u^{\epsilon}(t+s, Y^y(s)) \leq \frac{c_{-}}{f_C} \} \wedge (T-t), \\ \\
\sigma^{\epsilon}(t,y):=\inf\{s \in [0,T-t): u^{\epsilon}(t+s, Y^y(s)) \geq \frac{c_{+}}{f_C} \} \wedge (T-t).
\end{array}
\right.
\eeq
Take $\delta>0$ arbitrary but fixed and define the first exit time of $Y$ from the half-plane $(\delta,\infty)$ by $\tau_\delta(y):=\inf\{s\ge0:\, Y^y(s)\le \delta\}$. Note that for all $y>0$, one finds
\begin{align}\label{tdelta}
\tau_\delta(y)\uparrow\infty\quad\textrm{as $\delta\downarrow 0$, $\tilde{\mathbb{P}}$-a.s.}
\end{align}
as $\{0\}$ is a non-attainable boundary point for the process $Y$. For simplicity we set $\tau^{\epsilon}\equiv\tau^{\epsilon}(t,y)$, $\sigma^{\epsilon}\equiv\sigma^{\epsilon}(t,y)$ and $\tau_\delta\equiv\tau_\delta(y)$.

From Remark \ref{rem-unifconv} $u^\epsilon\to v$ uniformly on $[0,T]\times[\delta,\infty)$ as $\epsilon\downarrow 0$. Then, following the same arguments as in the proof of \cite{CDA}, Lemma $6.2$, we find that
$$\lim_{\epsilon\to0}\,\tau^*\wedge\tau^\epsilon\wedge\tau_\delta=\tau^*\wedge\tau_\delta \quad \mbox{ and } \quad \lim_{\epsilon\to0}\,\sigma^*\wedge\sigma^\epsilon\wedge\tau_\delta=\sigma^*\wedge\tau_\delta,\quad\text{$\tilde{\mathbb{P}}$-a.s.}$$
for all $(t,y)\in[0,T]\times (0,\infty)$ and with $\tau^*$ and $\sigma^*$ as in \eqref{stoppingtimesv}. Therefore, we also have
\begin{align}\label{lim-st}
\lim_{\epsilon\to\infty}\,\sigma^{*}\wedge\sigma^\epsilon\wedge \tau^{*}\wedge\tau^\epsilon\wedge\tau_\delta
=\sigma^{*}\wedge \tau^{*}\wedge\tau_\delta\qquad
\textrm{$\tilde{\mathbb{P}}$-a.s.,}
\end{align}
for all $(t,y)\in[0,T]\times (0,\infty)$.

Again, to simplify notation we set $\rho_{\delta,\epsilon}:=\sigma^{*}\wedge\sigma^\epsilon\wedge \tau^{*} \wedge\tau^\epsilon \wedge \tau_\delta$ and we obtain
\begin{align}\label{ueps01}
u^{\epsilon}(t,y) = \tilde{\mathbb{E}}\bigg\{e^{-\bar{\mu}\rho_{\delta,\epsilon}}u^{\epsilon}(t + \rho_{\delta,\epsilon}, Y^y(\rho_{\delta,\epsilon})) + \int_{0}^{\rho_{\delta,\epsilon}}e^{-\bar{\mu}s} R_c(Y^y(s))\,ds\bigg\},
\end{align}
by (\ref{uepsilondamgproperty}).
Taking limits in \eqref{ueps01} first as $\epsilon\to 0$ and then as $\delta\to0$, the left-hand side converges to $v$ by uniform convergence. For the right-hand side we employ dominated convergence, Remark \ref{rem-unifconv}, \eqref{lim-st} and continuity of $v$ when taking $\epsilon\to0$; whence, when $\delta\to0$ we employ monotone convergence and \eqref{tdelta} for the integral term, and dominated convergence, \eqref{tdelta} and continuity of $v$ for the other one. We thus obtain
$v(t,y) = \Psi(t,y;\sigma^*,\tau^*)$ (cf.~\eqref{Jey}).

Note that
\begin{align}\label{v00}
&e^{-\bar{\mu}\sigma^{*}\wedge \tau^{*}}v(t + \sigma^{*}\wedge \tau^{*}, Y^y(\sigma^{*}\wedge \tau^{*}))\nonumber\\
&=e^{-\bar{\mu}\tau^*}\frac{c_{-}}{f_C}\mathds{1}_{\{\tau^* < \sigma^*\}} + e^{-\bar{\mu}\sigma^*}\frac{c_{+}}{f_C}\mathds{1}_{\{\sigma^* \leq \tau^*\}}\mathds{1}_{\{\sigma^* < T-t\}} + e^{-\bar{\mu}(T-t)}\frac{c_{-}}{f_C}\mathds{1}_{\{\sigma^* = \tau^* = T-t\}}\qquad\textrm{$\tilde{\mathbb{P}}$-a.s.}
\end{align}
and therefore
\begin{align}\label{v01}
v(t,y) = \tilde{\mathbb{E}}&\bigg\{e^{-\bar{\mu}\tau^*}\frac{c_{-}}{f_C}\mathds{1}_{\{\tau^* < \sigma^*\}} + e^{-\bar{\mu}\sigma^*}\frac{c_{+}}{f_C}\mathds{1}_{\{\sigma^* \leq \tau^*\}}\mathds{1}_{\{\sigma^* < T-t\}} + e^{-\bar{\mu}(T-t)}\frac{c_{-}}{f_C}\mathds{1}_{\{\sigma^* = \tau^* = T-t\}} \nonumber \\
&\hspace{5pt} + \int_{0}^{\tau^* \wedge \sigma^*}e^{-\bar{\mu}s} R_c(Y^y(r))\,dr\bigg\}.
\end{align}

It remains now to show that $(\tau^*,\sigma^*)$ is indeed a saddle point for the functional $\Psi$ of \eqref{Jey}. Take an arbitrary stopping time $\sigma\in[0,T-t]$, define $\tau_{\delta,\epsilon}:=\tau^*\wedge\tau^\epsilon\wedge\tau_\delta$ and replace $\tau \wedge \sigma$ in \eqref{uepsilondamgproperty} by $\sigma\wedge\tau_{\delta,\epsilon}$. It gives
\begin{align}\label{semarm01}
u^{\epsilon}(t,y) & \leq \tilde{\mathbb{E}}\bigg\{e^{-\bar{\mu}(\sigma\wedge\tau_{\delta,\epsilon})}u^{\epsilon}(t + \sigma\wedge\tau_{\delta,\epsilon} , Y^y(\sigma\wedge\tau_{\delta,\epsilon})) + \int_{0}^{\sigma\wedge\tau_{\delta,\epsilon}}e^{-\bar{\mu}s} R_c(Y^y(s))\,ds\bigg\}.
\end{align}
First we let $\epsilon$ go to zero and then take limits as $\delta\downarrow0$; using arguments as in \eqref{ueps01} 
we obtain
\begin{align}\label{semarm02}
v(t,y) \leq \tilde{\mathbb{E}}\bigg\{e^{-\bar{\mu}(\sigma \wedge \tau^*)}v(t + \sigma \wedge \tau^*, Y^y(\sigma \wedge \tau^*)) + \int_{0}^{\sigma \wedge \tau^*}e^{-\bar{\mu}s} R_c(Y^y(s))\,ds\bigg\}.
\end{align}
From \eqref{stoppingtimesv}, \eqref{v00} and the fact that $v \leq \frac{c_{+}}{f_C}$ we find $v(t,y)\le\Psi(t,y;\sigma,\tau^*)$.
Analogously, take an arbitrary stopping time $\tau\in[0,T-t]$, define $\sigma_{\delta,\epsilon}:=\sigma^*\wedge\sigma^\epsilon\wedge\tau_\delta$ and set $\tau \wedge \sigma:= \tau\wedge\sigma_{\delta,\epsilon}$ in \eqref{uepsilondamgproperty}. Same arguments as in \eqref{semarm01} and \eqref{semarm02} give
\begin{align}\label{semarm04}
v(t,y) \geq \tilde{\mathbb{E}}\bigg\{e^{-\bar{\mu}(\sigma^* \wedge \tau)}v(t + \sigma^* \wedge \tau, Y^y(\sigma^* \wedge \tau)) + \int_{0}^{\sigma^* \wedge \tau}e^{-\bar{\mu}s} R_c(Y^y(s))\,ds\bigg\},
\end{align}
and hence $v(t,y)\ge\Psi(t,y;\sigma^*,\tau)$
by \eqref{stoppingtimesv} and the bound $v \geq \frac{c_{-}}{f_C}$.

\subsection{Complements to the proof of Lemma \ref{cciuno}}\label{cciappendix}

In this section we will prove \eqref{cciloro05}. Full details are provided only for the integral involving $\frac{\partial K_1}{\partial y}$ as the cases of $\frac{\partial K_i}{\partial y}$, $i=2,3$ follow by straightforward generalization. Recall $p_C(y,s;z)$ as in \eqref{cciloro02-3}, $t_o$, $\delta$ and $\epsilon_o$ as in the proof of Lemma \ref{cciuno}, then
\begin{equation}
\label{deriv01}
\frac{\partial K_1}{\partial y}(y;s,\alpha_+(t_o+s),\alpha_-(t_o+s))
 = \int^{\alpha_-(t_o+s)}_{\alpha_+(t_o+s)}{R_c(z)\frac{\big[\ln\big(z/y\big)-\hat{\mu}_Cs
\big]}{\sigma^2_C\,s\,y}p_C(y,s;z)dz}, \nonumber
\end{equation}
for $(s,y)\in[\delta,T-t_o-\delta]\times[\epsilon_o,+\infty)$. We take modulus of the previous equation and use H\"older's inequality to obtain
\begin{align}\label{deriv02}
\hspace{-20pt}&\Big|\frac{\partial K_1}{\partial y}(y;s,\alpha_+(t_o\hspace{-2pt}+\hspace{-2pt}s),\alpha_-(t_o\hspace{-2pt}+\hspace{-2pt}s))\Big|\nonumber\\
&\leq\hspace{-2pt}\frac{1}{\sigma_C\,s\,y}\tilde{\mathbb{E}}\Big\{R^2_c\big(yC^{0}(s)\big)\mathds{1}_{\{\alpha_+(t_o+s)<yC^{0}(s)<\alpha_-(t_o+s)\}}\Big\}^{\frac{1}{2}}
\tilde{\mathbb{E}}\Big\{\big(\tilde{W}(s)\big)^2
\mathds{1}_{\{\alpha_+(t_o+s)<yC^{0}(s)<\alpha_-(t_o+s)\}}\Big\}^{\frac{1}{2}}\nonumber\\
&\le 1/(\sqrt{s}\,\sigma_C\,y)\tilde{\mathbb{E}}\Big\{R^2_c\big(yC^{0}(s)\big)\Big\}^{\frac{1}{2}},
\end{align}
by \eqref{nubarradefinizione} and \eqref{Girsanov}. Now from \eqref{deriv02} and calculations as in the proof of Lemma \ref{stimaconcavitaR} it follows that
\begin{align*}
\int^{\delta}_0&{e^{-\bar{\mu}s}\Big|\frac{\partial K_1}{\partial y}(y;s,\alpha_+(t_o\hspace{-2pt}+\hspace{-2pt}s),\alpha_-(t_o\hspace{-2pt}+\hspace{-2pt}s))\Big|ds}\hspace{-2pt}\le \hspace{-2pt}\gamma\sqrt{\delta}\Big(1+\frac{1}{y}\Big)\\
\int_{T-t_o-\delta}^{T-t_o}&{\hspace{-2pt}e^{-\bar{\mu}s}\Big|\frac{\partial K_1}{\partial y}(y;s,\alpha_+(t_o\hspace{-2pt}+\hspace{-2pt}s),\alpha_-(t_o\hspace{-2pt}+\hspace{-2pt}s))\Big|ds}\hspace{-2pt}\le\hspace{-2pt} \gamma\Big(\hspace{-2pt}\sqrt{T\hspace{-2pt}-\hspace{-2pt}t_o}-\sqrt{{T\hspace{-2pt}-\hspace{-2pt}t_o\hspace{-2pt}-\hspace{-2pt}\delta}}\hspace{-2pt}\Big)\Big(1+\frac{1}{y}\Big)
\end{align*}
for a suitable constant $\gamma>0$ and hence \eqref{cciloro05} holds.

\vspace{+8pt}
\ackn{The first author was supported by EPSRC grant EP/K00557X/1; Financial support by the German Research Foundation (DFG) via grant Ri--1128--4--1 is gratefully acknowledged by the second author. This paper was completed when the authors were visiting the Hausdorff Research Institute for Mathematics (HIM) at the University of Bonn in the framework of the Trimester Program ``Stochastic Dynamics in Economics and Finance''. We thank HIM for the hospitality. We wish also to thank the associate editor and two anonymous referees for their pertinent and useful comments and J.~Moriarty, G.~Peskir and F.~Riedel for many useful discussions}

\end{document}